\documentclass[11pt, twoside]{article}

\usepackage[english]{babel}

\usepackage{bbm}
\usepackage{amssymb}
\usepackage{amsfonts}
\usepackage[reqno,tbtags]{amsmath}
\usepackage{amsthm}
\usepackage{color}
\usepackage{mathrsfs}
\usepackage{txfonts}
\usepackage{enumerate}
\usepackage{anysize}
\usepackage{indentfirst}
\usepackage[T1]{fontenc}
\usepackage{latexsym}

\usepackage[colorlinks=true,
  linkcolor=blue,
  citecolor=red,
  urlcolor=magenta,
  ]{hyperref}

\allowdisplaybreaks

\newtheorem{theorem}{Theorem}[section]

\newtheorem{lemma}[theorem]{Lemma}

\newtheorem{proposition}[theorem]{Proposition}
\newtheorem{example}[theorem]{Example}
\theoremstyle{definition}
\newtheorem{remark}[theorem]{Remark}

\newtheorem{definition}[theorem]{Definition}

\newcounter{assum}

\renewcommand{\appendix}{\par
\setcounter{section}{0}%
\setcounter{subsection}{0}%
\setcounter{subsubsection}{0}%
\gdef\thesection{\@Alph\c@section}%
\gdef\thesubsection{\@Alph\c@section.\@arabic\c@subsection}%
\gdef\theHsection{\@Alph\c@section.}%
\gdef\theHsubsection{\@Alph\c@section.\@arabic\c@subsection}%
\csname appendixmore\endcsname
}

\numberwithin{equation}{section}

\begin{document}

\title{\bf\Large Sharp Quantitative Weighted Estimates for Lifted
Maximal Operators with Applications
\footnotetext{\hspace{-0.35cm} 2020
\emph{Mathematics Subject Classification}. Primary 42B25;
Secondary 42B30, 46E35, 46E30.
\endgraf \emph{Key words and phrases}.
lifted maximal operator, Muckenhoupt weight,
quantitative weighted estimate,
weak-type characterization,
Hardy--Sobolev space.
\endgraf
This project is partially supported by the National
Natural Science Foundation of China
(Grant Nos. 12431006, 12371093, and 125B2008), the Beijing Natural Science
Foundation (Grant No. 1262011), and
the Fundamental Research Funds
for the Central Universities (Grant No. 2253200028).}}
\author{Meng Geng, Dachun Yang, 
 Wen Yuan and Yirui Zhao
}
\maketitle

\vspace{-0.7cm}

\begin{center}
\begin{minipage}{13cm}
{\small {\bf Abstract}\quad
In this article, we establish sharp quantitative weighted estimates
for the lifted Hardy--Littlewood (resp Sobolev sharp) maximal operators
$\{\mathcal M_\theta\}_{\theta\in(0,\infty)}$ (resp $\{\mathcal M_\theta^\sharp\}_{\theta\in(0,\infty)}$).
Specifically, for $p\in(1,\infty)$,
$\gamma\in\mathbb R\setminus\{0\}$, $\upsilon\in A_p$,
and $f\in L^p_\upsilon$, we prove that
\begin{align*}
\sup_{\theta,\lambda\in(0,\infty)}\lambda^p
\underset{\mathcal M_\theta(f)(x,t)>\lambda t^\gamma}
{\int_{\mathbb R^n}\int_0^\infty}
t^{\gamma p-1}\,dt\,\upsilon(x)\,dx
\lesssim [\upsilon]_{A_p}^{\alpha_{p,\gamma}}
\|f\|_{L^p_\upsilon}^p,
\end{align*}
where $\alpha_{p,\gamma}:=\frac{p}{p-1}$ for
$\gamma\in[-n,0)$ and $\alpha_{p,\gamma}:=1$ otherwise
and where the implicit positive constant depends only on $n$, $p$,
and $\gamma$.
For $\mathcal M_\theta^\sharp$, the corresponding estimate
holds with $\|f\|_{L^p_\upsilon}^p$ replaced by
$\|\nabla f\|_{L^p_\upsilon}^p$
for any $f\in\dot W^{1,p}_\upsilon$,
and with $\alpha_{p,\gamma}:=\frac{p}{p-1}$
for $\gamma\in[-1,0)$ and $\alpha_{p,\gamma}:=1$ otherwise.
We show that, for an $A_1$ weight $\upsilon$,
the weight constant in the above estimates
becomes $[\upsilon]_{A_1}$
for the sharp ranges $p\in(1,\infty)$
and $\gamma\in\mathbb R\setminus\{0\}$,
or $p=1$ and $\gamma\in\mathbb R\setminus[-n,0]$ 
(resp. $\gamma\in\mathbb R\setminus[-1,0]$).
We also prove that $\alpha_{p,\gamma}$ is sharp,
and we establish exact limiting formulae
and characterize Muckenhoupt weights via
weak-type bounds.
As applications, we characterize the weighted Hardy(--Sobolev) space
$H^1_\upsilon$ with $\upsilon\in A_1$ in terms of the truncated
Riesz transforms with optimal
range $\gamma\in \mathbb R\setminus\{0\}$.
More applications include the sharp weighted
Caffarelli--Silvestre extension and the generalization
to ball Banach function spaces.
}
\end{minipage}

\end{center}

\vspace{0.2cm}

\tableofcontents

\vspace{0.2cm}

\section{Introduction}

Throughout this article, we work in $\mathbb{R}^n$ and, unless
otherwise specified, we always take $\mathbb{R}^n$ as the underlying space.

A well-known defect of the Gagliardo seminorm
is that it does not recover the first-order Sobolev seminorm at the
endpoint of the smoothness exponent. To remedy this defect,
Brezis et al.\ \cite{bsvy21}
replaced the strong $L^p$ norm in the Gagliardo seminorm
by the weak $L^p$ norm.
Subsequently, Brezis et al.\ \cite{bssvy24} showed that,
for any given $p\in[1,\infty)$ and $\gamma\in\mathbb R$, the equivalence
\begin{align}\label{eq:intro-bsvy}
\sup_{\lambda\in(0,\infty)}\lambda^p
\underset{|f(x)-f(y)|>\lambda|x-y|^{1+\frac\gamma p}}
{\int_{\mathbb R^n}\int_{\mathbb R^n}}
|x-y|^{\gamma-n}\,dy\,dx
\sim \|\nabla f\|_{L^p}^p
\end{align}
holds for any $f\in C^2$ with compactly supported $\nabla f$
if and only if either $p\in(1,\infty)$ and $\gamma\in\mathbb R\setminus\{0\}$, or $p=1$ and
$\gamma\in \mathbb R\setminus[-1,0]$, where the positive equivalence constants depend only on $n$, $p$, and
$\gamma$. Here, and thereafter, a condition written below
an iterated integral sign means to restrict the inner integration to that condition.
More precisely, we always use the \emph{abbreviation}
\begin{align*}
\underset{g(x,y)>h(x,y)}
{\int_X\int_Y}
\cdots\,d\nu(y)\,d\mu(x)
:=\int_X
\left[
\int_{\{y\in Y:\,g(x,y)>h(x,y)\}}
\cdots\,d\nu(y)
\right]d\mu(x).
\end{align*}
In particular, \eqref{eq:intro-bsvy} in the case $\gamma=n$
recovers \cite[Theorem 1.1]{bsvy21}.
These formulae connect the local first-order smoothness with
nonlocal difference quotients, and provide a representation of
first-order Sobolev seminorms valid even at the endpoint case $p=1$.
This discovery has stimulated various further developments and applications
in harmonic analysis (see, for example,
\cite{bssvy22,bvy-lorentz21,dlyyz22,gh23,hxz25,hyyz26,poliakovsky22}),
including weighted extensions (see, for example, \cite{dly,hlyyz25,scm,pssy26,zyy}),
and has close connections with the
calculus of variations (see, for example, \cite{agmp20,gp25,lahti25,nguyen25}).

Furthermore, weak-type representations have also been
established for operators beyond difference quotients (see, for example,
\cite{dlyyz,dm22,dssvy23,frank24,hytonen26,hk25,zlyyz24}).
In particular, Dom\'inguez and Milman
\cite{dm22} established a general framework associated with
one-parameter families of operators, which yields representations
of the norms of various differential operators beyond the gradient.
Their method relies on the strong $L^p$
boundedness of the relevant maximal operators. The failure of this
boundedness at the endpoint case $p=1$ prevents their framework from
providing a
satisfactory result at the endpoint case, a problem explicitly raised
in \cite[p.\,22]{dm22}.

Dai et al.\ \cite{dlyyz} addressed this endpoint problem by
introducing the lifted Hardy--Littlewood maximal operators.
The main idea is to restrict the radii of balls
appearing in the definition of the Hardy--Littlewood maximal operator
to some interval depending on the time variable
$t\in(0,\infty)$ and the given size scale $\theta\in(0,\infty)$,
instead of taking the supremum over all balls,
and then establish a weak-type estimate on the upper half-space
$\mathbb R^{n+1}_+:=\mathbb R^n\times(0,\infty)$.
Recall that, for any given $\theta\in(0,\infty)$ and for any $f\in L^1_{\rm loc}$
(the set of all locally integrable functions on $\mathbb R^n$),
the \emph{lifted Hardy--Littlewood maximal function}
$\mathcal M_\theta(f):\mathbb R^{n+1}_+\to[0,\infty]$ is defined by setting,
for any $(x,t)\in\mathbb R^{n+1}_+$,
\begin{align*}
\mathcal M_\theta(f)(x,t)
:=\sup_{\genfrac{}{}{0pt}{}{B\ni x}{
\operatorname{rad}(B)\in[\theta t,2\theta t)}}
\frac{1}{|B|}\int_B|f(y)|\,dy,
\end{align*}
where the supremum is taken over all balls $B$ containing $x$
with the radius, denoted by $\operatorname{rad}(B)$, satisfying
$\operatorname{rad}(B)\in[\theta t, 2\theta t)$.
Dai et al.\ \cite[Theorem~2.3]{dlyyz} proved that the equivalence
\begin{align}\label{eq:intro-unweighted-lifted-hl}
\sup_{\theta,\lambda\in(0,\infty)}\lambda^p
\underset{\mathcal M_\theta(f)(x,t)>\lambda t^\gamma}
{\int_{\mathbb R^n}\int_0^\infty}
t^{\gamma p-1}\,dt\,dx
\sim \|f\|_{L^p}^p
\end{align}
holds for any $f\in L^p$ if and only if
either $p\in(1,\infty)$ and $\gamma\in\mathbb{R}\setminus\{0\}$, or $p=1$ and $\gamma\in\mathbb{R}\setminus[-n,0]$, where the positive equivalence constants depend only on $n$, $p$, and $\gamma$.
The estimate \eqref{eq:intro-unweighted-lifted-hl} yields
endpoint weak-type representations for generators of
one-parameter operator families, including the Laplacian through
ball averages and the fractional Laplacian through the fractional
Poisson (Caffarelli--Silvestre) extension. Dai et al.\ \cite[Sections~3--5]{dlyyz} further applied these results
to semigroups and to characterizations of Hardy(--Sobolev)
spaces via truncated Riesz transforms.

Subsequently, Dai et al.\ \cite{dlyyz2} introduced the following
lifted Sobolev sharp maximal function.
Recall that, for any given $\theta\in(0,\infty)$ and for any $f\in L^1_{\rm loc}$, the \emph{lifted Sobolev sharp maximal function}
$\mathcal M_\theta^\sharp(f):
\mathbb R^{n+1}_+\to[0,\infty]$ is defined by setting, for any
$(x,t)\in\mathbb R^{n+1}_+$,
\begin{align*}
\mathcal M_\theta^\sharp(f)(x,t)
:=
\sup_{\genfrac{}{}{0pt}{}{B\ni x}{
\operatorname{rad}(B)\in[\theta t,2\theta t)}}
\frac{1}{\operatorname{rad}(B)}
\fint_B|f(y)-f_B|\,dy,
\end{align*}
here and thereafter, $$f_E:=\fint_E f(y)\,dy:=\frac{1}{|E|}\int_E f(y)\,dy$$
for any measurable set $E\subset\mathbb R^n$ with $|E|\in (0, \infty)$.
Dai et al.\ \cite[Theorem~1.1]{dlyyz2} showed that
\begin{align}\label{eq:intro-unweighted-lifted-sharp}
\sup_{\genfrac{}{}{0pt}{}{\theta\in(0,1]}{\lambda\in(0,\infty)}}\lambda^p
\underset{\mathcal M_\theta^\sharp(f)(x,t)>\lambda t^\gamma}
{\int_{\mathbb R^n}\int_0^\infty}
t^{\gamma p-1}\,dt\,dx
\sim \|\nabla f\|_{L^p}^p
\end{align}
holds for any $f\in\dot W^{1,p}$ if and only if either
$p\in(1,\infty)$ and $\gamma\in\mathbb R\setminus\{0\}$, or $p=1$ and
$\gamma\in\mathbb R\setminus[-1,0]$, where the
positive equivalence constants depend only on $n$, $p$, and
$\gamma$. In \cite{dlyyz2}, estimate \eqref{eq:intro-unweighted-lifted-sharp} yields formula
\eqref{eq:intro-bsvy} and
 an extension of Frank's characterization involving mean oscillation in
\cite[Theorem~1.1]{frank24} both with optimal parameter ranges,
including the endpoint case $p=1$. The corresponding formula also
applies to complete Poincar\'e spaces in \cite{dlyyz2}.

A natural \emph{question} is whether the equivalences in
\eqref{eq:intro-unweighted-lifted-hl} and
\eqref{eq:intro-unweighted-lifted-sharp} remain valid in the
weighted setting, especially at the endpoint case $p=1$, and how the corresponding
estimates depend on the weight constants.
In this article, we establish quantitative weighted analogues of
\eqref{eq:intro-unweighted-lifted-hl} and
\eqref{eq:intro-unweighted-lifted-sharp} with Muckenhoupt weights,
including the endpoint case $p=1$. We investigate the sharp powers of the Muckenhoupt
constants and characterize Muckenhoupt weights through
weak-type estimates for these operators.
In what follows, since we handle the weighted estimates for $\mathcal M_\theta$ and
$\mathcal M_\theta^\sharp$ simultaneously, to simplify the presentation
we introduce the following \emph{unified notation}.
Let $\tau\in\{\mathrm{av},\mathrm{osc}\}$ and write
\begingroup
\begin{align}\label{eq:unified-notation}
\begin{array}{c@{\qquad}|@{\qquad}c@{\qquad}|@{\qquad}c}
&\tau:=\mathrm{\color{red}av}
&\tau:=\mathrm{\color{blue}osc}\\
\hline
\mathcal M_\theta^\tau(f)
&\mathcal M_\theta(f)
&\mathcal M_\theta^\sharp(f)\\
D^\tau f
&f
&\nabla f\\
X^{\tau,p}_\upsilon
&L^p_\upsilon
&\dot W^{1,p}_\upsilon\\
c_n^\tau
&1
&K_{(n)}
\end{array}
\end{align}
\endgroup
where
$K_{(n)}
:=
\frac{1}{\sqrt{\pi}}
\frac{\Gamma(\frac{n+2}{2})}
{\Gamma(\frac{n+3}{2})}$ and the other notation
can be found in Definitions \ref{df:c-v} and \ref{df:hwss}. Let
\begin{align}\label{eq:unified-gamma-range}
\Gamma_p^\tau
:=
\begin{cases}
\mathbb R\setminus\{0\}
&\text{if }p\in(1,\infty),\\
\mathbb R\setminus[-n,0]
&\text{if }p=1\text{ and }\tau=\mathrm{av},\\
\mathbb R\setminus[-1,0]
&\text{if }p=1\text{ and }\tau=\mathrm{osc}.
\end{cases}
\end{align}
For any $p\in(1,\infty)$ and $\gamma\ne0$, let
\begin{align*}
\alpha_{p,\gamma}^{\mathrm{av}}
:=
\begin{cases}
\displaystyle\frac{p}{p-1}
&\text{if }\gamma\in[-n,0),\\
1
&\text{if }\gamma\in\mathbb R\setminus[-n,0]
\end{cases}
\ \mathrm{and}\
\alpha_{p,\gamma}^{\mathrm{osc}}
:=
\begin{cases}
\displaystyle\frac{p}{p-1}
&\text{if }\gamma\in[-1,0),\\
1
&\text{if }\gamma\in\mathbb R\setminus[-1,0].
\end{cases}
\end{align*}

Our main weighted estimates are as follows.

\begin{theorem}\label{th:intro-m}
Let $\tau\in\{\mathrm{av},\mathrm{osc}\}$,
$p\in[1,\infty)$, and $\gamma\in\Gamma_p^\tau$. Then the
following assertions hold.
\begin{enumerate}
\item[\textup{(i)}]
If $\upsilon\in A_1$, then, for any
$f\in X^{\tau,p}_\upsilon$,
\begin{align}\label{eq:intro-m-1}
\frac{(c_n^\tau)^p}{p|\gamma|}
\|D^\tau f\|_{L^p_\upsilon}^p
\leq
\sup_{\theta,\lambda\in(0,\infty)}
\lambda^p
\underset{
\mathcal M_\theta^\tau(f)(x,t)>
\lambda t^\gamma}
{\int_{\mathbb R^n}\int_0^\infty}
t^{\gamma p-1}\,dt\,\upsilon(x)\,dx
\lesssim
[\upsilon]_{A_1}
\|D^\tau f\|_{L^p_\upsilon}^p.
\end{align}

\item[\textup{(ii)}]
If $p\in(1,\infty)$ and $\upsilon\in A_p$, then, for any
$f\in X^{\tau,p}_\upsilon$,
\begin{align}\label{eq:intro-m-p}
\frac{(c_n^\tau)^p}{p|\gamma|}
\|D^\tau f\|_{L^p_\upsilon}^p
\leq
\sup_{\theta,\lambda\in(0,\infty)}
\lambda^p
\underset{
\mathcal M_\theta^\tau(f)(x,t)>
\lambda t^\gamma}
{\int_{\mathbb R^n}\int_0^\infty}
t^{\gamma p-1}\,dt\,\upsilon(x)\,dx
\lesssim
[\upsilon]_{A_p}^{\alpha_{p,\gamma}^{\tau}}
\|D^\tau f\|_{L^p_\upsilon}^p.
\end{align}
\end{enumerate}
The implicit positive constants depend only on $n$, $p$, and
$\gamma$. For any $\gamma\in\Gamma_p^\tau$,
the powers of the Muckenhoupt constants are \emph{sharp}.
\end{theorem}

\begin{remark}\label{re:intro-sm}
On Theorem \ref{th:intro-m}, we have the following observations.
\begin{enumerate}
\item[\textup{(i)}]
The proofs on the assertions that the powers of the Muckenhoupt constants in Theorem
\ref{th:intro-m} are sharp for any fixed $\gamma\in\Gamma_p^\tau$
are given in Propositions \ref{pro:sharp-m} and
\ref{prop:sharp-sm}.

\item[\textup{(ii)}]
For any $f\in L^1_{\rm loc}$ and
$(x,t)\in\mathbb R^{n+1}_+$, let
\begin{align*}
m_f(x,t)
:=
\fint_{B(x,t)}
|f(y)-f_{B(x,t)}|\,dy.
\end{align*}
By the estimate
$\frac{m_f(x,t)}{t}\leq\mathcal M_1^\sharp(f)(x,t)$
and \cite[Theorem~1.3(i)]{zlyyz24},
we obtain the lower estimate of Theorem \ref{th:intro-m} for $\tau=\mathrm{osc}$ directly.
In Theorem \ref{th:exact-m}, we establish the exact limiting formulae
for both operators $\mathcal{M}_\theta$ and
$\mathcal{M}^\sharp_\theta$, which are new even in the unweighted case.

\item[\textup{(iii)}]
When $\upsilon\equiv1$, Theorem \ref{th:intro-m} when $\tau=\mathrm{av}$ reduces to
the case $\mathcal X:=\mathbb R^n$ of \cite[Theorem~2.3]{dlyyz} and when $\tau=\mathrm{osc}$ reduces to the sufficiency part of
\cite[Theorem~1.1]{dlyyz2}.
The necessity parts of \cite[Theorem~1.1]{dlyyz} and
\cite[Theorem~1.1]{dlyyz2} also prove that the range
$\Gamma_p^\tau$ of $\gamma$ cannot be enlarged and hence is sharp.
\end{enumerate}
\end{remark}

We also characterize the Muckenhoupt class
$A_p$ through the upper weak-type estimates for the two operators (see Theorem \ref{th:AP-chara}).
For $\mathcal M_\theta$, the estimate involves
$\|f\|_{L^p_\upsilon}$. For $\mathcal M_\theta^\sharp$, it involves
$\|\nabla f\|_{L^p_\upsilon}$ when $n=1$ or $p=1$, and
$\|(-\Delta)^{\frac12}f\|_{L^p_\upsilon}$ when
$n\in\mathbb N\setminus\{1\}$ and $p\in(1,\infty)$,
following \cite{zlyyz24}.
In the latter case, we show that an upper estimate in terms of the gradient
norm alone cannot characterize the class $A_p$; see
Proposition \ref{prop:sm-gradient-not-characterization}.

Our applications to Hardy(--Sobolev) spaces build on the
unweighted estimates in terms of truncated Riesz transforms
established by Dai et al.\ \cite[Theorems~5.1 and~5.3]{dlyyz}.
We extend these results to $A_1$
weights, enlarge the admissible range from
$\gamma\in\mathbb R\setminus[-n,0]$ to
$\gamma\in\mathbb R\setminus\{0\}$
(see Theorems \ref{th:weighted-hardy-riesz} and
\ref{th:weighted-hardy-sobolev}),
and prove that the enlarged range is
sharp in this weighted setting (see Lemmas \ref{prop:weighted-hardy-gamma-zero} and
\ref{prop:weighted-hardy-sobolev-gamma-zero}). In particular, the unweighted case of
Theorem \ref{th:weighted-hardy-sobolev} answers the question raised in
\cite[Remark~5.4]{dlyyz}.
It is known that the real-variable theory of Hardy-type spaces 
plays an important role in areas of analysis
such as harmonic analysis and partial differential equations;
see \cite{b26,bdl18,bdl20,hs25,ins22,inns23,yns26} 
for some recent developments.

We also give two further applications. One application is
to establish weighted estimates for the Caffarelli--Silvestre
extension in the sharp parameter range (see Theorem \ref{th:weighted-CS}).
The other application is to extend the weighted estimates in
\eqref{eq:intro-m-1} to general ball Banach function spaces (see Theorem
\ref{th:m-x-sim}), which include several known function spaces
as special cases, such as Morrey spaces (see \cite{ho22morrey,ho24bilinear,ms19,ms20,noss21}),
mixed-norm Lebesgue spaces (see \cite{cgn17,cgn19}),
and Orlicz spaces (see \cite{ho23orliczmorrey,kns23,san19}),
and hence are of wide generality.

The upper estimates in Theorem \ref{th:intro-m} use the
weighted maximal estimates in \cite[Theorem~2.5 and (2.6)]{b93}
and, after dyadic discretization, the weighted average estimate
in \cite[Proposition~2.7]{scm} for $\tau=\mathrm{av}$ and the
weighted Cohen--Dahmen--Daubechies--DeVore inequality in
\cite[Theorem~1.1]{scm} for $\tau=\mathrm{osc}$.
Moreover, for $p\in(1,\infty)$ and $\gamma\in[-n,-1)$ with $n>1$,
we improve the exponent of $[\upsilon]_{A_p}$ in \cite[Theorem~1.1]{scm}
from $\frac{p}{p-1}$ to $1$ when $\tau=\mathrm{osc}$
[see Lemma \ref{thm-cddd}\textup{(ii)}], which is shown to be sharp [see Proposition \ref{prop:sharp-sm}\textup{(ii)}].

A key difficulty in the proof of Theorem \ref{th:intro-m}
is to determine the sharp powers of $[\upsilon]_{A_p}$
for $p\in(1,\infty)$ and fixed $\gamma\in[-n,0)$.
Inspired by some ideas from the construction of sparse sets,
we use $s$-medians and an inductive selection of function values
to construct suitable truncations. Using these truncations
and the Sobolev--Poincar\'e inequality,
we obtain Lemma \ref{thm-cddd}\textup{(ii)}.
To establish sharpness for $\mathcal M_\theta$
and $\gamma\in(-n,0)$, we take powers of the distance to compact
Ahlfors regular self-similar sets of dimension $n+\gamma$
as both weights and test functions, and then estimate the measures
of their tubular neighborhoods. To show the sharpness for
$\mathcal M_\theta^\sharp$ and $\gamma\in(-1,0)$,
we take the primitives of the aforementioned one-dimensional
test functions as new test functions and then extend the
one-dimensional examples to higher dimensions.
These constructions prove the sharpness for fixed $\gamma$
[see Propositions \ref{pro:sharp-m}\textup{(ii)}
and \ref{prop:sharp-sm}\textup{(iv)}].

To obtain the characterizations of the weighted Hardy space $H^1_\upsilon$,
we use its atomic characterization to estimate the weak
norm of the truncated Riesz transforms directly, rather than
relying on the pointwise estimate derived from Cotlar's
inequality in \cite[Proposition~5.6]{dlyyz}.
This approach allows us to extend the admissible range from
$\gamma\in\mathbb R\setminus[-n,0]$ to the optimal range
$\gamma\in\mathbb R\setminus\{0\}$.

The remainder of this article is organized as follows.
In Section \ref{sec:preliminaries}, we collect some preliminary
results on weights and dyadic grids, and establish some pointwise
relations for lifted maximal operators needed in the subsequent proofs.
In Section \ref{sec:weighted-limits}, we show the weighted
estimates and the exact limiting formulae for the lifted maximal
operators and establish some new characterizations
of $L^{\infty}$ and $\dot W^{1,\infty}$ when $\gamma=0$.
Section \ref{sec:m-sharp} is devoted to the sharpness of the
powers of the Muckenhoupt constants.
In Section \ref{sec:weight-characterizations}, we characterize
Muckenhoupt weights via the weighted weak-type estimates for both operators.
Finally, in Section \ref{sec:applications}, we establish the
characterization of $H^1_\upsilon$ in terms of truncated Riesz
transforms and the corresponding sharp representation
of $\|\nabla f\|_{H^1_\upsilon}$ for Schwartz functions $f$,
obtain a sharp weighted estimate for the
Caffarelli--Silvestre extension, and extend the
weighted weak-type estimates for both operators
to ball Banach function spaces.

We end this introduction with some notational conventions.
Let $\mathbb N:=\{1,2,\ldots\}$, $\mathbb Z_+:=\mathbb N\cup\{0\}$,
and $\mathbb Z$ denote the set of all integers.
The symbol $C$ denotes a positive constant which is independent
of the main parameters involved, but may vary from line to line.
For any nonnegative quantities $F$ and $G$, we write $F\lesssim G$
or $G\gtrsim F$ if $F\leq CG$ for some such constant $C$.
If both $F\lesssim G$ and $G\lesssim F$ hold, we write $F\sim G$.
For any Lebesgue measurable set $E\subset\mathbb R^n$, we denote
by $|E|$ its Lebesgue measure and by $\mathbf1_E$ its
characteristic function. For any weight $\upsilon$, let
$\upsilon(E):=\int_E\upsilon(x)\,dx.$
We use $\mathbf0$ to denote the origin of $\mathbb R^n$.
For any $x\in\mathbb R^n$ and $r\in(0,\infty)$, let
$$B(x,r):=\left\{y\in\mathbb R^n:|x-y|<r\right\}$$
and $\omega_n:=|B(\mathbf0,1)|$.
For any ball $B:=B(x,r)$ and $\lambda\in(0,\infty)$, let
$\lambda B:=B(x,\lambda r)$. We use the \emph{notation $t\to {0^+}$}
to denote that there exists a constant $c_0\in(0,1)$ such that
$t\in(0, c_0)$ and $t\to 0$. For any $s\in\mathbb R$,
we use the \emph{notation $\lceil s\rceil$} to denote the smallest
integer not less than $s$.
To avoid technical arguments, we consider only real-valued
measurable functions throughout this article. Our main
results also hold for complex-valued functions.
Finally, in all the proofs we retain the notation introduced
in the corresponding statements.

\section{Preliminaries}\label{sec:preliminaries}
We first collect the weighted maximal estimates for the Hardy--Littlewood maximal operator and then establish some
pointwise estimates for lifted maximal operators needed for the exact limiting formulae.

\begin{definition}
The \emph{Hardy--Littlewood maximal operator} $M$ is defined by setting,
for any $f\in L^1_{\rm loc}$ and $x\in\mathbb R^n$,
\begin{align*}
M(f)(x)
:=
\sup_{Q\ni x}\fint_Q |f(y)|\,dy,
\end{align*}
where the supremum is taken over all cubes $Q$
containing $x$.
\end{definition}
\begin{definition}\label{df:c-v}
A nonnegative locally integrable function $\upsilon$ on $\mathbb R^n$ is called
a \emph{weight} if $\upsilon(x)$ is positive for almost every $x\in\mathbb R^n$. Let $p\in[1,\infty)$ and $\upsilon$ be a weight.
\begin{enumerate}
\item[{\rm (i)}]
The function $\upsilon$ is called a \emph{Muckenhoupt $A_p$-weight} if,
when $p\in(1,\infty)$,
\begin{align}\label{eq:c-v-p}
[\upsilon]_{A_p}:=
\sup_Q
\left[\fint_Q\upsilon(x)\,dx\right]
\left\{\fint_Q[\upsilon(x)]^{-\frac{1}{p-1}}\,dx\right\}^{p-1}
<\infty
\end{align}
and, when $p=1$,
\begin{align}\label{eq:c-v-1}
[\upsilon]_{A_1}:=\sup_Q
\left[\fint_Q\upsilon(x)\,dx\right]
\left\|\upsilon^{-1}\right\|_{L^\infty (Q)}
<\infty,
\end{align}
 where the suprema are taken over all cubes. Equivalently, the
suprema in \eqref{eq:c-v-p} and \eqref{eq:c-v-1} can be taken over
all balls, with positive equivalence constants depending only on $n$ and $p$.

\item[{\rm (ii)}] The \emph{weighted Lebesgue space} $L^p_\upsilon$
is defined to be the set of all $f\in\mathscr{M}$ such that
\begin{align*}
\|f\|_{L^p_\upsilon}
:=
\left[\int_{\mathbb R^n}|f(x)|^p\upsilon(x)\,dx\right]^{\frac{1}{p}}
<\infty,
\end{align*}
where $\mathscr{M}$ denotes the set of all
measurable functions on $\mathbb R^n$.
\end{enumerate}
\end{definition}
\begin{definition}\label{df:hwss}
Let $p\in[1,\infty)$ and $\upsilon$ be a weight. The
\emph{homogeneous weighted Sobolev space} $\dot W^{1,p}_\upsilon$ is
defined to be the set of all $f\in L^1_{\rm loc}$ such that the weak
gradient $\nabla f$ exists and $|\nabla f|\in L^p_\upsilon$. Moreover,
the \emph{seminorm $\|\cdot\|_{\dot W^{1,p}_\upsilon}$} is defined by
setting, for any $f\in\dot W^{1,p}_\upsilon$,
$\|f\|_{\dot W^{1,p}_\upsilon}:=\|\nabla f\|_{L^p_\upsilon}.$
\end{definition}
We now recall the following weighted estimates for the
Hardy--Littlewood maximal operator. Assertion \textup{(i)} of Lemma
\ref{le:M-v} is precisely \cite[Theorem~2.5]{b93}. Assertion
\textup{(ii)} follows from applying the Marcinkiewicz interpolation
theorem to the weighted weak-type $(1,1)$ estimate in
\cite[(2.6)]{b93} and the trivial $L^\infty$ estimate for $M$.
\begin{lemma}\label{le:M-v}
Let $p\in(1,\infty)$. The following assertions hold.
\begin{enumerate}
\item[{\rm (i)}] For any $\upsilon\in A_p$,
$\|M\|_{L^p_\upsilon\to L^p_\upsilon}
\lesssim [\upsilon]_{A_p}^{\frac{1}{p-1}},$
where the implicit positive constant depends only on $n$ and $p$.
\item[{\rm (ii)}] For any $\upsilon\in A_1$,
$\|M\|_{L^p_\upsilon
\to L^p_\upsilon}
\lesssim
[\upsilon]_{A_1}^{\frac{1}{p}},$
where the implicit positive constant depends only on $n$ and $p$.
\end{enumerate}
\end{lemma}
We also need the following auxiliary results. We first recall the
standard definition of dyadic grids; see, for instance,
\cite[Section~2]{ln19}.
Throughout this article, cubes have finite positive
edge length and edges parallel to the coordinate axes.
For any cube $Q$, we denote its edge length by $\ell(Q)$.

\begin{definition}\label{df:dyadic-grid}
A countable collection $\mathcal D$ of cubes is called a
\emph{dyadic grid} if the following conditions hold:
\begin{enumerate}
\item[\textup{(i)}]
For any $Q\in\mathcal D$, one has
$\ell(Q)=2^k$ for some $k\in\mathbb Z$.
\item[\textup{(ii)}]
For any $P,Q\in\mathcal D$, one has
$P\cap Q\in\{\emptyset, P,Q\}.$
\item[\textup{(iii)}]
For any $k\in\mathbb Z$, one has
$
\mathbb R^n
=
\bigcup_{\genfrac{}{}{0pt}{1}{Q\in\mathcal D}{\ell(Q)=2^k}}Q.$
\end{enumerate}
\end{definition}

We next introduce the shifted dyadic grids used in the covering
argument below.

\begin{example}\label{ex-dyadic}
For any $\alpha\in\{0,1,2\}^n$, let
\begin{align*}
\mathcal D^\alpha
:=
\left\{
2^{-j}\left[m+[0,1)^n+(-1)^j\frac{\alpha}{3}\right]:
j\in\mathbb Z,\ m\in\mathbb Z^n
\right\}.
\end{align*}
Then $\mathcal D^\alpha$ is a dyadic grid in the sense of Definition
\ref{df:dyadic-grid}.
\end{example}

The following covering property is an immediate consequence of the
one-third trick; see \cite[p.\,479]{mtt02}.
\begin{lemma}\label{le:dyadic}
For any ball $B$, there exist
$\alpha\in\{0,1,2\}^n$
and $Q\in\mathcal D^\alpha$ such that
$
B\subset Q$
and $c_1\operatorname{rad}(B)<\ell(Q)<c_2\operatorname{rad}(B),$
where $c_1$ and $c_2$ are positive
constants depending only on $n$.
\end{lemma}
The following lemma is a standard property of $A_p$ weights,
which can be found in, for instance, \cite[Proposition~7.1.5]{gtm249}.
\begin{lemma}\label{lem:A1-doubling}
Let $p\in[1,\infty)$, $\upsilon\in A_p$, $Q$ be a
cube, and $S\subset Q$ be a measurable set with $|S|>0$. Then one has
$\upsilon(Q) \le [\upsilon]_{A_p}(\frac{|Q|}{|S|})^{p}\upsilon(S).$
\end{lemma}

In what follows, the \textit{space $W^{1,1}_{\rm loc}$} denotes the space
of all locally integrable functions on $\mathbb{R}^n$ whose weak gradients are locally integrable.
\begin{lemma}\label{lem:exact-sm-pointwise}
Let $\tau\in\{\mathrm{av},\mathrm{osc}\}$. Assume that $f\in L^1_{\rm loc}$
when $\tau=\mathrm{av}$, and that $f\in W^{1,1}_{\rm loc}$ when $\tau=\mathrm{osc}$.
Then, for almost every
$x\in\mathbb R^n$,
\begin{align}\label{eq:exact-sm-pointwise}
\lim_{t\to0^+}\mathcal M_1^\tau(f)(x,t)
=c_n^\tau|D^\tau f(x)|.
\end{align}
Moreover, there exists a positive constant $c$,
depending only on $n$, such that, for any
$(x,t)\in\mathbb R^{n+1}_+$,
\begin{align}\label{eq:exact-sm-control}
\mathcal M_1^\tau(f)(x,t)
\leq c M(|D^\tau f|)(x).
\end{align}
\end{lemma}

\begin{proof}
For $\tau=\mathrm{av}$, \eqref{eq:exact-sm-pointwise} follows from the Lebesgue
differentiation theorem and estimate \eqref{eq:exact-sm-control} follows immediately from the definitions.

We now consider $\tau=\mathrm{osc}$. Applying the Poincar\'e inequality on balls
\cite[Section~5.8.1, Theorem~2]{evans10} with $p=1$,
we find that \eqref{eq:exact-sm-control} holds. From the Lebesgue differentiation theorem and \cite[Theorems~3.3 and~3.4]{ghu24}
 with $p=1$, we deduce that, for almost every $x\in\mathbb R^n$,
\begin{align}\label{eq:exact-sm-L1-diff}
\lim_{r\to0^+}
\frac{1}{r}\fint_{B(x,r)}
|f(y)-f(x)-\nabla f(x)\cdot(y-x)|\,dy
=0.
\end{align}

Fix such $x$ and $t\in(0,\infty)$, and let
$\ell_x(y):=f(x)+\nabla f(x)\cdot(y-x)$ for $y\in\mathbb{R}^n$.
If $B:=B(z,r)\ni x$ with $r\in[t,2t)$, then
\begin{align*}
\frac{1}{r}\fint_B|f(y)-\ell_x(y)|\,dy
\lesssim
\frac{1}{t}\fint_{B(x,4t)}|f(y)-\ell_x(y)|\,dy.
\end{align*}
Moreover, using the triangle inequality, we conclude that
\begin{align*}
\left|
\frac{1}{r}\fint_B|f(y)-f_B|\,dy
-
\frac{1}{r}\fint_B|\ell_x(y)-(\ell_x)_B|\,dy
\right|
\leq
\frac{2}{r}\fint_B|f(y)-\ell_x(y)|\,dy.
\end{align*}
From $(\ell_x)_B=\ell_x(z)$, the change of variables
$u=\frac{y-z}{r}$, and the definition of $K_{(n)}$, we infer that
\begin{align*}
\frac{1}{r}\fint_B|\ell_x(y)-(\ell_x)_B|\,dy
&=\fint_{B(\mathbf0,1)}|\nabla f(x)\cdot u|\,du\\
&=|\nabla f(x)|\fint_{B(\mathbf0,1)}|u_1|\,du
=K_{(n)}|\nabla f(x)|.
\end{align*}
Therefore, for any $B=B(z,r)\ni x$ with $r\in[t,2t)$,
\begin{align}\label{eq:exact-sm-affine-oscillation-error}
\left|
\frac{1}{r}\fint_B|f(y)-f_B|\,dy
-
K_{(n)}|\nabla f(x)|
\right|
\lesssim
\frac{1}{t}\fint_{B(x,4t)}|f(y)-\ell_x(y)|\,dy,
\end{align}
where the implicit positive constant is independent of $B$ and $t$.

Taking the supremum over all such balls $B$ in
\eqref{eq:exact-sm-affine-oscillation-error}, we obtain
\begin{align*}
|
\mathcal M_1^\sharp(f)(x,t)-K_{(n)}|\nabla f(x)|
|
\lesssim
\frac{1}{t}\fint_{B(x,4t)}|f(y)-\ell_x(y)|\,dy.
\end{align*}
Letting $t\to0^+$ and applying \eqref{eq:exact-sm-L1-diff},
we find that \eqref{eq:exact-sm-pointwise} holds, which then completes
the proof of Lemma \ref{lem:exact-sm-pointwise}.
\end{proof}

\section{Weighted Estimates and Exact Limiting Formulae}
\label{sec:weighted-limits}
 In this section, using the preceding preliminaries, we prove the weighted estimates
in Theorem \ref{th:intro-m}. We first establish the upper bounds in Subsection \ref{sub:3-1}, then derive the exact limiting formulae and lower bounds in Subsection \ref{sub:3-2}, and
finally examine the excluded parameters in Subsection \ref{subsec:excluded-parameters}.

\subsection{Upper Estimates}\label{sub:3-1}

We begin by establishing the upper bounds using weighted maximal
inequalities and dyadic estimates for averages and mean oscillations.
A scaling identity first reduces the argument to the case $\theta=1$.

By the definitions, we conclude that, for any
$\tau\in\{\mathrm{av},\mathrm{osc}\}$, $\theta\in(0,\infty)$, and $f\in L^1_{\rm loc}$,
\begin{align*}
\mathcal M_\theta^\tau(f)(x,t)
=\mathcal M_1^\tau(f)(x,\theta t), \quad\forall\,(x,t)\in \mathbb R^{n+1}_+.
\end{align*}
For any $\theta,\lambda\in(0,\infty)$, from the substitutions $s=\theta t$ and $\mu:=\lambda\theta^{-\gamma}$, it follows that
\begin{align}\label{eq:theta-change}
\lambda^p
\underset{\mathcal M_\theta^\tau(f)(x,t)>\lambda t^\gamma}
{\int_{\mathbb R^n}\int_0^\infty}
t^{\gamma p-1}\,dt\,\upsilon(x)\,dx
=\mu^p
\underset{\mathcal M_1^\tau(f)(x,s)>\mu s^\gamma}
{\int_{\mathbb R^n}\int_0^\infty}
s^{\gamma p-1}\,ds\,\upsilon(x)\,dx.
\end{align}
Consequently, taking the supremum over $\theta$ and $\lambda$
gives the same quantity as fixing $\theta=1$ and taking the
supremum over $\lambda$.

We next recall two weighted dyadic estimates for averages and
mean oscillations, respectively. The first one is exactly
\cite[Proposition~2.7]{scm}.
\begin{lemma}\label{es-dyadic1}
If $p\in[1,\infty)$ and
$\beta\in(-\infty,\frac{1}{p}-1)\cup(\frac{1}{p},\infty)$, then, for any dyadic grid
$\mathcal D$, any $\upsilon\in A_p$, and any $f\in L^\infty$ with
compact support,
\begin{align*}
\sup_{\lambda\in(0,\infty)}\lambda^p
\sum_{Q\in\mathcal E_{\lambda,\beta-\frac{1}{p}}[f]}
|Q|^{\beta p-1}\upsilon(Q)
\lesssim[\upsilon]_{A_p}\|f\|_{L^p_\upsilon}^p,
\end{align*}
where the implicit positive constant depends only on $n$, $p$,
and $\beta$, and where, for any dyadic grid $\mathcal D$, any
$\lambda\in(0,\infty)$, and any $b\in\mathbb R$,
\begin{align*}
\mathcal E_{\lambda,b}[f]
:=
\left\{
Q\in\mathcal D:
\fint_Q|f(x)|\,dx>\lambda|Q|^b
\right\}.
\end{align*}
\end{lemma}

The following lemma improves the exponent of $[\upsilon]_{A_p}$
in \cite[Theorem~1.1]{scm} to $1$ when
$p\in(1,\infty)$ and
$\beta\in[\frac1p-1,\frac1p-\frac1n)$.
\begin{lemma}\label{thm-cddd}
Let $p\in[1,\infty)$. Assume that
$\beta\in(-\infty,1-\frac{1}{n})\cup(1,\infty)$ when $p=1$,
and that $\beta\in\mathbb R\setminus\{\frac{1}{p}\}$ when
$p\in(1,\infty)$.
\begin{enumerate}
  \item[\textup{(i)}] If $\upsilon\in A_1$,
  then,
  for any dyadic grid $\mathcal{D}$ and
  any $f\in \dot{W}^{1,p}_\upsilon$,
  \begin{align*}
  \sup_{\lambda\in(0,\infty)}\lambda^p
  \sum_{Q\in\mathcal{D}_{\lambda,\beta+1-\frac{1}{p}}[f]}
  |Q|^{\beta p-1}\upsilon(Q)\lesssim
  [\upsilon]_{A_1}\|\nabla f\|_{L^p_{\upsilon}}^p,
\end{align*}
  where the implicit positive constant depends only on $n$, $p$, and $\beta$.
  \item[\textup{(ii)}] If $\upsilon\in A_p$,
  then,
  for any dyadic grid $\mathcal{D}$ and
  any $f\in \dot{W}^{1,p}_\upsilon$,
  \begin{align}\label{eq:thm-cddd-2}
  \sup_{\lambda\in(0,\infty)}\lambda^p
  \sum_{Q\in\mathcal{D}_{\lambda,\beta+1-\frac{1}{p}}[f]}
  |Q|^{\beta p-1}\upsilon(Q)\lesssim
  [\upsilon]_{A_p}^{\widetilde\alpha_{p,\beta}}
  \|\nabla f\|_{L^p_{\upsilon}}^p,
  \end{align}
  where the implicit positive constant depends only on $n$, $p$, and $\beta$, and where
  \begin{align*}
  \widetilde\alpha_{p,\beta}:=
  \begin{cases}
  \displaystyle\frac{p}{p-1}
  &\text{if } p\in(1,\infty)\ \text{and }
  \beta\in\left[\frac{1}{p}-\frac{1}{n},\frac{1}{p}\right),\\
  1
  &\text{otherwise}.
  \end{cases}
  \end{align*}
\end{enumerate}
Here, for any dyadic grid $\mathcal D$, any $\lambda\in(0,\infty)$,
any $b\in\mathbb R$, and any $f\in L^1_{\mathrm{loc}}$,
\begin{align}\label{eq:3.1x}
\mathcal D_{\lambda,b}[f]
:=\left\{Q\in\mathcal D:\frac{1}{|Q|^{1+\frac{1}{n}}}
\int_Q\int_Q|f(x)-f(y)|\,dx\,dy>\lambda|Q|^b\right\}.
\end{align}
\end{lemma}
\begin{remark}
Let $n\in\mathbb N\setminus\{1\}$ and $p\in(1,\infty)$.
For $\beta\in[\frac{1}{p}-1,\frac{1}{p}-\frac{1}{n})$,
Lemma \ref{thm-cddd}\textup{(ii)} improves the power of
$[\upsilon]_{A_p}$ in \cite[Theorem~1.1\textup{(ii)}]{scm}
from $\frac{p}{p-1}$ to $1$.
\end{remark}

\begin{proof}[Proof of Lemma \ref{thm-cddd}]
All assertions except \textup{(ii)} when $p\in(1,\infty)$ and
$\beta\in[\frac1p-1,\frac1p-\frac1n)$ are precisely the
corresponding estimates in \cite[Theorem~1.1]{scm}.
It remains to consider $n\in\mathbb N\setminus\{1\}$,
$p\in(1,\infty)$, and
$\beta\in[\frac1p-1,\frac1p-\frac1n)$. Fix a dyadic grid $\mathcal D$, $\upsilon\in A_p$,
and $f\in\dot W^{1,p}_\upsilon$.
We first consider real-valued $f$ and $\lambda=1$.
Fix a finite collection
$\mathcal S\subset\mathcal D_{1,\beta+1-\frac{1}{p}}[f]$.
For any $Q\in\mathcal S$ and $s\in(0,1)$,
the \emph{$s$-median} $F_Q(s)$ of $f$ on $Q$
is defined by setting
\begin{align*}
F_Q(s):=
\inf\{a\in\mathbb R:
|\{x\in Q:f(x)\leq a\}|\geq s|Q|\}.
\end{align*}
For any $m\in\mathbb Z_+$, let $\delta_m:=2^{-m-2}$.
Using the triangle inequality and
\cite[Chapter~2, Proposition~1.8]{bennett-sharpley88} with $p=1$,
we find that
\begin{align}\label{eq:osc-qint}
\frac12\fint_Q|f(x)-f_Q|\,dx
&\leq\fint_Q\left|f(x)-F_Q\left(\frac12\right)\right|\,dx=\int_0^1\left|F_Q(s)-F_Q\left(\frac12\right)\right|\,ds\notag\\
&=\int_0^{\frac12}
\left[F_Q(1-s)-F_Q(s)\right]\,ds
\leq\sum_{m=0}^\infty
\delta_m\left[F_Q(1-\delta_m)-F_Q(\delta_m)\right].
\end{align}
Indeed, the last inequality follows from splitting $(0,\frac12)$
into the intervals $(\delta_m,2\delta_m)$ and using the
monotonicity of $F_Q$.
Since $\fint_Q|f(x)-f_Q|\,dx >\frac12\ell(Q)^{n\beta+1-\frac{n}{p}}$, from \eqref{eq:osc-qint}, we deduce that there exists
$m\in\mathbb Z_+$ such that
\begin{align}\label{eq:osc-qgap}
F_Q(1-\delta_m)-F_Q(\delta_m)
>\frac{2^{m-1}}{(m+1)^2}\ell(Q)^{n\beta+1-\frac{n}{p}}.
\end{align}
Otherwise, by \eqref{eq:osc-qint}, we conclude that
\begin{align*}
\frac14\ell(Q)^{n\beta+1-\frac{n}{p}}
\leq\frac{\ell(Q)^{n\beta+1-\frac{n}{p}}}8
\sum_{m=0}^\infty\frac1{(m+1)^2}
=\frac{\pi^2}{48}\ell(Q)^{n\beta+1-\frac{n}{p}}
<\frac14\ell(Q)^{n\beta+1-\frac{n}{p}},
\end{align*}
which is a contradiction.
For any $Q\in\mathcal S$, let $m:=m(Q)\in\mathbb Z_+$
be the smallest index for which \eqref{eq:osc-qgap} holds,
and let
\begin{align}\label{eq:osc-qintv}
I_Q:=\left(F_Q(\delta_m),F_Q(1-\delta_m)\right).
\end{align}
From the definition of $F_Q$, we infer that
\begin{align}\label{eq:osc-lower}
|\{x\in Q:f(x)\leq F_Q(\delta_m)\}|&\geq\delta_m|Q|
\end{align}
and
\begin{align}
|\{x\in Q:f(x)\geq F_Q(1-\delta_m)\}|&\geq\delta_m|Q|.
\label{eq:osc-upper}
\end{align}

Let
$N:=\lceil\frac{1}{n(\frac{1}{p}-\beta)-1}\rceil.$ For any $P\in\mathcal S$, write $\ell(P)=2^{j(P)}$ with
$j(P)\in\mathbb Z$.
For any $m\in\mathbb Z_+$ and $r\in\{0,\ldots,N-1\}$, let
\begin{align*}
\mathcal S_{m,r}
:=\{P\in\mathcal S:j(P)\equiv r\;(\mathrm{mod}\,N),\ m(P)=m\}.
\end{align*}
Fix $m$ and $r$ such that $\mathcal S_{m,r}\ne\emptyset$.
Fix $Q\in\mathcal S_{m,r}$.
By the definition of $\mathcal S_{m,r}$, for any
$P\in\mathcal S_{m,r}$ with $Q\subsetneqq P$,
there exists $k\in\mathbb N$ such that
$\ell(P)=2^{kN}\ell(Q)$. Note that, for any fixed $k$,
there exists at most one such cube $P$.
Using $N[n(\frac{1}{p}-\beta)-1]\geq1$, we obtain
\begin{align}
\sum_{\genfrac{}{}{0pt}{}{P\in\mathcal S_{m,r}}{Q\subsetneqq P}}
\ell(P)^{n\beta+1-\frac{n}{p}}
\leq
\sum_{k=1}^\infty2^{-k}\ell(Q)^{n\beta+1-\frac{n}{p}}
=\ell(Q)^{n\beta+1-\frac{n}{p}}.
\label{eq:osc-anc}
\end{align}
Enumerate $\mathcal S_{m,r}=\{Q_1,\ldots,Q_K\}$ such that
$\ell(Q_1)\geq\cdots\geq\ell(Q_K)$.
We construct sets $E_{Q_1},\ldots,E_{Q_K}$ such that,
for any $j\in\{1,\ldots,K\}$, $E_{Q_j}$ is a finite union
of compact intervals and satisfies, for any
$j\in\{1,\ldots,K\}$,
\begin{align}
E_{Q_j}
\subset
I_{Q_j}\setminus
\bigcup_{\genfrac{}{}{0pt}{}{1\leq i<j}{Q_j\subsetneqq Q_i}}E_{Q_i} \, \text{and} \,
|E_{Q_j}|
=\frac{2^{m-2}}{(m+1)^2}\ell(Q_j)^{n\beta+1-\frac{n}{p}}.
\label{eq:osc-values}
\end{align}
Note that the union
$\bigcup_{\genfrac{}{}{0pt}{}{1\leq i<j}{Q_j\subsetneqq Q_i}}E_{Q_i}=\emptyset$
 if there is no $i<j$ such that
$Q_j\subsetneqq Q_i$.
In particular, \eqref{eq:osc-values}  ensure that
$E_Q\cap E_P=\emptyset$ whenever
$P\in\mathcal S_{m,r}$ and $Q\subsetneqq P$.
We show the existence of these sets by induction.
Since \eqref{eq:osc-qgap} holds, it follows that
$|I_{Q_1}|
>\frac{2^{m-1}}{(m+1)^2}\ell(Q_1)^{n\beta+1-\frac{n}{p}}.$
Thus, we can choose a compact interval
$E_{Q_1}\subset I_{Q_1}$ such that
$
|E_{Q_1}|
=\frac{2^{m-2}}{(m+1)^2}\ell(Q_1)^{n\beta+1-\frac{n}{p}}.$
This proves \eqref{eq:osc-values} for $j=1$.

Fix $j\in\{2,\ldots,K\}$ and suppose that
$E_{Q_1},\ldots,E_{Q_{j-1}}$ have been chosen
as finite unions of compact intervals satisfying
\eqref{eq:osc-values}.
Let
$U_{Q_j}:=
I_{Q_j}\setminus
\bigcup_{\genfrac{}{}{0pt}{}{1\leq i<j}{Q_j\subsetneqq Q_i}}E_{Q_i}.$
By the induction hypothesis, \eqref{eq:osc-anc}, and \eqref{eq:osc-qgap}, we find that
\begin{align*}
|U_{Q_j}|
&\geq |I_{Q_j}|-
\sum_{\genfrac{}{}{0pt}{}{1\leq i<j}{Q_j\subsetneqq Q_i}}|E_{Q_i}|=|I_{Q_j}|-\frac{2^{m-2}}{(m+1)^2}
\sum_{\genfrac{}{}{0pt}{}{1\leq i<j}{Q_j\subsetneqq Q_i}}
\ell(Q_i)^{n\beta+1-\frac{n}{p}}\\
&\geq
|I_{Q_j}|-\frac{2^{m-2}}{(m+1)^2}\ell(Q_j)^{n\beta+1-\frac{n}{p}}>
\frac{2^{m-2}}{(m+1)^2}\ell(Q_j)^{n\beta+1-\frac{n}{p}}.
\end{align*}
Since $U_{Q_j}$ is a finite union of bounded open intervals,
we can choose a finite union $E_{Q_j}\subset U_{Q_j}$
of compact intervals such that
$|E_{Q_j}|
=\frac{2^{m-2}}{(m+1)^2}\ell(Q_j)^{n\beta+1-\frac{n}{p}}.$
Consequently, \eqref{eq:osc-values} holds for any
$j\in\{1,\ldots,K\}$.

For any $Q\in\mathcal S_{m,r}$, we define
$T_Q(s):=\int_{-\infty}^s\mathbf1_{E_Q}(u)\,du $ for any $s\in\mathbb R.$
Since $T_Q$ is Lipschitz, from  the chain rule for Sobolev functions
[see, for instance, \cite[(1.1)--(1.2)]{leoni-morini07}], we deduce that
\begin{align}\label{eq:osc-grad}
\nabla[T_Q(f(\cdot))](x)=\mathbf1_{E_Q}(f(x))\nabla f(x)
\end{align}
for almost every $x\in Q$.
Moreover, by \eqref{eq:osc-qintv} and
\eqref{eq:osc-values}, $T_Q(f(\cdot))$ equals $0$ on
$\{x\in Q:f(x)\leq F_Q(\delta_m)\}$ and equals $|E_Q|$ on
$\{x\in Q:f(x)\geq F_Q(1-\delta_m)\}$.
If $(T_Q(f(\cdot)))_Q\geq\frac12|E_Q|$, then
\begin{align*}
|T_Q(f(x))-(T_Q(f(\cdot)))_Q|
=(T_Q(f(\cdot)))_Q
\geq\frac12|E_Q|
\end{align*}
on $\{x\in Q:f(x)\leq F_Q(\delta_m)\}$.
Otherwise,
\begin{align*}
|T_Q(f(x))-(T_Q(f(\cdot)))_Q|
=|E_Q|-(T_Q(f(\cdot)))_Q
>\frac12|E_Q|
\end{align*}
on $\{x\in Q:f(x)\geq F_Q(1-\delta_m)\}$.
Using these, \eqref{eq:osc-lower}, and
\eqref{eq:osc-upper}, we conclude that
\begin{align}\label{eq:osc-measure}
\left|
\left\{x\in Q:
|T_Q(f(x))-(T_Q(f(\cdot)))_Q|
\geq\frac12|E_Q|
\right\}
\right|
\geq\delta_m|Q|.
\end{align}
From \eqref{eq:osc-measure}, \eqref{eq:osc-values},
and \eqref{eq:osc-grad}, together with the
Sobolev--Poincar\'e inequality on $Q$, we infer that
\begin{align*}
\frac{2^{\frac mn}}{(m+1)^2}\ell(Q)^{n(\beta+1-\frac{1}{p})}
\lesssim
\|T_Q(f(\cdot))-(T_Q(f(\cdot)))_Q\|_{L^{\frac n{n-1}}(Q)}
\lesssim
\int_{\{x\in Q:f(x)\in E_Q\}}|\nabla f(x)|\,dx.
\end{align*}
By this, H\"older's inequality, and the $A_p$ condition, we obtain
\begin{align}\label{eq:osc-local}
|Q|^{\beta p-1}\upsilon(Q)
&\lesssim(m+1)^{2p}2^{-\frac{mp}{n}}
\frac{\upsilon(Q)}{|Q|^p}
\left[\int_{\{x\in Q:f(x)\in E_Q\}}|\nabla f(x)|\,dx\right]^p\notag\\
&\leq (m+1)^{2p}2^{-\frac{mp}{n}}
\frac{\upsilon(Q)}{|Q|^p}
\left[\int_{\{x\in Q:f(x)\in E_Q\}}|\nabla f(x)|^p\upsilon(x)\,dx\right]
\left[\int_Q\upsilon(x)^{-\frac1{p-1}}\,dx\right]^{p-1}\notag\\
&\lesssim
[\upsilon]_{A_p}(m+1)^{2p}2^{-\frac{mp}n}
\int_{\{x\in Q:f(x)\in E_Q\}}|\nabla f(x)|^p\upsilon(x)\,dx.
\end{align}
From
\eqref{eq:osc-values}, it follows that
\begin{align}\label{eq:osc-disjoint}
\sum_{Q\in\mathcal S_{m,r}}
\mathbf1_Q(x)\mathbf1_{E_Q}(f(x))\leq1
\end{align}
for almost every $x\in\mathbb R^n$.
Summing \eqref{eq:osc-local} and applying
\eqref{eq:osc-disjoint}, we conclude that
\begin{align}\label{eq:Qsmr-lQ-vQ-fv}
\sum_{Q\in\mathcal S_{m,r}}
|Q|^{\beta p-1}\upsilon(Q)
\lesssim
[\upsilon]_{A_p}(m+1)^{2p}2^{-\frac{mp}n}
\int_{\mathbb R^n}|\nabla f(x)|^p\upsilon(x)\,dx,
\end{align}
where the implicit positive constant is independent of $m$.
Using $\sum_{m=0}^\infty(m+1)^{2p}2^{-\frac{mp}n}<\infty$
and \eqref{eq:Qsmr-lQ-vQ-fv}, we find that
\begin{align}\label{eq:osc-finite}
\sum_{Q\in\mathcal S}
|Q|^{\beta p-1}\upsilon(Q)
\lesssim\sum_{m=0}^{\infty}\sum_{r=0}^{N-1}\sum_{Q\in\mathcal S_{m,r}}|Q|^{\beta p-1}\upsilon(Q)\lesssim
[\upsilon]_{A_p}\int_{\mathbb R^n}|\nabla f(x)|^p\upsilon(x)\,dx,
\end{align}
where the implicit positive constant depends only on $n$, $p$, and $\beta$.
Taking the supremum over finite subcollections $\mathcal S$,
we find that \eqref{eq:osc-finite} with
$\mathcal S$ replaced by $\mathcal D_{1,\beta+1-\frac{1}{p}}[f]$ holds.

Applying \eqref{eq:osc-finite} to $\frac{f}{\lambda}$
for any $\lambda\in(0,\infty)$, we conclude that \eqref{eq:thm-cddd-2} holds.
This completes the proof of Lemma \ref{thm-cddd}.
\end{proof}

\begin{proof}[Proof of the Upper Estimates in Theorem \ref{th:intro-m}]
Since \eqref{eq:theta-change} holds, we deduce that it suffices to consider
$\theta=1$. We first consider $\tau=\mathrm{av}$.
When $p\in(1,\infty)$, using \eqref{eq:exact-sm-control}
and Tonelli's theorem, we conclude that
\begin{align}\label{eq:m-by-M}
\lambda^p
\underset{\mathcal M_1(f)(x,t)>\lambda t^\gamma}
{\int_{\mathbb R^n}\int_0^\infty}
t^{\gamma p-1}\,dt\,\upsilon(x)\,dx
\lesssim\|M(f)\|_{L^p_\upsilon}^p.
\end{align}
Applying Lemma \ref{le:M-v} to \eqref{eq:m-by-M}, we find that
\begin{align*}
\sup_{\lambda\in(0,\infty)}\lambda^p
\underset{\mathcal M_1(f)(x,t)>\lambda t^\gamma}
{\int_{\mathbb R^n}\int_0^\infty}
t^{\gamma p-1}\,dt\,\upsilon(x)\,dx
\lesssim
\begin{cases}
[\upsilon]_{A_1}\|f\|_{L^p_\upsilon}^p
&\text{if }\upsilon\in A_1,\\
[\upsilon]_{A_p}^{\frac{p}{p-1}}\|f\|_{L^p_\upsilon}^p
&\text{if }\upsilon\in A_p.
\end{cases}
\end{align*}
This shows the upper estimate in \textup{(i)} when
$p\in(1,\infty)$ and that in \textup{(ii)} when
$\gamma\in[-n,0)$.

It remains to consider $\gamma\in\mathbb R\setminus[-n,0]$.
We first assume that $f\in L^\infty$ has compact support.
Fix $\lambda\in(0,\infty)$, and let
$\beta:=\frac{\gamma}{n}+\frac{1}{p}$.
If $\mathcal M_1(f)(x,t)>\lambda t^\gamma$, then
there exists a ball $B\ni x$ with
$\operatorname{rad}(B)\in[t,2t)$ such that
$\fint_B|f(y)|\,dy>\lambda t^\gamma$.
Applying Lemma \ref{le:dyadic}, we find that there exists a cube
$Q\in\mathcal D^\alpha$, for some
$\alpha\in\{0,1,2\}^n$, such that
$B\subset Q$ and $\ell(Q)\sim t$. Thus,
\begin{align*}
\fint_Q|f(y)|\,dy
\gtrsim\lambda\ell(Q)^\gamma
=\lambda|Q|^{\beta-\frac{1}{p}}.
\end{align*}
Consequently, there exist positive constants $a$ and $A$ with $a\in(0,A)$, depending only on $n$, and a positive constant $c$, depending only on $n$ and $\gamma$, such that
\begin{align*}
\{(x,t)\in\mathbb R^{n+1}_+:
\mathcal M_1(f)(x,t)>\lambda t^\gamma\}
\subset
\bigcup_{\alpha\in\{0,1,2\}^n}
\bigcup_{Q\in\mathcal E_{c\lambda,\beta-\frac{1}{p}}[f]}
(Q\times[a\ell(Q),A\ell(Q)]),
\end{align*}
where, for any $\alpha\in\{0,1,2\}^n$, the collection
$\mathcal E_{c\lambda,\beta-\frac{1}{p}}[f]$ is taken in
$\mathcal D^\alpha$.
From integration over these boxes and $\beta p-1=\frac{\gamma p}{n}$, we infer that
\begin{align}
\lambda^p
\underset{\mathcal M_1(f)(x,t)>\lambda t^\gamma}
{\int_{\mathbb R^n}\int_0^\infty}
t^{\gamma p-1}\,dt\,\upsilon(x)\,dx\lesssim\sum_{\alpha\in\{0,1,2\}^n}(c\lambda)^p
\sum_{Q\in\mathcal E_{c\lambda,\beta-\frac{1}{p}}[f]}
|Q|^{\beta p-1}\upsilon(Q),
\label{eq:unified-dyadic-reduction}
\end{align}
where the implicit positive constant depends only on $n$, $p$,
and $\gamma$.
Using Lemma \ref{es-dyadic1} with
$\beta\in(-\infty,\frac{1}{p}-1)\cup(\frac{1}{p},\infty)$
and \eqref{eq:unified-dyadic-reduction}, we conclude that
taking the supremum over $\lambda\in(0,\infty)$ gives the required upper
estimate with $[\upsilon]_{A_p}$ to the power $1$. Applying the upper bounds in \eqref{eq:intro-m-1} and \eqref{eq:intro-m-p},
already proved for bounded functions with compact support, to
$\min\{|f|,N\}{\mathbf 1}_{B(\mathbf0,N)}$ and then using the monotone
convergence theorem as $N\to\infty$, we obtain the same upper bounds for any
$f\in L^p_\upsilon$. This completes the proof when $\tau=\mathrm{av}$.

We now consider $\tau=\mathrm{osc}$. When $p\in(1,\infty)$,
by \eqref{eq:exact-sm-control} and an argument similar to that
leading to \eqref{eq:m-by-M}, we find that \eqref{eq:m-by-M} also holds
with $\mathcal M_1(f)$ and $M(f)$ replaced by
$\mathcal M_1^\sharp(f)$ and $M(|\nabla f|)$, respectively.
From Lemma \ref{le:M-v}, it follows that, for $p\in(1,\infty)$,
the upper bound in Theorem \ref{th:intro-m}\textup{(i)} holds for
$\upsilon\in A_1$ and $\gamma\in\mathbb R\setminus\{0\}$, while
the upper bound in Theorem \ref{th:intro-m}\textup{(ii)} holds for
$\upsilon\in A_p$ and $\gamma\in[-1,0)$.

For $\gamma\in\mathbb R\setminus[-1,0]$, we now show
the upper bound in Theorem \ref{th:intro-m}\textup{(i)}
when $p=1$ and the upper bound in
Theorem \ref{th:intro-m}\textup{(ii)} when $p\in(1,\infty)$.
Let $\beta:=\frac{\gamma}{n}+\frac{1}{p}$.
Repeating the argument leading to
\eqref{eq:unified-dyadic-reduction}, we find that
\begin{align}\label{eq:ocs-dyadic}
\lambda^p
\underset{\mathcal M_1^\sharp(f)(x,t)>\lambda t^\gamma}
{\int_{\mathbb R^n}\int_0^\infty}
t^{\gamma p-1}\,dt\,\upsilon(x)\,dx\lesssim\sum_{\alpha\in\{0,1,2\}^n}(c\lambda)^p
\sum_{Q\in\mathcal D_{c\lambda,\beta+1-\frac1p}[f]}
|Q|^{\beta p-1}\upsilon(Q).
\end{align}
By Lemma \ref{thm-cddd}\textup{(i)}, we conclude that
the upper bound in Theorem \ref{th:intro-m}\textup{(i)} holds
when $p=1$ and $\upsilon\in A_1$.
For $p\in(1,\infty)$ and $\upsilon\in A_p$, since
$\gamma\in\mathbb R\setminus[-1,0]$, we deduce that
$\widetilde\alpha_{p,\beta}=1$.
Applying Lemma \ref{thm-cddd}\textup{(ii)} to  \eqref{eq:ocs-dyadic}, we obtain
\begin{align}\label{eq:osc-linear-upper}
\lambda^p
\underset{\mathcal M_1^\sharp(f)(x,t)>\lambda t^\gamma}
{\int_{\mathbb R^n}\int_0^\infty}
t^{\gamma p-1}\,dt\,\upsilon(x)\,dx\lesssim
[\upsilon]_{A_p}\|\nabla f\|_{L^p_\upsilon}^p,
\end{align}
where the implicit positive constant depends only on $n$, $p$,
and $\gamma$.
Taking the supremum over $\lambda\in(0,\infty)$
in \eqref{eq:osc-linear-upper} gives the upper
bound in Theorem \ref{th:intro-m}\textup{(ii)} for $\gamma\in\mathbb R\setminus[-1,0]$.
This completes the proof of the upper estimates in Theorem
\ref{th:intro-m}.
\end{proof}

\subsection{Exact Limiting Formulae and Lower Estimates}\label{sub:3-2}
In this subsection, by combining the pointwise limits in Lemma
\ref{lem:exact-sm-pointwise} with the upper estimates, we obtain
exact limiting formulae and hence the lower bounds in Theorem
\ref{th:intro-m}. We retain the notation in
\eqref{eq:unified-notation} and \eqref{eq:unified-gamma-range}.
In what follows, $\lambda\to L_\gamma$ means
$\lambda\to\infty$ when $\gamma\in(0,\infty)$ and
$\lambda\to0^+$ when $\gamma\in(-\infty,0)$.
By \eqref{eq:theta-change}, it suffices to consider $\theta=1$.

\begin{theorem}\label{th:exact-m}
Let $\tau\in\{\mathrm{av},\mathrm{osc}\}$, $p\in[1,\infty)$,
$\gamma\in\Gamma_p^\tau$, and $\upsilon\in A_p$.
Then, for any $f\in X^{\tau,p}_\upsilon$,
\begin{align}\label{eq:exact-m}
\lim_{\lambda\to L_\gamma}\lambda^p
\underset{\mathcal M_1^\tau(f)(x,t)>\lambda t^\gamma}
{\int_{\mathbb R^n}\int_0^\infty}
t^{\gamma p-1}\,dt\,\upsilon(x)\,dx
=\frac{(c_n^\tau)^p}{p|\gamma|}\|D^\tau f\|_{L^p_\upsilon}^p.
\end{align}
\end{theorem}

\begin{remark}
When $\upsilon\equiv1$, Theorem \ref{th:exact-m} identifies
the exact limits corresponding to the characterizations in
\cite[Theorem~2.3]{dlyyz} and \cite[Theorem~1.1]{dlyyz2}.
Moreover, the range $\gamma\in\Gamma_p^\tau$ in Theorem \ref{th:exact-m} is sharp (see Subsection \ref{subsec:excluded-parameters}).
\end{remark}

From Theorem \ref{th:exact-m}, we infer that the lower estimates in Theorem
\ref{th:intro-m} hold. Together with the upper estimates proved
in Subsection \ref{sub:3-1}, this completes the proof of Theorem \ref{th:intro-m}. Thus, it remains to show Theorem \ref{th:exact-m}.

\begin{proof}[Proof of Theorem \ref{th:exact-m}]
Fix $f\in X^{\tau,p}_\upsilon$. For any $\lambda\in(0,\infty)$,
let $\varepsilon:=\lambda^{-\frac1\gamma}$.
Using the change of variables $t=\varepsilon s$, we find that
\begin{align}\label{eq:exact-m-scaled}
\lambda^p
\underset{\mathcal M_1^\tau(f)(x,t)>\lambda t^\gamma}
{\int_{\mathbb R^n}\int_0^\infty}
t^{\gamma p-1}\,dt\,\upsilon(x)\,dx
=
\underset{\mathcal M_1^\tau(f)(x,\varepsilon s)>s^\gamma}
{\int_{\mathbb R^n}\int_0^\infty}
s^{\gamma p-1}\,ds\,\upsilon(x)\,dx.
\end{align}
Note that $\lambda\to L_\gamma$ if and only if
$\varepsilon\to0^+$. From Lemma \ref{lem:exact-sm-pointwise}, it follows that
\begin{align}\label{eq:exact-m-pointwise}
\lim_{\varepsilon\to0^+}
{\mathbf 1}_{\{\mathcal M_1^\tau(f)(x,\varepsilon s)>s^\gamma\}}
={\mathbf 1}_{\{c_n^\tau|D^\tau f(x)|>s^\gamma\}}
\end{align}
for almost every $(x,s)\in\mathbb R^{n+1}_+$.
Applying Fatou's lemma and integrating directly in $s$, we conclude that
\begin{align}\label{eq:exact-m-lower-final}
\liminf_{\lambda\to L_\gamma}\lambda^p
\underset{\mathcal M_1^\tau(f)(x,t)>\lambda t^\gamma}
{\int_{\mathbb R^n}\int_0^\infty}
t^{\gamma p-1}\,dt\,\upsilon(x)\,dx
\geq\frac{(c_n^\tau)^p}{p|\gamma|}\|D^\tau f\|_{L^p_\upsilon}^p.
\end{align}

We first prove the reverse inequality when $p\in(1,\infty)$.
Since \eqref{eq:exact-sm-control} holds, we deduce that
\begin{align*}
{\mathbf 1}_{\{\mathcal M_1^\tau(f)(x,\varepsilon s)>s^\gamma\}}
\leq{\mathbf 1}_{\{cM(|D^\tau f|)(x)>s^\gamma\}},
\end{align*}
where $c$ is the positive constant in
\eqref{eq:exact-sm-control}.
Using the argument leading to \eqref{eq:m-by-M} and
Lemma \ref{le:M-v}\textup{(i)}, we find that the right-hand side of the above inequality
is integrable with respect to $s^{\gamma p-1}\upsilon(x)\,dx\,ds$
and is independent of $\varepsilon$.
Thus, from \eqref{eq:exact-m-scaled}, \eqref{eq:exact-m-pointwise},
and the dominated convergence theorem, we infer that \eqref{eq:exact-m} holds when $p\in(1,\infty)$.

It remains to consider $p=1$. We first show the limiting
formula for $g\in C^\infty$ with $|D^{\tau}g|\in C_\mathrm{c}$.
Here $C^\infty$ and $C_{\mathrm{c}}$ denote,
respectively, the spaces of all smooth functions and all continuous functions
with compact support in $\mathbb R^n$.
Choose $R\in(0,\infty)$ such that
$\operatorname{supp}|D^\tau g|\subset B(\mathbf0,R)$.
If $|D^\tau g|=0$, the conclusion is immediate. Otherwise, by
\eqref{eq:exact-sm-control}, there exists $A\in(0,\infty)$
such that, for any $(x,t)\in\mathbb R^{n+1}_+$,
$\mathcal M_1^\tau(g)(x,t)\leq A$.
Fix $t\in(0,\infty)$ and
$x\in\mathbb R^n\setminus B(\mathbf0,R+4t)$.
For any ball $B\ni x$ with $\operatorname{rad}(B)\in[t,2t)$,
we have $B\subset B(x,4t)$ and hence
$B\cap B(\mathbf0,R)=\emptyset$.
Thus, $|D^{\tau}g|=0$ on the ball $B$ and hence $\mathcal M_1^\tau(g)(x,t)=0$.
Therefore, for any $t\in(0,\infty)$,
\begin{align}\label{eq:exact-m-support}
\{x\in\mathbb R^n:\mathcal M_1^\tau(g)(x,t)>0\}
\subset B(\mathbf0,R+4t).
\end{align}
If $\gamma\in(0,\infty)$ and
$\varepsilon\in(0,\frac{1}{4A^{\frac1\gamma}})$, then
\begin{align*}
{\mathbf 1}_{\{\mathcal M_1^\tau(g)(x,\varepsilon s)>s^\gamma\}}
\leq{\mathbf 1}_{B(\mathbf0,R+1)}(x)
{\mathbf 1}_{(0,A^{\frac1\gamma})}(s).
\end{align*}
The right-hand side of the above inequality is integrable with respect to
$s^{\gamma-1}\upsilon(x)\,dx\,ds$.
By \eqref{eq:exact-m-pointwise} and the dominated
convergence theorem, we obtain
\begin{align}\label{eq:exact-m-core}
\lim_{\lambda\to L_\gamma}\lambda
\underset{\mathcal M_1^\tau(g)(x,t)>\lambda t^\gamma}
{\int_{\mathbb R^n}\int_0^\infty}
t^{\gamma-1}\,dt\,\upsilon(x)\,dx
=\frac{c_n^\tau}{|\gamma|}\|D^\tau g\|_{L^1_\upsilon}.
\end{align}

We next consider $\gamma\in\Gamma_1^\tau\cap(-\infty,0)$.
Fix $\delta\in(0,\infty)$ and split the scaled integral in
\eqref{eq:exact-m-scaled} at $s=\frac\delta\varepsilon$.
For its first part, \eqref{eq:exact-m-support} gives
\begin{align*}
{\mathbf 1}_{\{\mathcal M_1^\tau(g)(x,\varepsilon s)>s^\gamma\}}
{\mathbf 1}_{(0,\frac\delta\varepsilon)}(s)
\leq{\mathbf 1}_{B(\mathbf0,R+4\delta)}(x)
{\mathbf 1}_{(A^{\frac1\gamma},\infty)}(s).
\end{align*}
Since the right-hand side of the above inequality is integrable with respect to
$s^{\gamma-1}\upsilon(x)\,dx\,ds$, from the dominated convergence theorem, it follows that
\begin{align}\label{eq:exact-m-small-part}
\lim_{\varepsilon\to0^+}
\underset{\mathcal M_1^\tau(g)(x,\varepsilon s)>s^\gamma}
{\int_{\mathbb R^n}\int_0^{\frac\delta\varepsilon}}
s^{\gamma-1}\,ds\,\upsilon(x)\,dx
=\frac{c_n^\tau}{|\gamma|}\|D^\tau g\|_{L^1_\upsilon}.
\end{align}

For the second part, first suppose that $\tau=\mathrm{av}$, or that $\tau=\mathrm{osc}$ and
$n=1$. In either case, the assumed range gives
$\gamma\in(-\infty,-n)$.
For any $r\in[R,\infty)$, applying Lemma \ref{lem:A1-doubling}
with $p=1$, $Q:=(-r,r)^n$, and $S:=B(\mathbf0,R)$, we find that
\begin{align}
\upsilon(B(\mathbf0,r))
\leq\upsilon(Q)
\leq[\upsilon]_{A_1}\frac{(2r)^n}{|B(\mathbf0,R)|}
\upsilon(B(\mathbf0,R))
\lesssim r^n,
\label{eq:weighted-ball-growth}
\end{align}
where the implicit positive constant is independent of $r$.
In particular, for any $\varepsilon,s\in(0,\infty)$,
$\upsilon(B(\mathbf0,R+4\varepsilon s))
\lesssim(1+\varepsilon s)^n.$
From \eqref{eq:exact-m-support} and Tonelli's theorem, we deduce that
\begin{align}\label{eq:exact-m-tail-av}
&\underset{\mathcal M_1^\tau(g)(x,\varepsilon s)>s^\gamma}
{\int_{\mathbb R^n}\int_{\frac\delta\varepsilon}^\infty}
s^{\gamma-1}\,ds\,\upsilon(x)\,dx\notag\\
&\qquad\leq\int_{\frac\delta\varepsilon}^\infty
s^{\gamma-1}\upsilon(B(\mathbf0,R+4\varepsilon s))\,ds\notag\\
&\qquad\lesssim
\int_{\frac\delta\varepsilon}^\infty
s^{\gamma-1}(1+\varepsilon s)^n\,ds
=\varepsilon^{-\gamma}\int_\delta^\infty
u^{\gamma-1}(1+u)^n\,du\to0
\end{align}
as $\varepsilon\to0^+,$
where the implicit positive constant is independent of
$\varepsilon$, and the last integral is finite because
$\gamma\in(-\infty,-n)$.

We now treat the second part when $\tau=\mathrm{osc}$ and
$n\in\mathbb N\setminus\{1\}$. Using the fact that $\mathbb R^n\setminus\overline{B(\mathbf0,R)}$ is connected
and $|\nabla g|=0$ there, we conclude that $g$ equals a constant
on this set. Subtracting this constant leaves
$\mathcal M_1^\sharp(g)$ unchanged, so we may assume that
$g\in C_{\mathrm{c}}^\infty:=C^\infty \cap C_{\mathrm{c}}$.
For any ball $B\ni x$ with
$r:=\operatorname{rad}(B)\in[\varepsilon s,2\varepsilon s)$, we obviously have
\begin{align*}
\frac1r\fint_B|g-g_B|
\leq\frac{2\|g\|_{L^1}}{\omega_n r^{n+1}}
\leq\frac{2\|g\|_{L^1}}{\omega_n}(\varepsilon s)^{-n-1}.
\end{align*}
Consequently, on the level set
$\{(x,s)\in\mathbb{R}^{n+1}_+:\mathcal M_1^\sharp(g)(x,\varepsilon s)>s^\gamma\}$,
one has $(\varepsilon s)^n\lesssim s^{-\frac{n\gamma}{n+1}}$.
Since $\varepsilon s\geq\delta$, from \eqref{eq:exact-m-support} and
\eqref{eq:weighted-ball-growth}, we infer that
\begin{align*}
\upsilon(\{x:\mathcal M_1^\sharp(g)(x,\varepsilon s)>s^\gamma\})
\lesssim(\varepsilon s)^n\lesssim s^{-\frac{n\gamma}{n+1}}.
\end{align*}
Applying Tonelli's theorem, we obtain
\begin{align}\label{eq:exact-m-tail-osc}
  \underset{\mathcal M_1^\sharp(g)(x,\varepsilon s)>s^\gamma}
{\int_{\mathbb R^n}\int_{\frac\delta\varepsilon}^\infty}
s^{\gamma-1}\,ds\,\upsilon(x)\,dx\lesssim\int_{\frac\delta\varepsilon}^\infty
s^{\frac\gamma{n+1}-1}\,ds
=\frac{n+1}{-\gamma}
\left(\frac\delta\varepsilon\right)^{\frac\gamma{n+1}}\to 0
\end{align}
as $\varepsilon\to0^+,$
where the implicit positive constant is independent of
$\varepsilon$.
From \eqref{eq:exact-m-small-part}, \eqref{eq:exact-m-tail-av},
and \eqref{eq:exact-m-tail-osc}, it follows that \eqref{eq:exact-m-core} holds for the full
range $\Gamma_1^\tau$.

Finally, choose a sequence $\{g_k\}_{k\in\mathbb N}$ in $C^{\infty}$ with $|D^{\tau}g_k|\in C_{\mathrm{c}}$ such that
$\|D^\tau (f-g_k)\|_{L^1_\upsilon}\to0$ as $k\to\infty$.
For $\tau=\mathrm{av}$, this follows from the density of $C_{\mathrm{c}}^\infty$ in
$L^1_\upsilon$; for $\tau=\mathrm{osc}$, it follows from
\cite[Corollary~3.19]{dly}.
For any $\eta\in(0,1)$, sublinearity gives
\begin{align*}
\{\mathcal M_1^\tau(f)>\lambda t^\gamma\}
\subset
\{\mathcal M_1^\tau(g_k)>(1-\eta)\lambda t^\gamma\}
\cup
\{\mathcal M_1^\tau(f-g_k)>\eta\lambda t^\gamma\}.
\end{align*}
By \eqref{eq:exact-m-core}, we find that
\begin{align*}
&\limsup_{\lambda\to L_\gamma}\lambda
\underset{\mathcal M_1^\tau(f)(x,t)>\lambda t^\gamma}
{\int_{\mathbb R^n}\int_0^\infty}
t^{\gamma-1}\,dt\,\upsilon(x)\,dx\\
&\quad\leq
\frac{c_n^\tau}{(1-\eta)|\gamma|}\|D^\tau g_k\|_{L^1_\upsilon}
+\frac1\eta\sup_{\lambda\in(0,\infty)}\lambda
\underset{\mathcal M_1^\tau(f-g_k)(x,t)>\lambda t^\gamma}
{\int_{\mathbb R^n}\int_0^\infty}
t^{\gamma-1}\,dt\,\upsilon(x)\,dx.
\end{align*}
From Theorem \ref{th:intro-m}\textup{(i)}, we deduce that
\begin{align*}
\sup_{\lambda\in(0,\infty)}\lambda
\underset{\mathcal M_1^\tau(f-g_k)(x,t)>\lambda t^\gamma}
{\int_{\mathbb R^n}\int_0^\infty}
t^{\gamma-1}\,dt\,\upsilon(x)\,dx
\lesssim\|D^\tau (f-g_k)\|_{L^1_\upsilon},
\end{align*}
where the implicit positive constant is independent of $k$.
Letting $k\to\infty$ and then $\eta\to0^+$, and applying
\eqref{eq:exact-m-lower-final}, we obtain \eqref{eq:exact-m}.
This completes the proof of Theorem \ref{th:exact-m}.
\end{proof}

\subsection{The Excluded Parameters}\label{subsec:excluded-parameters}
We now examine the excluded parameters, beginning with $\gamma=0$,
as motivated by \cite[Theorem~1.5 and Corollary~1.6]{bssvy24}.
We then prove that the nonzero parameter ranges for the exact
limiting formulae at $p=1$ cannot be enlarged and hence are sharp.
\begin{theorem}\label{th:optimal-gamma-range}
Let $\tau\in\{\mathrm{av},\mathrm{osc}\}$ and $p\in[1,\infty)$.
The condition $\gamma\in\Gamma_p^\tau$ in
Theorem \ref{th:intro-m} is necessary even when $\upsilon\equiv1$.
\end{theorem}
\begin{lemma}\label{prop:m-gamma-zero}
Let $\upsilon$ be a weight. Then the following statements hold.
\begin{enumerate}
\item[\textup{(i)}]
For any $f\in L^1_{\rm loc}$,
\begin{align}\label{eq:m-gamma-zero-av}
\inf\left\{\lambda\in(0,\infty):
\underset{\mathcal M_1(f)(x,t)>\lambda}
{\int_{\mathbb R^n}\int_0^\infty}
t^{-1}\,dt\,\upsilon(x)\,dx<\infty\right\}
=\|f\|_{L^\infty}.
\end{align}
\item[\textup{(ii)}]
For any $f\in W^{1,1}_{\rm loc}$,
\begin{align}\label{eq:m-gamma-zero-osc}
\inf\left\{\lambda\in(0,\infty):
\underset{\mathcal M_1^\sharp(f)(x,t)>\lambda}
{\int_{\mathbb R^n}\int_0^\infty}
t^{-1}\,dt\,\upsilon(x)\,dx<\infty\right\}
\sim\|\nabla f\|_{L^\infty},
\end{align}
where the positive equivalence constants depend only on $n$.
\end{enumerate}
Consequently, the integral in either of the above two statements
is finite for any $\lambda\in(0,\infty)$ if and only if
$|D^\tau f|=0$ almost everywhere.
\end{lemma}
\begin{proof}
We show the lower estimates  in \textup{(i)} and \textup{(ii)} simultaneously. If
$\lambda\in(0,c_n^\tau\|D^\tau f\|_{L^\infty})$, then
$E_\lambda:=\{x\in\mathbb{R}^n:c_n^\tau|D^\tau f(x)|>\lambda\}$ has positive
Lebesgue measure and hence positive $\upsilon$-measure.
By Lemma \ref{lem:exact-sm-pointwise}, for almost every
$x\in E_\lambda$, there exists $t_x\in(0,\infty)$ such that
$\mathcal M_1^\tau(f)(x,t)>\lambda$ for any $t\in(0,t_x)$.
Using Tonelli's theorem, we conclude that
\begin{align*}
\underset{\mathcal M_1^\tau(f)(x,t)>\lambda}
{\int_{\mathbb R^n}\int_0^\infty}
t^{-1}\,dt\,\upsilon(x)\,dx
\geq\int_{E_\lambda}\left[\int_0^{t_x}\frac{dt}{t}\right]
\upsilon(x)\,dx=\infty.
\end{align*}
Thus, the corresponding infimum is at least
$c_n^\tau\|D^\tau f\|_{L^\infty}$.

For $\tau=\mathrm{av}$, the upper estimate follows directly from
$\mathcal M_1(f)(x,t)\leq\|f\|_{L^\infty}$, which also
proves that the integral in \eqref{eq:m-gamma-zero-av} vanishes when
$\lambda\geq\|f\|_{L^\infty}$.
For $\tau=\mathrm{osc}$, \eqref{eq:exact-sm-control} gives
$\mathcal M_1^\sharp(f)(x,t)\leq c\|\nabla f\|_{L^\infty}$,
where the positive constant $c$ depends only on $n$, so the integral in \eqref{eq:m-gamma-zero-osc} vanishes when
$\lambda\geq c\|\nabla f\|_{L^\infty}$.
These observations show both upper estimates. Thus, \textup{(i)} and \textup{(ii)} hold.

Applying \textup{(i)} and \textup{(ii)}, we immediately obtain the last statement. This completes the proof of Lemma \ref{prop:m-gamma-zero}.
\end{proof}

The following two lemmas prove that, when $p=1$, the
nonzero parameter ranges in Theorem \ref{th:exact-m} cannot
be enlarged. Their proofs use different counterexamples.
\begin{lemma}\label{prop:exact-m-gamma}
Let $p=1$ and $\gamma\in[-n,0)$. Then there exists a nonzero
function $f\in L^1$ such that formula
\eqref{eq:exact-m} with $\upsilon\equiv1$ does not hold. More
precisely, if $\gamma=-n$, then
\begin{align*}
\lim_{\lambda\to0^+}
\lambda
\underset{\mathcal M_1(f)(x,t)>\lambda t^{-n}}
{\int_{\mathbb R^n}\int_0^\infty}
t^{-n-1}\,dt\,dx=\infty
\end{align*}
and, if $\gamma\in(-n,0)$, then
\begin{align*}
\liminf_{\lambda\to0^+}
\lambda
\underset{\mathcal M_1(f)(x,t)>\lambda t^\gamma}
{\int_{\mathbb R^n}\int_0^\infty}
t^{\gamma-1}\,dt\,dx
>
\frac{1}{|\gamma|}\|f\|_{L^1}.
\end{align*}
\end{lemma}
\begin{proof}
Throughout the proof, let
$f:=\mathbf 1_{B(\mathbf 0,1)}$. First, let $\gamma\in(-n,0)$ and fix
$\delta\in(0,\frac{1}{3})$. Note that the argument used to show
\eqref{eq:exact-m-small-part} requires only $\gamma<0$. From  this argument  with $\upsilon\equiv1$ and $g$ therein replaced by $f$, combined with the changes of variables $\varepsilon:=\lambda^{-\frac{1}{\gamma}}$
and $t=\varepsilon s$, we infer that
\begin{align}\label{eq:exact-m-gamma-local}
\lim_{\lambda\to0^+}
\lambda
\underset{\mathcal M_1(f)(x,t)>\lambda t^\gamma}
{\int_{\mathbb R^n}\int_0^\delta}
t^{\gamma-1}\,dt\,dx
=\frac{1}{|\gamma|}\|f\|_{L^1}.
\end{align}

If $t\in[\frac{2}{3},\infty)$ and
$x\in B(\mathbf 0,\frac{3t}{2})$, then, by $B(\mathbf 0,1)\subset B(\mathbf 0,\frac{3t}{2})$
and the definitions of $f$ and $\mathcal M_1$, we obtain
\begin{align}\label{eq:exact-m-gamma-tail-bound}
\mathcal M_1(f)(x,t)
\geq
\frac{1}{\omega_n(\frac{3t}{2})^n}
\int_{B(\mathbf 0,\frac{3t}{2})}f(y)\,dy=\left(\frac{3t}{2}\right)^{-n}.
\end{align}
For any $\lambda\in(0,\infty)$, let
$T_\lambda:=[(\frac{3}{2})^{-n}\lambda^{-1}]^{\frac{1}{n+\gamma}}$.
Since $n+\gamma\in(0,\infty)$, it follows that $T_\lambda\to\infty$ as
$\lambda\to0^+$. Using \eqref{eq:exact-m-gamma-tail-bound}, we find that $\mathcal M_1(f)(x,t)>\lambda t^\gamma$ whenever
$\frac{2}{3}\leq t<T_\lambda$ and
$x\in B(\mathbf 0,\frac{3t}{2})$. Thus,
\begin{align}\label{eq:exact-m-gamma-tail-lower}
\liminf_{\lambda\to0^+}\lambda
\underset{\mathcal M_1(f)(x,t)>\lambda t^\gamma}
{\int_{\mathbb R^n}\int_{\frac{2}{3}}^\infty}
t^{\gamma-1}\,dt\,dx
&\geq
\lim_{\lambda\to0^+}\lambda
\int_{\frac{2}{3}}^{T_\lambda}
\int_{B(\mathbf 0,\frac{3t}{2})}
t^{\gamma-1}\,dx\,dt\notag\\
&=
\lim_{\lambda\to0^+}
\frac{\lambda\omega_n(\frac{3}{2})^n}{n+\gamma}
\left[T_\lambda^{n+\gamma}
-\left(\frac{2}{3}\right)^{n+\gamma}\right]
=\frac{\omega_n}{n+\gamma}.
\end{align}
From $\delta\in(0,\frac23)$, together with
\eqref{eq:exact-m-gamma-local} and
\eqref{eq:exact-m-gamma-tail-lower},
we deduce that
\begin{align*}
&\liminf_{\lambda\to0^+}\lambda
\underset{\mathcal M_1(f)(x,t)>\lambda t^\gamma}
{\int_{\mathbb R^n}\int_0^\infty}
t^{\gamma-1}\,dt\,dx\\
&\quad\geq
\lim_{\lambda\to0^+}\lambda
\underset{\mathcal M_1(f)(x,t)>\lambda t^\gamma}
{\int_{\mathbb R^n}\int_0^\delta}
t^{\gamma-1}\,dt\,dx
+
\liminf_{\lambda\to0^+}\lambda
\underset{\mathcal M_1(f)(x,t)>\lambda t^\gamma}
{\int_{\mathbb R^n}\int_{\frac{2}{3}}^\infty}
t^{\gamma-1}\,dt\,dx\\
&\quad\geq
\frac{1}{|\gamma|}\|f\|_{L^1}
+\frac{\omega_n}{n+\gamma}
>
\frac{1}{|\gamma|}\|f\|_{L^1},
\end{align*}
which proves the desired strict inequality for $\gamma\in(-n,0)$.

Finally, consider the case where $\gamma=-n$. If
$\lambda\in(0,(\frac{3}{2})^{-n})$, then
\eqref{eq:exact-m-gamma-tail-bound} implies that
\begin{align*}
\lambda
\underset{\mathcal M_1(f)(x,t)>\lambda t^{-n}}
{\int_{\mathbb R^n}\int_0^\infty}
t^{-n-1}\,dt\,dx
\geq
\lambda\omega_n\left(\frac{3}{2}\right)^n
\int_{\frac{2}{3}}^\infty t^{-1}\,dt
=\infty.
\end{align*}
Consequently, the desired limit is infinite.
This completes the proof of Lemma \ref{prop:exact-m-gamma}.
\end{proof}

Using \cite[Proposition~5.2]{zlyyz24} and
$\frac{m_f(x,t)}{t}\leq\mathcal M_1^\sharp(f)(x,t)$,
we immediately obtain the following lemma.
\begin{lemma}\label{lem:exact-sm-gamma}
Let $p=1$ and $\gamma\in[-1,0)$. Then there exists a function
$f\in W^{1,1}$ such that
\begin{align*}
\limsup_{\lambda\to0^+}
\lambda
\underset{\mathcal M_1^\sharp(f)(x,t)>\lambda t^\gamma}
{\int_{\mathbb R^n}\int_0^\infty}
t^{\gamma-1}\,dt\,dx
=\infty.
\end{align*}
\end{lemma}
\begin{proof}[Proof of Theorem \ref{th:optimal-gamma-range}]
For $\upsilon\equiv1$, Lemma \ref{prop:m-gamma-zero} excludes $\gamma=0$.
When $p=1$ and $\tau=\mathrm{av}$, Lemma \ref{prop:exact-m-gamma} and
the necessity part of \cite[Theorem~1.1]{dlyyz} exclude $\gamma=-n$
and $\gamma\in(-n,0)$, respectively.
When $p=1$ and $\tau=\mathrm{osc}$, Lemma \ref{lem:exact-sm-gamma}
excludes $\gamma\in[-1,0)$.
This then completes the proof of Theorem \ref{th:optimal-gamma-range}.
\end{proof}

\section{Sharpness of Powers of Muckenhoupt Weight Constants}
\label{sec:m-sharp}
We next show the sharpness assertions in Theorem \ref{th:intro-m}.
The upper bounds have different optimal powers when
$\gamma\in[-n,-1)$ and $n\in\mathbb N\setminus\{1\}$.
Their sharpness requires different arguments. We therefore treat
$\mathcal M_\theta$ and $\mathcal M_\theta^\sharp$, respectively, in Subsections \ref{subsec:sharp-m} and \ref{subsec:sharp-sm}.

For $A_1$ weights, the power $1$ is optimal for both operators.
For $A_p$ weights with $p\in(1,\infty)$, the results are
summarized below:
\begin{center}
\begin{tabular}{c|c|c}
$\gamma$ & $\color{red}\mathcal M_\theta$: optimal power
& $\color{blue}\mathcal M_\theta^\sharp$: optimal power\\
\hline
$\mathbb R\setminus[-n,0]$ & $1$ & $1$\\
$[-1,0)$ & $\frac{p}{p-1}$ & $\frac{p}{p-1}$\\
$[-n,-1)$, $n\in\mathbb N\setminus\{1\}$
& $\frac{p}{p-1}$ & $1$
\end{tabular}
\end{center}
The powers in this table are sharp for every fixed admissible
$\gamma$.

\subsection{Sharpness for \texorpdfstring{$\mathcal M_\theta$}{M(theta)}}
\label{subsec:sharp-m}

We first prove that the powers for
$\mathcal M_\theta$ in Theorem \ref{th:intro-m} are sharp. For
$\gamma\in(-n,0)$, the proofs use the weights defined
by powers of the distance to Ahlfors regular fractal sets of Hausdorff dimension $n+\gamma$.
\begin{proposition}\label{pro:sharp-m}
Let $p\in[1,\infty)$ and $\gamma\in\mathbb R$.
Then the following statements hold with the positive constants
$C$ independent of $\upsilon$ and $f$.
\begin{enumerate}
\item[\textup{(i)}]Assume that $\gamma\in\mathbb R\setminus\{0\}$ when $p\in(1,\infty)$ and $\gamma\in\mathbb R\setminus[-n,0]$ when $p=1$. If there exist $\eta\in\mathbb R$ and a positive constant $C$ such that, for any $\upsilon\in A_1$ and $f\in L^p_\upsilon$,
\begin{align}\label{eq:sharp-m-p-1}
\sup_{\theta,\lambda\in(0,\infty)}\lambda^p\underset{\mathcal M_\theta(f)(x,t)>\lambda t^\gamma}{\int_{\mathbb R^n}\int_0^\infty}t^{\gamma p-1}\,dt\,\upsilon(x)\,dx
\le C[\upsilon]_{A_1}^{\eta}\|f\|_{L^p_\upsilon}^p,
\end{align}
then $\eta\ge1$.

\item[\textup{(ii)}]
Let $p\in(1,\infty)$ and $\gamma\in[-n,0)$.
If there exist $\eta\in\mathbb R$ and a positive constant $C$
such that, for any $\upsilon\in A_p$ and
$f\in L^p_\upsilon$,
\begin{align}\label{eq:sharp-m-p-2}
\sup_{\theta,\lambda\in(0,\infty)}
\lambda^p
\underset{\mathcal M_\theta(f)(x,t)>\lambda t^\gamma}
{\int_{\mathbb R^n}\int_0^\infty}
t^{\gamma p-1}\,dt\,\upsilon(x)\,dx
\leq C[\upsilon]_{A_p}^{\eta}\|f\|_{L^p_\upsilon}^p,
\end{align}
then $\eta\geq\frac{p}{p-1}$.

\item[\textup{(iii)}]
Let $p\in(1,\infty)$ and $\gamma\in\mathbb R\setminus[-n,0]$. If there exist $\eta\in\mathbb R$ and a positive constant $C$ such that, for any $\upsilon\in A_p$ and $f\in L^p_\upsilon$,
\begin{align}\label{eq:sharp-m-p-3}
\sup_{\theta,\lambda\in(0,\infty)}\lambda^p
\underset{\mathcal M_\theta(f)(x,t)>\lambda t^\gamma}{\int_{\mathbb R^n}\int_0^\infty}
t^{\gamma p-1}\,dt\,\upsilon(x)\,dx
\le C[\upsilon]_{A_p}^{\eta}
\|f\|_{L^p_\upsilon}^p,
\end{align}
then $\eta\ge1$.

\end{enumerate}
\end{proposition}

We first establish two auxiliary lemmas needed in the proof of
Proposition \ref{pro:sharp-m}\textup{(ii)} for
$\gamma\in(-n,0)$. Throughout Lemmas \ref{lem:fixed-gamma-tubular-integrals} and \ref{lem:fixed-gamma-weight-test} and their proofs, let
$p\in(1,\infty)$, $\gamma\in(-n,0)$, and
$d:=n+\gamma$.
For the terminology and basic facts concerning self-similar sets used
below, we refer to \cite[Chapter~4]{mattila95}.
For any $e\in\{0,1\}^n$, define the \textit{contracting similarity}
\begin{align*}
S_e(x):=2^{-\frac{n}{d}}x+
\left(1-2^{-\frac{n}{d}}\right)e, \quad\forall\, x\in\mathbb R^n.
\end{align*}
Let $E$ be the invariant set of the family
$\{S_e\}_{e\in\{0,1\}^n}$, namely the unique nonempty compact subset
of $[0,1]^n$ satisfying
\begin{align}\label{eq:4.3x}
E=\bigcup_{e\in\{0,1\}^n}S_e(E).
\end{align}
Since $2^{-\frac{n}{d}}<\frac{1}{2}$, the family
$\{S_e\}_{e\in\{0,1\}^n}$ satisfies the open set condition (see \cite[p.\,67]{mattila95} for its precise definition) with
$O=(0,1)^n$.
By $\sum_{e\in\{0,1\}^n}(2^{-\frac{n}{d}})^d=1$
and \cite[(1) in Theorem~4.14]{mattila95}, we conclude that
the similarity dimension of the family
$\{S_e\}_{e\in\{0,1\}^n}$ is $d$. Applying \cite[Theorem~4.14]{mattila95}, we obtain
$\mathcal H^d(E)\in(0,\infty)$ and, for any $z\in E$ and $r\in(0,1)$,
\begin{align}\label{eq:fixed-gamma-ahlfors}
C_E^{-1}r^d
\leq \mathcal H^d(E\cap B(z,r))
\leq C_Er^d,
\end{align}
where $\mathcal H^d$ denotes
the $d$-dimensional Hausdorff measure and $C_E\in[1,\infty)$ depends
only on $n$, $E$, and $\gamma$. Using $d<n$ and $\mathcal H^d(E)<\infty$, we conclude that $|E|=0$.

Let\begin{align}\label{4.4x}
   r_0:=\frac{1}{20} \ \text{and} \ \rho:=\min\{\operatorname{dist}(\cdot,E),1\}.
\end{align}
For any $r\in(0,\infty)$, define
$$\mathcal T_r(E):=\{x\in\mathbb R^n:\operatorname{dist}(x,E)<r\}.$$
The first lemma estimates the integrals of powers of $\rho$ over
$B(z,R)$ and $\mathcal T_R(E)$.

\begin{lemma}\label{lem:fixed-gamma-tubular-integrals}
Let $\delta\in(0,\min\{-\frac{\gamma}{2},1\})$.
Then, for any $z\in E$ and $R\in(0,r_0]$,
\begin{align}\label{eq:fixed-gamma-local-negative}
\int_{B(z,R)}\rho(x)^{\gamma+\delta}\,dx
 \sim \frac{1}{\delta}R^{d+\delta}
\end{align}
and
\begin{align}\label{eq:fixed-gamma-local-positive}
\int_{B(z,R)}\rho(x)^{(p-1)(-\gamma-\delta)}\,dx
\sim R^{d-p\gamma-(p-1)\delta}.
\end{align}
Moreover,
\begin{align}\label{eq:fixed-gamma-global-negative}
\int_{\mathcal T_R(E)}\rho(x)^{\gamma+\delta}\,dx
\sim \frac{1}{\delta}R^\delta
\end{align}
and
\begin{align}\label{eq:fixed-gamma-global-positive}
\int_{\mathcal T_R(E)}\rho(x)^{(p-1)(-\gamma-\delta)}\,dx
\sim R^{-p\gamma-(p-1)\delta},
\end{align}
where the positive equivalence constants depend only on $n$, $p$,
and $\gamma$.
\end{lemma}

\begin{proof}
Fix $z\in E$ and $R\in(0,r_0]$. We first establish the volume estimate
\begin{align}\label{eq:fixed-gamma-local-tube}
|B(z,R)\cap\mathcal T_r(E)|\sim R^dr^{-\gamma}
\end{align}
for any $r\in(0,R)$.
Let $\{z_j\}_{j=1}^N$ be a maximal $r$-separated subset of
$E\cap B(z,\frac{R}{2})$. By maximality, the balls $B(z_j,r)$,
$j\in\{1,\ldots,N\}$, cover $E\cap B(z,\frac{R}{2})$.
Thus, since \eqref{eq:fixed-gamma-ahlfors} holds, we infer that
\begin{align*}
R^d
\lesssim \mathcal H^d\left(E\cap B\left(z,\frac{R}{2}\right)\right)
\leq \sum_{j=1}^N\mathcal H^d(E\cap B(z_j,r))
\lesssim Nr^d.
\end{align*}
The balls $\{B(z_j,\frac{r}{4})\}_{j=1}^N$ are pairwise disjoint and are
contained in $B(z,R)\cap\mathcal T_r(E)$. Consequently,
\begin{align*}
|B(z,R)\cap\mathcal T_r(E)|
\geq \sum_{j=1}^N\left|B\left(z_j,\frac{r}{4}\right)\right|
\gtrsim Nr^n
\gtrsim R^dr^{-\gamma}.
\end{align*}
For the upper estimate, choose a maximal $r$-separated subset
$\{\zeta_j\}_{j=1}^M$ of $E\cap B(z,2R)$.
The balls $\{B(\zeta_j,\frac{r}{2})\}_{j=1}^M$ are pairwise disjoint and lie
in $B(z,3R)$. By \eqref{eq:fixed-gamma-ahlfors}, we find that
\begin{align*}
Mr^d
\lesssim \sum_{j=1}^M
\mathcal H^d\left(E\cap B\left(\zeta_j,\frac{r}{2}\right)\right)
\leq \mathcal H^d(E\cap B(z,3R))
\lesssim R^d.
\end{align*}
For any $x\in B(z,R)\cap\mathcal T_r(E)$, choose $y\in E$ such that
$|x-y|<r$. Since $r<R$, it follows that $y\in E\cap B(z,2R)$.
By maximality, there exists $j\in\{1,\ldots,M\}$ such that
$y\in B(\zeta_j,r)$, and hence $x\in B(\zeta_j,2r)$. Thus, we have
\begin{align*}
|B(z,R)\cap\mathcal T_r(E)|
\leq \sum_{j=1}^M|B(\zeta_j,2r)|
\lesssim Mr^n
\lesssim R^dr^{-\gamma},
\end{align*}
which shows \eqref{eq:fixed-gamma-local-tube}.
Next, let
$\{\zeta_j\}_{j=1}^M$ be a maximal $r$-separated subset of $E$.
By maximality, the balls $\{B(\zeta_j,r)\}_{j=1}^M$ cover $E$, and the balls
$\{B(\zeta_j,\frac{r}{2})\}_{j=1}^M$ are pairwise disjoint. Applying \eqref{eq:fixed-gamma-ahlfors}, we conclude that
\begin{align*}
Mr^d
\lesssim
\sum_{j=1}^M
\mathcal H^d\left(E\cap B\left(\zeta_j,\frac{r}{2}\right)\right)\leq \mathcal H^d(E)
\leq \sum_{j=1}^M\mathcal H^d(E\cap B(\zeta_j,r))
\lesssim Mr^d.
\end{align*}
Since $\mathcal H^d(E)\in(0,\infty)$, we deduce that
$M\sim r^{-d}.$
Moreover, for any $x\in\mathcal T_r(E)$, there exist $y\in E$ and
$j\in\{1,\ldots,M\}$ such that $|x-y|<r$ and
$y\in B(\zeta_j,r)$. This further implies that $x\in B(\zeta_j,2r)$, which proves that
\begin{align*}
\bigcup_{j=1}^M B\left(\zeta_j,\frac{r}{2}\right)
\subset\mathcal T_r(E)
\subset\bigcup_{j=1}^M B(\zeta_j,2r).
\end{align*}
Consequently, we have
\begin{align}\label{eq:fixed-gamma-global-tube}
|\mathcal T_r(E)|
\sim Mr^n
\sim r^{n-d}
=r^{-\gamma}
\end{align}
for any $r\in(0,r_0]$, where the positive equivalence
constants in both \eqref{eq:fixed-gamma-local-tube} and \eqref{eq:fixed-gamma-global-tube}
depend only on $n$ and $\gamma$.

We next derive
\eqref{eq:fixed-gamma-local-negative}--\eqref{eq:fixed-gamma-global-positive} from
\eqref{eq:fixed-gamma-local-tube} and
\eqref{eq:fixed-gamma-global-tube}. Using $|E|=0$, we obtain
$0<\rho<R$ almost everywhere on $B(z,R)$.
From Tonelli's theorem, we infer that
\begin{align*}
\int_{B(z,R)}\rho(x)^{\gamma+\delta}\,dx
&=R^{\gamma+\delta}|B(z,R)|
 -(\gamma+\delta)\int_{B(z,R)}\int_{\rho(x)}^R
r^{\gamma+\delta-1}\,dr\,dx\\
&=R^{\gamma+\delta}|B(z,R)|
-(\gamma+\delta)\int_0^R r^{\gamma+\delta-1}
|\{x\in B(z,R):\rho(x)<r\}|\,dr\\
&=R^{\gamma+\delta}|B(z,R)|
-(\gamma+\delta)\int_0^R
r^{\gamma+\delta-1}|B(z,R)\cap\mathcal T_r(E)|\,dr\\
&\sim R^{d+\delta}
-(\gamma+\delta)R^d\int_0^R r^{\delta-1}\,dr
\sim\frac{1}{\delta}R^{d+\delta},
\end{align*}
where the first equivalence follows from
\eqref{eq:fixed-gamma-local-tube} and the second follows from
$\frac{\gamma}{2}>\gamma+\delta>\gamma$.
This shows \eqref{eq:fixed-gamma-local-negative}.
By the same calculation with $B(z,R)$ replaced by $\mathcal T_R(E)$,
and by \eqref{eq:fixed-gamma-global-tube}, we find that
\eqref{eq:fixed-gamma-global-negative} holds.

To prove \eqref{eq:fixed-gamma-local-positive} and
\eqref{eq:fixed-gamma-global-positive}, using \eqref{eq:fixed-gamma-local-tube} and
\eqref{eq:fixed-gamma-global-tube}, we choose
$\kappa\in(0,1)$ sufficiently small such that
$|B(z,R)\setminus\mathcal T_{\kappa R}(E)|\gtrsim R^n$ and
$|\mathcal T_R(E)\setminus\mathcal T_{\kappa R}(E)|
\gtrsim R^{-\gamma}.$
On these two sets, one has $\kappa R\leq\rho\leq R$.
Using integration over these sets, we conclude that the lower estimates in
\eqref{eq:fixed-gamma-local-positive} and
\eqref{eq:fixed-gamma-global-positive} hold. The corresponding upper
estimates follow from $\rho\leq R$ on $B(z,R)$ and
$\mathcal T_R(E)$.
The estimates are uniform in $\delta$ because
$(p-1)(\gamma+\delta)\in((p-1)\gamma, 0)$.
This completes the proof of Lemma
\ref{lem:fixed-gamma-tubular-integrals}.
\end{proof}
For any $\delta\in(0,\min\{-\frac{\gamma}{2},1\})$, define
\begin{align}\label{eq:fixed-gamma-weight-test}
\upsilon_\delta:=\rho^{-(p-1)(\gamma+\delta)}
\ \text{and}\
f_\delta(x):=
\begin{cases}
\rho(x)^{\gamma+\delta}
&\text{if } x\in\mathcal T_{r_0}(E)\setminus E,\\
0
&\text{if } x\notin\mathcal T_{r_0}(E)\setminus E.
\end{cases}
\end{align}
Since $|E|=0$, from \eqref{eq:fixed-gamma-global-negative} with $R=r_0$, we deduce that
\begin{align}\label{eq:fixed-gamma-test-norm}
\|f_\delta\|_{L^p_{\upsilon_\delta}}^p
=\int_{\mathcal T_{r_0}(E)\setminus E}
\rho(x)^{\gamma+\delta}\,dx
\sim\frac{1}{\delta}.
\end{align}
In particular, $f_\delta\in L^p_{\upsilon_\delta}\cap L^1$.
It remains to estimate $[\upsilon_\delta]_{A_p}$.

\begin{lemma}\label{lem:fixed-gamma-weight-test}
Let $\delta\in(0,\min\{-\frac{\gamma}{2},1\})$, and let
$\upsilon_\delta$ be as in \eqref{eq:fixed-gamma-weight-test}.
Then one has
\begin{align}\label{eq:fixed-gamma-Ap}
\upsilon_\delta\in A_p
\ \text{and} \
[\upsilon_\delta]_{A_p}\sim\delta^{1-p},
\end{align}
where the positive equivalence constants depend only on $n$, $p$,
and $\gamma$.
\end{lemma}

\begin{proof}
For any $z\in E$ and $R\in(0,r_0]$, by \eqref{eq:fixed-gamma-local-negative} and
\eqref{eq:fixed-gamma-local-positive}, we obtain
\begin{align*}
\left[\fint_{B(z,R)}\upsilon_\delta(x)\,dx\right]
\left\{\fint_{B(z,R)}
[\upsilon_\delta(x)]^{-\frac{1}{p-1}}\,dx\right\}^{p-1}
&\sim R^{(p-1)(-\gamma-\delta)}
\left(\frac{1}{\delta}R^{\gamma+\delta}\right)^{p-1}=\delta^{1-p}.
\end{align*}
Using this and \eqref{eq:c-v-p} shows the lower estimate in \eqref{eq:fixed-gamma-Ap}.
To prove the upper estimate, fix a ball $B:=B(x_B,R)$.
It suffices to show that
\begin{align}\label{eq:fixed-gamma-Ap-upper-ball}
\left[\fint_B\upsilon_\delta(x)\,dx\right]
\left\{\fint_B[\upsilon_\delta(x)]^{-\frac{1}{p-1}}\,dx\right\}^{p-1}
\lesssim\delta^{1-p},
\end{align}
where the implicit positive constant is independent of $B$ and $\delta$. Next, we prove this by considering three cases for $R$.

\textit{Case 1: $R\geq\frac{r_0}{4}$.}
In this case, since $\upsilon_\delta\leq1$, the first factor on the left-hand
side of \eqref{eq:fixed-gamma-Ap-upper-ball} is bounded by $1$.
It remains to estimate the second factor.
The fact that $\rho=1$ outside $\mathcal T_1(E)$ yields
\begin{align}\label{eq:4.16x}
\fint_B[\upsilon_\delta(x)]^{-\frac{1}{p-1}}\,dx
={}&\frac{1}{|B|}
\int_{B\cap\mathcal T_1(E)}
\rho(x)^{\gamma+\delta}\,dx
+\frac{|B\setminus\mathcal T_1(E)|}{|B|}.
\end{align}
The second term on the right-hand side of \eqref{eq:4.16x} is bounded by $1$.
To estimate its first term, note that
$E\subset[0,1]^n$ implies
$\mathcal T_1(E)\subset(-1,2)^n$ and hence
$|\mathcal T_1(E)|<\infty$. Moreover, by the definition of $\mathcal T_{r_0}(E)$, we have
$\rho\geq r_0$ on
$\mathcal T_1(E)\setminus\mathcal T_{r_0}(E)$.
Thus, since \eqref{eq:fixed-gamma-global-negative} holds, we infer that
\begin{align*}
\int_{\mathcal T_1(E)}\rho(x)^{\gamma+\delta}\,dx
\leq
\int_{\mathcal T_{r_0}(E)}
\rho(x)^{\gamma+\delta}\,dx
+r_0^{\gamma}|\mathcal T_1(E)|\lesssim\frac{1}{\delta}.
\end{align*}
Using $R\in[\frac{r_0}{4},\infty)$, we find that $|B|\gtrsim1$, which further implies that
$\fint_B[\upsilon_\delta(x)]^{-\frac{1}{p-1}}\,dx
\lesssim\frac{1}{\delta}.$
From the above estimates, it follows that
\eqref{eq:fixed-gamma-Ap-upper-ball} holds in this case.

\textit{Case 2: $R<\frac{r_0}{4}$ and
$\operatorname{dist}(B,E)\geq2R$.}
In this case, for any $y_1,y_2\in B$, one has
$$\operatorname{dist}(y_1,E)
\leq\operatorname{dist}(y_2,E)+2R
\leq2\operatorname{dist}(y_2,E).$$
Hence, $\rho(y_1)\sim\rho(y_2)$ uniformly on $B$.
Since $(p-1)(\gamma+\delta)$ is bounded independently of $\delta$,
the left-hand side of \eqref{eq:fixed-gamma-Ap-upper-ball}
is bounded by a positive constant independent of $B$ and $\delta$.
This proves \eqref{eq:fixed-gamma-Ap-upper-ball} in this case because $\delta<1$.

\textit{Case 3: $R<\frac{r_0}{4}$ and
$\operatorname{dist}(B,E)<2R$.} In this case,
choose $z\in E$ such that $|x_B-z|<3R$. Then
$B\subset B(z,4R)$ and $4R<r_0$.
Applying \eqref{eq:fixed-gamma-local-negative} and
\eqref{eq:fixed-gamma-local-positive} to $B(z,4R)$, we conclude that
\begin{align*}
\fint_B\upsilon_\delta(x)\,dx
\lesssim R^{(p-1)(-\gamma-\delta)}
\ \text{and} \
\fint_B[\upsilon_\delta(x)]^{-\frac{1}{p-1}}\,dx
\lesssim\frac{1}{\delta}R^{\gamma+\delta},
\end{align*}
which shows \eqref{eq:fixed-gamma-Ap-upper-ball}.
The preceding three cases establish the desired upper bound in
\eqref{eq:fixed-gamma-Ap}. This completes the proof of Lemma
\ref{lem:fixed-gamma-weight-test}.
\end{proof}

Using these auxiliary estimates, we now prove Proposition
\ref{pro:sharp-m}.

\begin{proof}[Proof of Proposition \ref{pro:sharp-m}]
We first show that \textup{(i)} holds. Let
$\gamma\in\mathbb R\setminus\{0\}$ when $p\in(1,\infty)$ and
$\gamma\in\mathbb R\setminus[-n,0]$ when $p=1$.
For any $\delta\in(0,1)$, define
$\upsilon:=|\cdot|^{-n+\delta}$. From
\cite[Lemma~1.4]{b93}, we deduce that $\upsilon\in A_1$ and
$[\upsilon]_{A_1}\sim\delta^{-1},$
where the positive equivalence constants depend only on $n$. Choose
$f:=\mathbf1_{\{x\in\mathbb R^n:2<|x|<3\}}.$ Clearly, $\|f\|_{L^p_\upsilon} \sim 1$ with the positive equivalence constants independent of $\delta$.

Let $J:=B(\mathbf 0,1)$. If $t\in(2,4]$ and $x\in J$, then, using the definitions of $f$ and $\mathcal M_1$, we obtain
\begin{align*}
\mathcal M_1(f)(x,t)
\ge\fint_{B(\mathbf0,4)}f(y)\,dy
=:c_n>0.
\end{align*}
Choose $\lambda\in(0,\infty)$ so small that $\lambda t^\gamma<c_n$
for any $t\in(2,4]$. Then $\mathcal M_1(f)(x,t)>\lambda t^\gamma$
on $J\times(2,4]$. From this pointwise bound, we infer that
\begin{align*}
\sup_{\lambda\in(0,\infty)}
\lambda^p\underset{\mathcal M_1(f)(x,t)>\lambda t^\gamma}{\int_{\mathbb R^n}\int_0^\infty}
t^{\gamma p-1}\,dt\,\upsilon(x)\,dx\ge
\lambda^p\int_J\upsilon(x)\,dx
\int_2^4t^{\gamma p-1}\,dt
\gtrsim\frac{1}{\delta},
\end{align*}
where the implicit positive constant is independent of $\delta$. By assumption \eqref{eq:sharp-m-p-1}, we find that
\begin{align*}
\frac{1}{\delta}
\lesssim
[\upsilon]_{A_1}^{\eta}
\|f\|_{L^p_\upsilon}^p
\lesssim\frac{1}{\delta^{\eta}},
\end{align*}
where the implicit positive constants are independent of $\delta$. Since $\delta^{-1}\lesssim\delta^{-\eta}$ as $\delta\to0^+$, it follows that $\eta\ge1$, which completes the proof of \textup{(i)}.

Next, we prove \textup{(ii)}.
Fix $p\in(1,\infty)$ and $\gamma\in[-n,0)$.
We consider the following two cases for $\gamma$.

\textit{Case 1: $\gamma=-n$.}
In this case, let $\varepsilon\in(0,1)$, and define
$\upsilon:=|\cdot|^{(n-\varepsilon)(p-1)}$ and
$f:=|\cdot|^{-n+\varepsilon}\mathbf1_{B(\mathbf 0,1)}$.
Using \cite[Lemma~1.4]{b93} with $\delta:=\frac{\varepsilon}{n}$, we conclude that
$\upsilon\in A_p$ and
$[\upsilon]_{A_p}\sim\varepsilon^{1-p},$
where the positive equivalence constants depend only on $n$ and $p$. From a direct calculation, we deduce that
\begin{align}\label{eq:power-function-norm}
\|f\|_{L^p_\upsilon}^p
=\int_{B(\mathbf 0,1)}|x|^{-n+\varepsilon}\,dx
\sim\frac{1}{\varepsilon},
\end{align}
which implies that $f\in L^p_\upsilon$.

For any $t\in(0,1)$ and $x\in B(\mathbf 0,\frac{t}{2})$,
the ball $B(x,t)$ contains $B(\mathbf 0,\frac{t}{2})$.
This further implies that
\begin{align*}
\mathcal M_1(f)(x,t)
\geq\frac{1}{|B(x,t)|}\int_{B(\mathbf 0,\frac{t}{2})}|y|^{-n+\varepsilon}\,dy
\geq\frac{c}{\varepsilon}t^{-n+\varepsilon},
\end{align*}
where $c$ is a positive constant depending only on $n$.
By choosing $\lambda:=\frac{c}{2\varepsilon}$, we obtain
\begin{align}\label{eq:sharp-m-gamma-minus-n-lower}
&\sup_{\lambda\in(0,\infty)}
\lambda^p\underset{\mathcal M_1(f)(x,t)>\lambda t^{-n}}{\int_{\mathbb R^n}\int_0^\infty}
t^{-np-1}\,dt\,\upsilon(x)\,dx\notag\\
&\quad\ge
\left(\frac{c}{2\varepsilon}\right)^p
\int_{2^{-\frac{1}{\varepsilon}}}^1t^{-np-1}
\left[\int_{B(\mathbf 0,\frac{t}{2})}
|x|^{(n-\varepsilon)(p-1)}\,dx\right]\,dt\notag\\
&\quad\gtrsim
\varepsilon^{-p}
\int_{2^{-\frac{1}{\varepsilon}}}^1
t^{-1-(p-1)\varepsilon}\,dt
\sim\varepsilon^{-p-1},
\end{align}
where the implicit positive constants depend only on $n$ and $p$.
From \eqref{eq:sharp-m-gamma-minus-n-lower}, \eqref{eq:sharp-m-p-2},
\eqref{eq:power-function-norm}, and
$[\upsilon]_{A_p}\sim\varepsilon^{1-p}$, we infer that
\begin{align*}
\varepsilon^{-p-1}
\lesssim
[\upsilon]_{A_p}^{\eta}
\|f\|_{L^p_\upsilon}^p
\lesssim
\varepsilon^{-(p-1)\eta-1}.
\end{align*}
By letting $\varepsilon\to0^+$, we find that
$\eta\geq\frac{p}{p-1}$, which shows \textup{(ii)}
in the case where $\gamma=-n$.

\textit{Case 2: $\gamma\in(-n,0)$.} In this case,
let $E$, $r_0$, and $\rho$ be as in \eqref{eq:4.3x} and \eqref{4.4x}.
For any $\delta\in(0,\min\{-\frac{\gamma}{2},1\})$, let
$\upsilon_\delta$ and $f_\delta$ be as in
\eqref{eq:fixed-gamma-weight-test}.
For any $t\in(0,r_0)$ and $x\in\mathcal T_t(E)$, choose
$z\in E$ such that $|x-z|<t$. Then
$x\in B(z,t)\subset\mathcal T_{r_0}(E)$. Since $|E|=0$, from the definitions of $\mathcal M_1$ and $f_\delta$
and from \eqref{eq:fixed-gamma-local-negative}, it follows that
\begin{align*}
\mathcal M_1(f_\delta)(x,t)
\geq\frac{1}{|B(z,t)|}
\int_{B(z,t)\setminus E}\rho(y)^{\gamma+\delta}\,dy \geq\frac{c_*}{\delta}t^{\gamma+\delta},
\end{align*}
where $c_*$ is a positive constant depending only on $n$ and $\gamma$.
Let $\lambda_\delta:=\frac{c_*r_0^\delta}{2\delta}$ and
$\tau_\delta:=r_0 2^{-\frac{1}{\delta}}$.
For any $t\in(\tau_\delta,r_0)$,
one has $t^\delta>\frac{r_0^\delta}{2}$, and hence
$
\mathcal T_t(E)\subset
\{x\in\mathbb R^n:
\mathcal M_1(f_\delta)(x,t)>\lambda_\delta t^\gamma\}.$
Using this and \eqref{eq:fixed-gamma-global-positive} with $R=t$, we conclude that
\begin{align}
&\sup_{\lambda\in(0,\infty)}
\lambda^p
\underset{\mathcal M_1(f_\delta)(x,t)>\lambda t^\gamma}
{\int_{\mathbb R^n}\int_0^\infty}
t^{\gamma p-1}\,dt\,\upsilon_\delta(x)\,dx\notag\\
&\quad\geq\lambda_\delta^p
\int_{\tau_\delta}^{r_0}
\int_{\mathcal T_t(E)}
t^{\gamma p-1}\upsilon_\delta(x)\,dx\,dt
\gtrsim\delta^{-p}
\int_{\tau_\delta}^{r_0}
t^{-1-(p-1)\delta}\,dt\notag\\
&\quad=\frac{2^{p-1}-1}{p-1}
r_0^{-(p-1)\delta}\delta^{-p-1}
\sim\delta^{-p-1},
\label{eq:fixed-gamma-lifted-lower}
\end{align}
where we used $\lambda_\delta\sim\delta^{-1}$.
Applying \eqref{eq:sharp-m-p-2} with both
$\upsilon=\upsilon_\delta$ and $f=f_\delta$, we find that the left-hand
side of \eqref{eq:fixed-gamma-lifted-lower} admits an
upper estimate. Thus, by \eqref{eq:fixed-gamma-test-norm} and \eqref{eq:fixed-gamma-Ap}, we conclude that
\begin{align*}
\delta^{-p-1}
\lesssim[\upsilon_\delta]_{A_p}^\eta
\|f_\delta\|_{L^p_{\upsilon_\delta}}^p
\sim\delta^{(1-p)\eta-1},
\end{align*}
where the implicit positive constants are independent of $\delta$.
Since $\delta^{-p-1}\lesssim\delta^{(1-p)\eta-1}$ as $\delta\to0^+$, we deduce that
$\eta\geq\frac{p}{p-1}$.
This completes the proof of \textup{(ii)}.

Finally, we prove \textup{(iii)}.
Let $p\in(1,\infty)$, $\gamma\in\mathbb R\setminus[-n,0]$, and $\delta\in(0,1)$. Define
$\upsilon:=|\cdot|^{-n+\delta}$ and
$f:=\mathbf1_{\{x\in\mathbb R^n:1<|x|<2\}}$. Using \cite[Lemma~1.4]{b93}, we find that
$\upsilon\in A_p$ and $[\upsilon]_{A_p}\sim\delta^{-1}$,
where the positive equivalence constants depend only on $n$ and $p$.
From a direct calculation, we infer that
$\|f\|_{L^p_\upsilon}^p
\sim1.$
If $t\in(2,3]$ and $x\in B(\mathbf 0,1)$, then, by the definitions of $f$ and $\mathcal M_1$, we conclude that
$
\mathcal M_1(f)(x,t)
\ge c_n'>0.$
Choose $\lambda\in(0,\infty)$ so small that $\lambda t^\gamma<c_n'$ for
any $t\in(2,3]$. Then $\mathcal M_1(f)(x,t)>\lambda t^\gamma$ holds on
$B(\mathbf 0,1)\times(2,3]$. This further implies that
\begin{align}\label{eq:sharp-m-p-3-lower}
\sup_{\lambda\in(0,\infty)}
\lambda^p\underset{\mathcal M_1(f)(x,t)>\lambda t^\gamma}{\int_{\mathbb R^n}\int_0^\infty}
t^{\gamma p-1}\,dt\,\upsilon(x)\,dx
\ge
\lambda^p\int_{B(\mathbf 0,1)}\upsilon(x)\,dx
\int_2^3t^{\gamma p-1}\,dt
\gtrsim\frac{1}{\delta},
\end{align}
where the implicit positive constant is independent of $\delta$.
From \eqref{eq:sharp-m-p-3} and \eqref{eq:sharp-m-p-3-lower}, it follows that
$
\frac{1}{\delta}
\lesssim [\upsilon]_{A_p}^{\eta}\|f\|_{L^p_\upsilon}^p.$
Using $[\upsilon]_{A_p}\sim\delta^{-1}$ and
$\|f\|_{L^p_\upsilon}^p\sim1$, we obtain
$\delta^{\eta-1}\lesssim1$, with an implicit positive constant independent
of $\delta$. Letting $\delta\to0^+$, we obtain
$\eta\geq1$, which shows \textup{(iii)} and hence completes the proof
of Proposition \ref{pro:sharp-m}.
\end{proof}

\subsection{Sharpness for \texorpdfstring{$\mathcal M^\sharp_\theta$}{M sharp theta}}\label{subsec:sharp-sm}
We next prove the sharpness assertions for $\mathcal M^\sharp_\theta$
in Theorem \ref{th:intro-m}.
\begin{proposition}\label{prop:sharp-sm}
Let $n\in\mathbb N$ and $p\in[1,\infty)$.
In the following assertions, the positive constant $C$ is independent
of $\upsilon$ and $f$.
\begin{enumerate}
\item[\textup{(i)}]
 Assume that
$\gamma\in\mathbb R\setminus\{0\}$ when $p\in(1,\infty)$ and that
$\gamma\in\mathbb R\setminus[-1,0]$ when $p=1$.
If there exist $\eta\in\mathbb R$ and a positive constant $C$
such that, for any $\upsilon\in A_1$ and
$f\in\dot W^{1,p}_\upsilon$,
\begin{align}\label{eq:sharp-sm-AP-1}
\sup_{\lambda\in(0,\infty)}
\lambda^p
\underset{
\mathcal M_1^\sharp(f)(x,t)>\lambda t^\gamma}
{\int_{\mathbb R^n}\int_0^\infty}
t^{\gamma p-1}\,dt\,\upsilon(x)\,dx
\leq
C[\upsilon]_{A_1}^{\eta}
\|\nabla f\|_{L^p_\upsilon}^p,
\end{align}
then $\eta\geq1$.

\item[\textup{(ii)}]
Let $p\in(1,\infty)$ and
$\gamma\in\mathbb R\setminus[-1,0]$.
If there exist $\eta\in\mathbb R$ and a positive constant $C$
such that, for any $\upsilon\in A_p$ and
$f\in\dot W^{1,p}_\upsilon$,
\begin{align}\label{eq:sharp-sm-AP-3}
\sup_{\lambda\in(0,\infty)}
\lambda^p
\underset{
\mathcal M_1^\sharp(f)(x,t)>\lambda t^\gamma}
{\int_{\mathbb R^n}\int_0^\infty}
t^{\gamma p-1}\,dt\,\upsilon(x)\,dx
\leq
C[\upsilon]_{A_p}^{\eta}
\|\nabla f\|_{L^p_\upsilon}^p,
\end{align}
then $\eta\geq1$.

\item[\textup{(iii)}]
Let $p\in(1,\infty)$.
If there exist $\eta\in\mathbb R$ and a positive constant $C$
such that, for any $\upsilon\in A_p$ and
$f\in\dot W^{1,p}_\upsilon$, \eqref{eq:sharp-sm-AP-3} holds with $\gamma=-1$, then $\eta\geq\frac{p}{p-1}$.

\item[\textup{(iv)}]
Let $p\in(1,\infty)$ and $\gamma\in(-1,0)$.
If there exist $\eta\in\mathbb R$ and a positive constant $C$
such that, for any $\upsilon\in A_p$ and
$f\in\dot W^{1,p}_\upsilon$, \eqref{eq:sharp-sm-AP-3} holds,
then $\eta\geq\frac{p}{p-1}$.
\end{enumerate}
\end{proposition}
\begin{remark}
Recall that the positive constant $C$ in 
\cite[Proposition~2.11\textup{(iii)}]{scm} is required 
to be uniform in the parameter $\beta$. By contrast,
Proposition \ref{prop:sharp-sm}\textup{(iv)} provides a continuous counterpart
for $\mathcal M_1^\sharp$ in any dimension, allowing the positive constant $C$ in
\eqref{eq:sharp-sm-AP-3} to depend on the fixed $\gamma\in(-1,0)$.
\end{remark}
Before showing Proposition \ref{prop:sharp-sm}, we recall \cite[\textup{(i)} and \textup{(ii)} of
Proposition~2.11]{scm}, which is restated as follows according to the notation used in this article.
\begin{lemma}
\label{lem:sharp-weighted-cddd}
Let $p\in[1,\infty)$. In the following assertions, the positive
constant $C$ is independent of $\upsilon$ and $f$.
\begin{enumerate}
\item[\textup{(i)}]
Let $\beta\in\mathbb R$. If there exist $\eta\in\mathbb R$
and a positive constant $C$ such that, for any
$\upsilon\in A_1(\mathbb R)$ and
$f\in\dot W^{1,p}_\upsilon(\mathbb R)$,
\begin{align*}
\sup_{\lambda\in(0,\infty)}\lambda^p
\sum_{I\in\mathcal D^0_{\lambda,\beta+1-\frac1p}[f]}
|I|^{\beta p-1}\upsilon(I)
\leq C[\upsilon]_{A_1(\mathbb R)}^{\eta}
\|f'\|_{L^p_\upsilon(\mathbb R)}^p,
\end{align*}
then $\eta\geq1$.

\item[\textup{(ii)}]
Let $p\in(1,\infty)$. If there exist $\eta\in\mathbb R$
and a positive constant $C$ such that, for any
$\upsilon\in A_p(\mathbb R)$ and
$f\in\dot W^{1,p}_\upsilon(\mathbb R)$,
\begin{align*}
\sup_{\lambda\in(0,\infty)}\lambda^p
\sum_{I\in\mathcal D^1_{\lambda,0}[f]}
|I|^{-p}\upsilon(I)
\leq C[\upsilon]_{A_p(\mathbb R)}^{\eta}
\|f'\|_{L^p_\upsilon(\mathbb R)}^p,
\end{align*}
then $\eta\geq\frac{p}{p-1}$.

\end{enumerate}
Here and in what follows, for any $\alpha\in\{0,1\}$,
$\mathcal D^\alpha_{\lambda,b}[f]$ is as in \eqref{eq:3.1x} with $n=1$ and
$\mathcal D$ replaced by $\mathcal D^\alpha$ and hence, correspondingly, we replace $\mathcal D_{\lambda,b}[f]$ by $\mathcal D^\alpha_{\lambda,b}[f]$, where $\mathcal D^\alpha$ is
as in Example \ref{ex-dyadic}.
\end{lemma}
\begin{proof}[Proof of Proposition \ref{prop:sharp-sm}]
It suffices to prove $\eta\ge1$ in \textup{(i)}--\textup{(ii)} and
$\eta\ge\frac{p}{p-1}$ in \textup{(iii)}--\textup{(iv)} only for $\eta\in[0,\infty)$.
Indeed, $[\upsilon]_{A_q}^{\eta}\le1=[\upsilon]_{A_q}^{0}$
for any $q\in\{1,p\}$ and $\eta\in(-\infty,0)$. Thus,  the assumed
estimates \eqref{eq:sharp-sm-AP-1} and \eqref{eq:sharp-sm-AP-3} 
with a negative $\eta$ imply the corresponding estimates in the case $\eta=0$,
which contradict the shown lower bounds for $\eta$ under the assumption $\eta\in[0,\infty)$.

We first prove \textup{(i)}--\textup{(iv)}
when $n=1$, and then extend the argument to higher dimensions.

Fix $\gamma\in\mathbb R\setminus\{0\}$ and
$\lambda\in(0,\infty)$, and let
$\beta:=\gamma+\frac{1}{p}$. Let $\mathcal D$ be a dyadic grid
on $\mathbb R$.
For any $I\in\mathcal D_{\lambda,\gamma+1}[f]$,
$x\in I$, and $t\in(\frac{|I|}{4},\frac{|I|}{2}]$,
by the definition of $\mathcal M_1^\sharp$ we find that
\begin{align*}
\mathcal M_1^\sharp(f)(x,t)
\geq
\frac{2}{|I|}\fint_I |f(y)-f_I|\,dy
>c_{\gamma}\lambda t^\gamma,
\end{align*}
where $c_{\gamma}$ is a positive constant depending only on
$\gamma$. Since the boxes
$\{I\times(\frac{|I|}{4},\frac{|I|}{2}]\}_{I\in\mathcal D}$
are pairwise disjoint, we deduce that,
for any nonnegative weight $\upsilon$,
\begin{align}
&(c_\gamma\lambda)^p\underset{\mathcal M_1^\sharp(f)(x,t)
>c_\gamma\lambda t^\gamma}{\int_{\mathbb R}\int_0^\infty}t^{\gamma p-1}\,dt\,\upsilon(x)\,dx\notag\\
&\quad\geq(c_\gamma\lambda)^p\sum_{I\in\mathcal D_{\lambda,\gamma+1}[f]}\int_{I}\int_{\frac{|I|}{4}}^{\frac{|I|}{2}}t^{\gamma p-1}\,dt\,\upsilon(x)\,dx
\gtrsim\lambda^{p}\sum_{I\in\mathcal D_{\lambda,\beta+1-\frac{1}{p}}[f]}
|I|^{\beta p-1}\upsilon(I),
\label{eq:discrete-by-lifted-sharp}
\end{align}
where the implicit positive constant depends only on $p$ and
$\gamma$.
Using \eqref{eq:discrete-by-lifted-sharp}, Lemma \ref{lem:sharp-weighted-cddd}\textup{(i)}
with $\mathcal D=\mathcal D^0$, and the hypothesis \eqref{eq:sharp-sm-AP-1}, we conclude that assertion \textup{(i)} holds. For $p\in(1,\infty)$, the hypothesis
in \textup{(ii)} implies the one in \textup{(i)} because
$[\upsilon]_{A_p(\mathbb R)}\leq[\upsilon]_{A_1(\mathbb R)}$ and $\eta\geq0$.
Thus, \textup{(ii)} follows from \textup{(i)}.
Similarly, from \eqref{eq:discrete-by-lifted-sharp} and
Lemma \ref{lem:sharp-weighted-cddd}\textup{(ii)} with $\mathcal D=\mathcal D^1$
and $\gamma=-1$, together with \eqref{eq:sharp-sm-AP-3}, we infer that
assertion \textup{(iii)} holds.

We next show assertion \textup{(iv)} holds when $n=1$.
Fix $p\in(1,\infty)$ and $\gamma\in(-1,0)$.
Let $E$ be the self-similar set in \eqref{eq:4.3x} with $n=1$, and let $r_0$ and $\rho$ be as in \eqref{4.4x}.
For any $\delta\in(0,\min\{-\frac{\gamma}{2},1\})$, let
$\upsilon_\delta$ and $f_\delta$ be defined in
\eqref{eq:fixed-gamma-weight-test}. Define
\begin{align}\label{eq:sharp-sm-primitive}
F_\delta(x):=\int_{-\infty}^x f_\delta(y)\,dy,
\quad \forall\,x\in\mathbb R.
\end{align}
Since $f_\delta\in L^1(\mathbb R)\cap
L^p_{\upsilon_\delta}(\mathbb R)$,
its primitive $F_\delta$ satisfies $F_\delta'=f_\delta$ almost everywhere and belongs to $\dot W^{1,p}_{\upsilon_\delta}(\mathbb R)$.

Fix $t\in(0,r_0)$ and $x\in\mathcal T_t(E)$.
Choose $z\in E$ such that $|x-z|<t$, and let $I:=B(z,t)$.
By $f_\delta\geq0$, we obtain, for any
$y\in(z-t,z-\frac{t}{2})$ and
$w\in(z+\frac{t}{2},z+t)$,
\begin{align}\label{eq:F-delta-wy}
F_\delta(w)-F_\delta(y)
\geq
\int_{z-\frac{t}{2}}^{z+\frac{t}{2}}f_\delta(u)\,du.
\end{align}
Since \eqref{eq:F-delta-wy} holds, from \eqref{eq:fixed-gamma-local-negative} with $R=\frac{t}{2}$, it follows that
\begin{align}
\fint_I\fint_I|F_\delta(w)-F_\delta(y)|\,dw\,dy
&\geq
\frac{1}{|I|^2}
\int_{z-t}^{z-\frac{t}{2}}
\int_{z+\frac{t}{2}}^{z+t}
|F_\delta(w)-F_\delta(y)|\,dw\,dy\notag\\
&\geq\frac{1}{16}
\int_{z-\frac{t}{2}}^{z+\frac{t}{2}}f_\delta(u)\,du
\gtrsim
\frac{1}{\delta}t^{1+\gamma+\delta},
\label{eq:sharp-sm-4-mean-oscillation}
\end{align}
where we used $|I|=2t$, $d=1+\gamma$,
$B(z,\frac{t}{2})\subset\mathcal T_{r_0}(E)$, and $|E|=0$.
Using $x\in I$, $\operatorname{rad}(I)=t$,
\eqref{eq:sharp-sm-4-mean-oscillation}, and the definition of $\mathcal M_1^\sharp$, we find that
\begin{align}\label{eq:sharp-sm-4-pointwise}
\mathcal M_1^\sharp(F_\delta)(x,t)
\geq
\frac{1}{t}\fint_I|F_\delta-(F_\delta)_I|
\geq
\frac{c_*}{\delta}t^{\gamma+\delta},
\end{align}
where $c_*$ is a positive constant depending only on $p$
and $\gamma$.

Let
$\lambda_\delta:=\frac{c_*r_0^\delta}{2\delta}$
and
$\tau_\delta:=r_0 2^{-\frac{1}{\delta}}$.
For any $t\in(\tau_\delta,r_0)$, one has
$t^\delta>\frac{r_0^\delta}{2}$.
Thus, \eqref{eq:sharp-sm-4-pointwise} implies that
\begin{align}\label{eq:TE-inclusion-levelset}
\mathcal T_t(E)
\subset
\left\{
x\in\mathbb R:
\mathcal M_1^\sharp(F_\delta)(x,t)>
\lambda_\delta t^\gamma
\right\}.
\end{align}
From \eqref{eq:TE-inclusion-levelset} and
\eqref{eq:fixed-gamma-global-positive}, we deduce that
\begin{align*}
&\sup_{\lambda\in(0,\infty)}
\lambda^p
\underset{
\mathcal M_1^\sharp(F_\delta)(x,t)>
\lambda t^\gamma}
{\int_{\mathbb R}\int_0^\infty}
t^{\gamma p-1}\,dt\,\upsilon_\delta(x)\,dx\\
&\quad\geq
\lambda_\delta^p
\int_{\tau_\delta}^{r_0}
t^{\gamma p-1}
\int_{\mathcal T_t(E)}\upsilon_\delta(x)\,dx\,dt
\gtrsim
\delta^{-p}
\int_{\tau_\delta}^{r_0}
t^{-1-(p-1)\delta}\,dt\\
&\quad=\frac{2^{p-1}-1}{p-1}
r_0^{-(p-1)\delta}\delta^{-p-1}
\sim\delta^{-p-1},
\end{align*}
where we used $\lambda_\delta\sim\delta^{-1}$.
By \eqref{eq:sharp-sm-AP-3} with
$\upsilon=\upsilon_\delta$ and $f=F_\delta$, together with
\eqref{eq:fixed-gamma-test-norm} and \eqref{eq:fixed-gamma-Ap}, we conclude that
\begin{align}\label{eq:delta-eta-0}
\delta^{-p-1}
\lesssim
[\upsilon_\delta]_{A_p(\mathbb R)}^{\eta}
\|F_\delta'\|_{L^p_{\upsilon_\delta}(\mathbb{R})}^p
\sim
\delta^{(1-p)\eta-1},
\end{align}
where the implicit positive constants are independent of $\delta$.
Letting $\delta\to0^+$ in \eqref{eq:delta-eta-0}, we obtain
$\eta\geq\frac{p}{p-1}$, which proves \textup{(iv)} for $n=1$.

We next show that the assertions hold when $n\in\mathbb N\setminus\{1\}$.
It suffices to reduce the desired conclusion in $\mathbb R^n$ to the corresponding one-dimensional estimate.
In this part of the proof, we write $\mathcal M_{\theta,d}^{\sharp}$
for the lifted Sobolev sharp maximal function on $\mathbb R^d$,
where $d\in\{1,n\}$.

We assume that $q\in[1,\infty)$, $w\in A_q(\mathbb R)$, and
$g\in\dot W^{1,p}_w(\mathbb R)$ satisfy
\begin{align}\label{eq:sharp-n1-lim-gw}
\lim_{R\to\infty}\frac1{R^p}
\int_{-2R}^{2R}|g(s)|^pw(s)\,ds=0.
\end{align}
Choose a fixed function $\chi\in C_{\mathrm{c}}^{\infty}$
such that $0\leq\chi\leq1$, $\chi=1$ on $[-1,1]^n$, and
$\operatorname{supp}\chi\subset[-2,2]^n$.
For any $R\in(0,\infty)$ and $x=(x_1,x')\in\mathbb R \times\mathbb R^{n-1}$, let
$W(x):=w(x_1)$ and $G_R(x):=g(x_1)\chi(\frac{x}{R}).$
From the definition of $W$ and Definition \ref{df:c-v}\textup{(i)}, we infer that, for any $q\in[1,\infty)$,
\begin{align}\label{eq:sharp-sm-extension-weight}
[W]_{A_q(\mathbb R^n)}=[w]_{A_q(\mathbb R)}.
\end{align}
Moreover, $G_R(x)=g(x_1)$ on $[-R,R]^n$ and
$\operatorname{supp}G_R\subset[-2R,2R]^n$.
Using the product rule for weak derivatives, we obtain
\begin{align*}
\nabla G_R(x)
=g'(x_1)\chi\left(\frac{x}{R}\right)e_1
+\frac{1}{R}g(x_1)\nabla\chi\left(\frac{x}{R}\right),
\end{align*}
where $e_1:=(1,0,\ldots,0)$.
Since $0\leq\chi\leq1$, from the boundedness of $\nabla\chi$ and
the supports of $\chi(\frac{\cdot}{R})$ and $(\nabla\chi)(\frac{\cdot}{R})$, it follows that
\begin{align}
&\int_{\mathbb R^n}|\nabla G_R(x)|^pW(x)\,dx\notag\\
&\quad\lesssim
\int_{[-2R,2R]^n}|g'(x_1)|^pw(x_1)\,dx+
\frac{1}{R^p}\int_{[-2R,2R]^n}|g(x_1)|^pw(x_1)\,dx\notag\\
&\quad=(4R)^{n-1}
\left[
\int_{-2R}^{2R}|g'(s)|^pw(s)\,ds
+\frac{1}{R^p}\int_{-2R}^{2R}|g(s)|^pw(s)\,ds
\right]\notag\\
&\quad\lesssim R^{n-1}
\left[
\|g'\|_{L^p_w(\mathbb R)}^p
+\frac1{R^p}\int_{-2R}^{2R}|g(s)|^pw(s)\,ds
\right],
\label{eq:sharp-sm-extension-gradient}
\end{align}
where the implicit positive constant depends only on $n$ and
$p$. Fix $T\in(0,\infty)$ and $R\geq16T$.
Let $s\in(-T,T)$, $t\in(0,T)$, and
$I=B(z,r)\ni s$ with $r\in[t,2t)$.
For any $x'\in[-\frac R2,\frac R2]^{n-1}=:Q_R'$, let
$B:=B((z,x'),2r)$.
Then $(s,x')\in B\subset[-R,R]^n$ and hence
$G_R(y)=g(y_1)$ on $B$.
By $I\times B(x',r)\subset B$ and $G_R(y)=g(y_1)$ on $B$
and by Fubini's theorem, we find that
\begin{align}\label{eq:G-g}
\frac{1}{2r}\fint_B|G_R(y)-(G_R)_B|\,dy
&\geq
\frac{1}{2r|B|}
\int_I\int_{B(x',r)}
|g(u)-(G_R)_B|\,dy'\,du\notag\\
&\gtrsim
\frac{1}{r}\fint_I|g(u)-(G_R)_B|\,du\gtrsim
\frac{1}{r}\fint_I|g(u)-g_I|\,du,
\end{align}
where the last inequality follows from
$\fint_I|g-g_I|\leq2\fint_I|g-a|$ for any constant $a$.
From taking the supremum over $I$ in \eqref{eq:G-g}, we deduce that
\begin{align}
\mathcal M_{1,n}^{\sharp}(G_R)((s,x'),2t)
\geq c_n\mathcal M_{1,1}^{\sharp}(g)(s,t),
\label{eq:sharp-sm-extension-oscillation}
\end{align}
where $c_n$ is a positive constant depending only on $n$.

Fix $\lambda\in(0,\infty)$, and let
$\mu:=c_n2^{-\gamma}\lambda$.
From \eqref{eq:sharp-sm-extension-oscillation}, we infer that, if
$(s,t)\in(-T,T)\times(0,T)$ satisfies
$\mathcal M_{1,1}^{\sharp}(g)(s,t)>\lambda t^\gamma$,
then, for every $x'\in Q_R'$,
$\mathcal M_{1,n}^{\sharp}(G_R)((s,x'),2t)
>\mu(2t)^\gamma$.
Let $\tau:=2t$. Using $W(s,x')=w(s)$ and $|Q_R'|=R^{n-1}$, we conclude that
\begin{align*}
&R^{n-1}\lambda^p
\underset{\mathcal M_{1,1}^{\sharp}(g)(s,t)>
\lambda t^\gamma}
{\int_{-T}^{T}\int_0^T}
t^{\gamma p-1}\,dt\,w(s)\,ds\\
&\quad\leq
\lambda^p
\underset{\mathcal M_{1,n}^{\sharp}(G_R)(x,2t)>
\mu(2t)^\gamma}
{\int_{(-T,T)\times Q_R'}\int_0^T}
t^{\gamma p-1}\,dt\,W(x)\,dx\\
&\quad=c_n^{-p}\mu^p
\underset{\mathcal M_{1,n}^{\sharp}(G_R)(x,\tau)>
\mu\tau^\gamma}
{\int_{(-T,T)\times Q_R'}\int_0^{2T}}
\tau^{\gamma p-1}\,d\tau\,W(x)\,dx\\
&\quad\lesssim
\sup_{\nu\in(0,\infty)}\nu^p
\underset{\mathcal M_{1,n}^{\sharp}(G_R)(x,\tau)>
\nu\tau^\gamma}
{\int_{\mathbb R^n}\int_0^\infty}
\tau^{\gamma p-1}\,d\tau\,W(x)\,dx,
\end{align*}
where we used
$\lambda^p2^{-\gamma p}=c_n^{-p}\mu^p$.
From \eqref{eq:sharp-sm-AP-1} in case \textup{(i)} and \eqref{eq:sharp-sm-AP-3}
in cases \textup{(ii)}--\textup{(iv)}, both with $\upsilon=W$ and $f=G_R$,
together with \eqref{eq:sharp-sm-extension-weight} and \eqref{eq:sharp-sm-extension-gradient}, it follows that
\begin{align}\label{eq:sharp-sm-transfer-truncated}
\lambda^p
\underset{\mathcal M_{1,1}^{\sharp}(g)(s,t)>\lambda t^\gamma}
{\int_{-T}^{T}\int_0^T}
t^{\gamma p-1}\,dt\,w(s)\,ds\lesssim[w]_{A_q(\mathbb R)}^{\eta}
\left[
\|g'\|_{L^p_w(\mathbb R)}^p
+\frac1{R^p}\int_{-2R}^{2R}|g(s)|^pw(s)\,ds
\right],
\end{align}
where the implicit positive constant is independent of
$w$, $g$, $R$, $T$, and $\lambda$.
Letting $R\to\infty$ in
\eqref{eq:sharp-sm-transfer-truncated} and using
\eqref{eq:sharp-n1-lim-gw}, then letting $T\to\infty$
and taking the supremum over $\lambda\in(0,\infty)$, we find that
\begin{align}\label{eq:sharp-sm-transferred-one-dimensional}
\sup_{\lambda\in(0,\infty)}\lambda^p
\underset{\mathcal M_{1,1}^{\sharp}(g)(s,t)>\lambda t^\gamma}
{\int_{\mathbb R}\int_0^\infty}
t^{\gamma p-1}\,dt\,w(s)\,ds
\lesssim[w]_{A_q(\mathbb R)}^{\eta}\|g'\|_{L^p_w(\mathbb R)}^p.
\end{align}
 Applying \eqref{eq:sharp-sm-transferred-one-dimensional} and the proven one-dimensional condition, we obtain the desired
lower bounds for $\eta$ when $n\in\mathbb N\setminus\{1\}$.

It remains to verify \eqref{eq:sharp-n1-lim-gw}. To this end, it suffices to verify it for the following chosen
one-dimensional pair $w\in A_q(\mathbb R)$ and
$g\in\dot W^{1,p}_w(\mathbb R)$, where $q=1$ in \textup{(i)}
and $q=p$ in \textup{(ii)}--\textup{(iv)}.
For \textup{(i)} and \textup{(ii)}, we use the weight $w_\delta:=|\cdot-\frac12|^{\delta-1}$ with $\delta\in(0,1)$
and
$g\in C_{\mathrm c}^\infty(\mathbb R)$ as in the proof of
\cite[Proposition~2.11\textup{(i)}]{scm}.
Direct calculation shows that \eqref{eq:sharp-n1-lim-gw} holds for this pair $(w_{\delta},g_{\delta})$.
For \textup{(iii)}, as in the proof of
Lemma \ref{lem:sharp-weighted-cddd}\textup{(ii)}, we use
\begin{align*}
w_\delta(\cdot):=|\cdot|^{(p-1)(1-\delta)}
\ \text{and}\
g_\delta(\cdot):=\int_{-\infty}^{\cdot}u^{\delta-1}\mathbf1_{(0,1)}(u)\,du.
\end{align*}
Since $0\leq g_\delta\leq\frac1\delta$, we immediately deduce that
\begin{align*}
\frac1{R^p}\int_{-2R}^{2R}|g_\delta(s)|^pw_\delta(s)\,ds
\lesssim\delta^{-p}R^{-(p-1)\delta}\to0
\quad\text{as }R\to\infty,
\end{align*}
i.e., \eqref{eq:sharp-n1-lim-gw} holds for this pair $(w_{\delta},g_{\delta})$. For \textup{(iv)}, take $w:=\upsilon_\delta$ and
$g:=F_\delta$ as in
\eqref{eq:sharp-sm-primitive}.
Using $0\leq\upsilon_\delta\leq1$ and
$\|F_\delta\|_{L^\infty(\mathbb{R})}\leq\|f_\delta\|_{L^1(\mathbb{R})}<\infty$, we find that
\begin{align*}
\frac1{R^p}\int_{-2R}^{2R}|F_\delta(s)|^p\upsilon_\delta(s)\,ds
\lesssim\|F_\delta\|_{L^\infty(\mathbb{R})}^pR^{1-p}
\to0
\quad\text{as }R\to\infty,
\end{align*}
i.e., \eqref{eq:sharp-n1-lim-gw} holds for this pair $(w,g)$.
This completes the proof of Proposition \ref{prop:sharp-sm}.
\end{proof}

\section{Muckenhoupt Weight Characterizations}
\label{sec:weight-characterizations}

We now characterize the Muckenhoupt class $A_p$ through weak-type
estimates for $\mathcal M_\theta$ and $\mathcal M^\sharp_\theta$.
Let $p\in[1,\infty)$ and $\upsilon$ be a weight on $\mathbb R^n$.
Following \cite[Definition~2.2]{zlyyz24}, let
$\dot H^{1,p}_\upsilon$ denote the
\emph{homogeneous weighted Riesz
potential space} with seminorm
$\|f\|_{\dot H^{1,p}_\upsilon}
:=\|I_{-1}f\|_{L^p_\upsilon},$
where $I_{-1}:=(-\Delta)^{\frac12}$ and $\Delta:=\sum_{j=1}^n\frac{\partial^2}{\partial x_j^2}$.
Moreover, let
\begin{align*}
\dot Y^{1,p}_\upsilon:=
\begin{cases}
\dot W^{1,p}_\upsilon&\text{if }n=1\text{ or }p=1,\\
\dot H^{1,p}_\upsilon&\text{if }n\in\mathbb N\setminus\{1\}
\text{ and }p\in(1,\infty),
\end{cases}
\end{align*}
and write $Y^{\mathrm{av},p}_\upsilon:=L^p_\upsilon$ and
$Y^{\mathrm{osc},p}_\upsilon:=\dot Y^{1,p}_\upsilon$.
When $\upsilon\in A_p$, all the seminorms on $Y^{\tau,p}_\upsilon$
are equivalent to $\|D^\tau f\|_{L^p_\upsilon}$, while, when $\upsilon\notin A_p$,   this may not be true.

\begin{theorem}\label{th:AP-chara}
Let $\tau\in\{\mathrm{av},\mathrm{osc}\}$, $p\in[1,\infty)$,
$\gamma\in\Gamma_p^\tau$, $\theta\in(0,\infty)$, and
$\upsilon$ be a weight.
Then the following assertions are equivalent:
\begin{enumerate}[\rm (i)]
\item $\upsilon\in A_p$.
\item There exists a positive constant $c$, independent of $f$,
such that, for any $f\in Y^{\tau,p}_\upsilon\cap L^1_{\rm loc}$,
\begin{align*}
\sup_{\lambda\in(0,\infty)}\lambda^p
\underset{\mathcal M_\theta^\tau(f)(x,t)>\lambda t^\gamma}
{\int_{\mathbb R^n}\int_0^\infty}
t^{\gamma p-1}\,dt\,\upsilon(x)\,dx
\leq c\|f\|_{Y^{\tau,p}_\upsilon}^p.
\end{align*}
\end{enumerate}
\end{theorem}

\begin{proof}
Since \eqref{eq:theta-change} holds, we infer that it suffices to consider
$\theta=1$. Suppose first that $\upsilon\in A_p$.
When $\tau=\mathrm{av}$, Theorem \ref{th:intro-m}
gives \textup{(ii)} directly.
When $\tau=\mathrm{osc}$, \textup{(ii)} follows directly
from Theorem \ref{th:intro-m} if $n=1$ or $p=1$.
In the remaining case, applying \cite[Theorem~2.8 and Remark~2.9]{bui82}
(see also \cite[Lemma~2.10]{zlyyz24}), we obtain
$\|\nabla f\|_{L^p_\upsilon}
\sim\|f\|_{\dot H^{1,p}_\upsilon}.$
From this, together with Theorem \ref{th:intro-m}\textup{(ii)},
we deduce that
\textup{(ii)} holds also in the remaining case.

Conversely, assume that \textup{(ii)} holds. We first consider
$\tau:=\mathrm{av}$. Fix a ball $B$ and a bounded nonnegative function $f$
supported in $B$ with $\fint_B f(y)\,dy\in(0,\infty)$, and let
\begin{align*}
\lambda:=\frac{1}
{2\max\{1,2^{-\gamma}\}\operatorname{rad}(B)^\gamma}\fint_Bf(y)\,dy.
\end{align*}
Then, for any $x\in B$ and
$t\in(\frac{\operatorname{rad}(B)}2,\operatorname{rad}(B)]$,
we obviously have
\begin{align*}
\lambda t^\gamma\leq\frac12\fint_Bf(y)\,dy
<\fint_Bf(y)\,dy\leq\mathcal M_1(f)(x,t).
\end{align*}
Applying \textup{(ii)} to
$B\times(\frac{\operatorname{rad}(B)}{2},\operatorname{rad}(B)]$, we conclude that
\begin{align*}
\lambda^p\upsilon(B)
\int_{\frac{\operatorname{rad}(B)}2}^{\operatorname{rad}(B)}
t^{\gamma p-1}\,dt
\lesssim\int_Bf(y)^p\upsilon(y)\,dy.
\end{align*}
Using the definition of $\lambda$ and the change of
variables $t=\operatorname{rad}(B)s$, we find that
\begin{align}\label{eq:testing-inequality}
\upsilon(B)\left[\fint_Bf(y)\,dy\right]^p
\lesssim\int_Bf(y)^p\upsilon(y)\,dy,
\end{align}
where the implicit positive constant is independent of $B$ and
$f$. By truncation and the monotone convergence theorem, we extend
\eqref{eq:testing-inequality} to every
nonnegative measurable function supported in $B$.
From this and \cite[Remark~7.1.4 and Proposition~7.1.5\textup{(8)}]{gtm249}, it follows that $\upsilon\in A_p$.

We next consider $\tau:=\mathrm{osc}$. Let $m_f$ be as in Remark
\ref{re:intro-sm}\textup{(ii)}. By $\frac{m_f(x,t)}t\leq\mathcal M_1^\sharp(f)(x,t)$ and
\textup{(ii)}, we conclude that
\begin{align*}
\sup_{\lambda\in(0,\infty)}\lambda^p
\underset{m_f(x,t)>\lambda t^{1+\gamma}}
{\int_{\mathbb R^n}\int_0^\infty}
t^{\gamma p-1}\,dt\,\upsilon(x)\,dx
\lesssim\|f\|_{\dot Y^{1,p}_\upsilon}^p
\end{align*}
for any $f\in\dot Y^{1,p}_\upsilon\cap L^1_{\rm loc}$.
From \cite[Theorem~1.1]{zlyyz24} with $q=p$ and the
parameter $\gamma$ therein replaced by $p\gamma$, we deduce that $\upsilon\in A_p$. This completes the proof of Theorem \ref{th:AP-chara}.
\end{proof}

For $\tau:=\mathrm{osc}$, when $n=1$ or $p=1$, Theorem \ref{th:AP-chara}
characterizes $A_p$ by an upper estimate involving the gradient
seminorm. The next proposition proves that this conclusion fails
when $n\in\mathbb N\setminus\{1\}$ and $p\in(1,\infty)$.
\begin{proposition}\label{prop:sm-gradient-not-characterization}
Let $n\in\mathbb N\setminus\{1\}$ and $p\in(1,\infty)$.
Then there exists a weight $\upsilon\notin A_p$ such that,
for any $\gamma\in\mathbb R\setminus\{0\}$ and
$f\in\dot W^{1,p}_\upsilon$,
\begin{align}\label{eq:sm-gradient-lifted-counterexample}
\sup_{\theta,\lambda\in(0,\infty)}\lambda^p
\underset{\mathcal M_\theta^\sharp(f)(x,t)>\lambda t^\gamma}
{\int_{\mathbb R^n}\int_0^\infty}
t^{\gamma p-1}\,dt\,\upsilon(x)\,dx
\lesssim\frac{1}{|\gamma|}\|\nabla f\|_{L^p_\upsilon}^p,
\end{align}
where the implicit positive constant depends only on $n$ and
$p$. Consequently, $\dot Y^{1,p}_\upsilon$ in Theorem
\ref{th:AP-chara} with $\tau=\mathrm{osc}$ cannot be replaced
by $\dot W^{1,p}_\upsilon$.
\end{proposition}

We first establish the following auxiliary estimate.

\begin{lemma}\label{prop:sm-gradient-Ap-counterexample}
Let $n\in\mathbb N\setminus\{1\}$, $p\in(1,\infty)$, and
$\upsilon(x):=|x|^{n(p-1)}$ for any $x\in\mathbb R^n$.
Then $\upsilon\notin A_p$. Nevertheless, for any
$f\in\dot W^{1,p}_\upsilon$, the \emph{Sobolev sharp maximal function}
\begin{align*}
f_1^\sharp(x):=
\sup_{B\ni x}\frac{1}{\operatorname{rad}(B)}
\fint_B|f(y)-f_B|\,dy
\end{align*}
satisfies
\begin{align}\label{eq:sm-gradient-sharp-counterexample}
\left\|f_1^\sharp\right\|_{L^p_\upsilon}
\lesssim\|\nabla f\|_{L^p_\upsilon},
\end{align}
where the implicit positive constant depends only on $n$ and $p$.
\end{lemma}
\begin{proof}
Using the fact that $\upsilon^{-\frac{1}{p-1}}(x)=|x|^{-n}$ is not locally
integrable at the origin, we obtain $\upsilon\notin A_p$.
Fix $f\in\dot W^{1,p}_\upsilon$.

We first show that there exists a constant $a_f$ such that
\begin{align}\label{eq:sm-counterexample-Hardy}
\left\|\frac{f-a_f}{|\cdot|}\right\|_{L^p_\upsilon}
\lesssim\|\nabla f\|_{L^p_\upsilon}.
\end{align}
For any $k\in\mathbb Z$, let
$A_k:=\{x\in\mathbb R^n:2^k<|x|<2^{k+1}\},$ $A_k^*:=\{x\in\mathbb R^n:2^{k-1}<|x|<2^{k+2}\},$
and
$$e_k:=\left[\int_{A_k^*}|\nabla f(x)|^p\upsilon(x)\,dx
\right]^{\frac{1}{p}}.$$
Since $n\in\mathbb N\setminus\{1\}$, the fixed annuli
$\{x\in\mathbb R^n:\frac{1}{2}<|x|<4\}$ and
$\{x\in\mathbb R^n:1<|x|<2\}$ are connected and have smooth
boundaries.
For $N\in(0,\infty)$, let $T_Nf:=\max\{-N,\min\{f,N\}\}$.
Applying the Poincar\'e inequality
\cite[Section~5.8.1, Theorem~1]{evans10} to $T_Nf$ on $A_k$ and $A_k^*$ and then taking the limit as $N\to\infty$, we obtain
\begin{align*}
\|f-f_{A_k^*}\|_{L^p(A_k^*)}
\lesssim2^k\|\nabla f\|_{L^p(A_k^*)} \ \text{and} \
\|f-f_{A_k}\|_{L^p(A_k)}
\lesssim2^k\|\nabla f\|_{L^p(A_k)}.
\end{align*}
By $|A_k|\sim|A_{k+1}|\sim2^{kn}$ and
$\upsilon(x)\sim2^{kn(p-1)}$ on $A_k^*$ and applying H\"older's inequality, we find that
\begin{align}\label{eq:sm-counterexample-fak1-fak}
|f_{A_{k+1}}-f_{A_k}|
\leq|f_{A_{k+1}}-f_{A_k^*}|+|f_{A_k}-f_{A_k^*}|
\lesssim
2^{-\frac{kn}{p}}\|f-f_{A_k^*}\|_{L^p(A_k^*)}
\lesssim2^{-k(n-1)}e_k.
\end{align}
Moreover, since $|x|\sim2^k$ and
$\upsilon(x)\sim2^{kn(p-1)}$ on $A_k$, we infer that
\begin{align}\label{eq:Ap-sm-f-ak}
\left[\int_{A_k}\frac{|f(x)-f_{A_k}|^p}{|x|^p}
\upsilon(x)\,dx\right]^{\frac{1}{p}}
\lesssim
2^{-k}2^{\frac{kn(p-1)}{p}}
\|f-f_{A_k}\|_{L^p(A_k)}\lesssim
2^{\frac{kn(p-1)}{p}}
\|\nabla f\|_{L^p(A_k)}
\lesssim e_k.
\end{align}
Finally, using the fact that any $x\in\mathbb R^n\setminus\{\mathbf0\}$
belongs to at most three of the enlarged annuli $A_k^*$, we conclude that
\begin{align}\label{eq:sm-counterexample-energy-sum}
\sum_{k\in\mathbb Z}e_k^p
=\sum_{k\in\mathbb Z}\int_{A_k^*}|\nabla f(x)|^p\upsilon(x)\,dx\lesssim\|\nabla f\|_{L^p_\upsilon}^p.
\end{align}
For any $k\in\mathbb Z$ and $l\in\mathbb Z\cap(k,\infty)$, from \eqref{eq:sm-counterexample-fak1-fak} and H\"older's inequality, it follows that
\begin{align}\label{eq:sm-counterexample-average-difference}
|f_{A_l}-f_{A_k}|
\leq
\sum_{j=k}^{l-1}|f_{A_{j+1}}-f_{A_j}|\lesssim
\sum_{j=k}^{l-1}2^{-j(n-1)}e_j\leq
\left[
\sum_{j=k}^{l-1}2^{-j(n-1)p'}
\right]^{\frac{1}{p'}}
\left[
\sum_{j=k}^{l-1}e_j^p
\right]^{\frac{1}{p}}.
\end{align}
Consequently, by $n\in\mathbb N\setminus\{1\}$, we obtain
\begin{align*}
|f_{A_l}-f_{A_k}|
\lesssim
2^{-k(n-1)}
\left[
\sum_{j=k}^{\infty}e_j^p
\right]^{\frac{1}{p}}
\end{align*}
uniformly in $l\in\mathbb Z\cap(k,\infty)$.
Since \eqref{eq:sm-counterexample-energy-sum} holds, we deduce that the
right-hand side of the above inequality tends to zero as $k\to\infty$.
Thus, $\{f_{A_k}\}_{k\in\mathbb Z}$ is a Cauchy
sequence and hence there exists $a_f\in\mathbb R$
such that
$
a_f:=\lim_{k\to\infty}f_{A_k}.$
Letting $l\to\infty$ in
\eqref{eq:sm-counterexample-average-difference},
we find that, for any $k\in\mathbb Z$,
\begin{align}\label{eq:fAk-af}
|f_{A_k}-a_f|
\leq
\sum_{j=k}^{\infty}|f_{A_{j+1}}-f_{A_j}|\lesssim
\sum_{j=k}^{\infty}2^{-j(n-1)}e_j.
\end{align}
From Minkowski's inequality,  \eqref{eq:Ap-sm-f-ak}, and \eqref{eq:fAk-af}, we infer that
\begin{align*}
&\left[
\int_{A_k}
\frac{|f(x)-a_f|^p}{|x|^p}\upsilon(x)\,dx
\right]^{\frac{1}{p}}\notag\\
&\quad\leq
\left[
\int_{A_k}
\frac{|f(x)-f_{A_k}|^p}{|x|^p}\upsilon(x)\,dx
\right]^{\frac{1}{p}}+
|f_{A_k}-a_f|
\left[
\int_{A_k}\frac{\upsilon(x)}{|x|^p}\,dx
\right]^{\frac{1}{p}}\\
&\quad\lesssim
e_k+2^{k(n-1)}|f_{A_k}-a_f|\lesssim
e_k+
\sum_{j=k}^{\infty}
2^{-(j-k)(n-1)}e_j.
\end{align*}
Using this and applying Young's inequality for sequences to
$\sum_{j=k}^{\infty}2^{-(j-k)(n-1)}e_j$, together with
$\{2^{-j(n-1)}\}_{j\in\mathbb Z_+}\in\ell^1(\mathbb Z_+)$
and \eqref{eq:sm-counterexample-energy-sum}, we obtain
\eqref{eq:sm-counterexample-Hardy}.

We next prove \eqref{eq:sm-gradient-sharp-counterexample}.
Fix $x\in A_k$ and let $B\ni x$ with
$\operatorname{rad}(B)\in(0,\frac{|x|}{4})$. Then
$B\subset A_k^*$. Using the Poincar\'e inequality
on $B$ and using the definition of $M$, we conclude that
\begin{align*}
\frac{1}{\operatorname{rad}(B)}\fint_B|f(y)-f_B|\,dy
\lesssim
\fint_B|\nabla f(y)|\,dy\lesssim
M(|\nabla f|\mathbf1_{A_k^*})(x).
\end{align*}
Since $\upsilon(x)\sim2^{kn(p-1)}$ on $A_k^*$, from the  boundedness of $M$ on $L^p$
and from the fact that any point of $\mathbb R^n$ belongs to at most three of the annuli
$\{A_k^*\}_{k\in\mathbb Z}$, we deduce that
\begin{align}
&\int_{\mathbb R^n}
\left[
\sup_{\genfrac{}{}{0pt}{}{B\ni x}{\operatorname{rad}(B)<\frac{|x|}{4}}}
\frac{1}{\operatorname{rad}(B)}
\fint_B|f(y)-f_B|\,dy
\right]^p
\upsilon(x)\,dx\notag\\
&\quad\lesssim
\sum_{k\in\mathbb Z}2^{kn(p-1)}
\int_{\mathbb R^n}
[M(|\nabla f|\mathbf1_{A_k^*})(x)
]^p\,dx \lesssim
\sum_{k\in\mathbb Z}2^{kn(p-1)}
\int_{A_k^*}|\nabla f(x)|^p\,dx
\notag\\
&\quad\lesssim
\sum_{k\in\mathbb Z}
\int_{A_k^*}|\nabla f(x)|^p\upsilon(x)\,dx
\lesssim
\|\nabla f\|_{L^p_\upsilon}^p.
\label{eq:sm-counterexample-small-balls}
\end{align}

It remains to consider balls $B\ni x$ with
$r:=\operatorname{rad}(B)\in[\frac{|x|}{4},\infty)$.
By $B\subset B(\mathbf0,6r)$, we obtain
\begin{align*}
\frac{1}{r}\fint_B|f(y)-f_B|\,dy
\leq
\frac{2}{r}\fint_B|f(y)-a_f|\,dy\lesssim
\frac{1}{r^{n+1}}
\int_{B(\mathbf0,6r)}|f(y)-a_f|\,dy.
\end{align*}
Let $h(y):=\frac{|f(y)-a_f|}{|y|}$ for $y\in\mathbb R^n\setminus\{\mathbf0\}$
and $h(\mathbf0):=0$.
Since \eqref{eq:sm-counterexample-Hardy} holds, we infer that
$
\|h\|_{L^p_\upsilon}
\lesssim
\|\nabla f\|_{L^p_\upsilon}.$
For any $x\in\mathbb R^n\setminus\{\mathbf0\}$, define
\begin{align*}
Th(x):=
\sup_{R\geq|x|}
\frac{1}{R^{n+1}}
\int_{B(\mathbf0,R)}|y|h(y)\,dy.
\end{align*}
Using $6r\geq|x|$ and $R=6r$, we find that
\begin{align*}
\sup_{\genfrac{}{}{0pt}{}{B\ni x}
{\operatorname{rad}(B)\geq\frac{|x|}{4}}}
\frac{1}{\operatorname{rad}(B)}
\fint_B|f(y)-f_B|\,dy
\lesssim Th(x).
\end{align*}
It therefore suffices to show that
$\|Th\|_{L^p_\upsilon}\lesssim\|h\|_{L^p_\upsilon}$.

Fix $x\in\mathbb R^n\setminus\{\mathbf0\}$ and let $\rho:=|x|$.
For any $R\geq\rho$, splitting the integral at $|y|=\rho$
and using
$\frac{|y|}{R^{n+1}}\leq\frac{1}{|y|^n}$
when $\rho\leq|y|<R$, we obtain
\begin{align}\label{eq:yh-split}
\frac{1}{R^{n+1}}
\int_{|y|<R}|y|h(y)\,dy
&=
\frac{1}{R^{n+1}}
\int_{|y|<\rho}|y|h(y)\,dy
+
\int_{\rho\leq|y|<R}
\frac{|y|}{R^{n+1}}h(y)\,dy\notag\\
&\leq
\frac{1}{\rho^{n+1}}
\int_{|y|<\rho}|y|h(y)\,dy
+
\int_{|y|\geq\rho}\frac{h(y)}{|y|^n}\,dy.
\end{align}
Taking the supremum over $R\geq\rho$ on \eqref{eq:yh-split} yields
\begin{align}\label{eq:sm-counterexample-T-split}
Th(x)
\leq
\frac{1}{\rho^{n+1}}
\int_{|y|<\rho}|y|h(y)\,dy
+
\int_{|y|\geq\rho}\frac{h(y)}{|y|^n}\,dy.
\end{align}
We define, for any $s\in(0,\infty)$,
\begin{align*}
H(s):=s^n
\left[\int_{\mathbb S^{n-1}}|h(s\omega)|^p\,d\sigma(\omega)
\right]^{\frac1p},
\end{align*}
where $\sigma$ denotes the surface measure on $\mathbb S^{n-1}$.
Using $\upsilon(x)=|x|^{n(p-1)}$ and integration in polar coordinates, we find that
\begin{align}\label{eq:sm-counterexample-H-norm}
\int_0^\infty H(s)^p\,\frac{ds}{s}
=\|h\|_{L^p_\upsilon}^p.
\end{align}
Moreover, by H\"older's inequality
in \eqref{eq:sm-counterexample-T-split} and the change
of variables $s=\rho u$, we conclude that
\begin{align}
\rho^nTh(x)
&\lesssim
\frac{1}{\rho}\int_0^\rho H(s)\,ds
+\rho^n\int_\rho^\infty H(s)s^{-n-1}\,ds\notag\\
&=
\int_0^1H(\rho u)\,du
+\int_1^\infty H(\rho u)u^{-n-1}\,du.
\label{eq:sm-counterexample-T-pointwise}
\end{align}
Applying the Minkowski integral inequality to
\eqref{eq:sm-counterexample-T-pointwise} with respect to the
dilation-invariant measure $\frac{d\rho}{\rho}$, we find that
\begin{align}\label{eq:sm-counterexample-T-bound}
\|Th\|_{L^p_\upsilon}
\lesssim
\left[\int_0^1du+\int_1^\infty u^{-n-1}\,du\right]
\left[\int_0^\infty H(s)^p\,\frac{ds}{s}\right]^{\frac1p}
\lesssim\|h\|_{L^p_\upsilon},
\end{align}
where the last inequality follows from
\eqref{eq:sm-counterexample-H-norm}.
Combining \eqref{eq:sm-counterexample-T-bound} with \eqref{eq:sm-counterexample-Hardy}
and \eqref{eq:sm-counterexample-small-balls} yields
\eqref{eq:sm-gradient-sharp-counterexample}.
This completes the proof of Lemma
\ref{prop:sm-gradient-Ap-counterexample}.
\end{proof}
Now, we are ready to show Proposition \ref{prop:sm-gradient-not-characterization}.
\begin{proof}[Proof of Proposition \ref{prop:sm-gradient-not-characterization}]
By Lemma \ref{prop:sm-gradient-Ap-counterexample}, we find that
the weight $\upsilon(x):=|x|^{n(p-1)}$ for $x\in\mathbb{R}^n$ does not belong to $A_p$
and satisfies \eqref{eq:sm-gradient-sharp-counterexample}.
Fix $f\in\dot W^{1,p}_\upsilon$.
Using $\mathcal M_\theta^\sharp(f)(x,t)\leq f_1^\sharp(x)$ for any $(x,t)\in\mathbb{R}^{n+1}_{+}$, we obtain, for any $\theta,\lambda\in(0,\infty)$ and
$\gamma\in\mathbb R\setminus\{0\}$,
\begin{align}\label{eq:sm-f1-sharp}
\lambda^p
\underset{\mathcal M_\theta^\sharp(f)(x,t)>\lambda t^\gamma}
{\int_{\mathbb R^n}\int_0^\infty}
t^{\gamma p-1}\,dt\,\upsilon(x)\,dx
\leq\lambda^p\int_{\mathbb R^n}
\int_{\{t\in(0,\infty):\,f_1^\sharp(x)>\lambda t^\gamma\}}
t^{\gamma p-1}\,dt\,\upsilon(x)\,dx
=\frac{1}{p|\gamma|}\|f_1^\sharp\|_{L^p_\upsilon}^p.
\end{align}
Taking the supremum over $\theta$ and $\lambda$ in \eqref{eq:sm-f1-sharp} and applying \eqref{eq:sm-gradient-sharp-counterexample}, we find that
\eqref{eq:sm-gradient-lifted-counterexample} holds. This completes the
proof of Proposition \ref{prop:sm-gradient-not-characterization}.
\end{proof}

\section{Applications}\label{sec:applications}
In this section, we give applications to weighted Hardy(--Sobolev) spaces in Subsection \ref{sub:6-1}, to
the Caffarelli--Silvestre extension in Subsection \ref{sub:6-2}, and to ball Banach function spaces in Subsection \ref{sub:6-3}.

\subsection{Weighted Hardy Spaces via Truncated Riesz Transforms}\label{sub:6-1}
We establish a weighted Hardy-space characterization and a weighted
Hardy--Sobolev norm estimate for the optimal range
$\gamma\in\mathbb R\setminus\{0\}$.
Let $\mathcal S$ and $\mathcal S'$
denote the Schwartz space on $\mathbb R^n$ and its dual, respectively.
Let $\Phi(x):=\pi^{-\frac n2}e^{-|x|^2}$ for any
$x\in\mathbb R^n$, and let
$\Phi_t(x):=t^{-n}\Phi(\frac{x}{t})$ for any $t\in(0,\infty)$.
For any $\upsilon\in A_1$, recall that the \textit{weighted Hardy space}
$H^1_\upsilon$ consists of those $f\in\mathcal S'$ for which
\begin{align*}
\left\|f\right\|_{H^1_\upsilon}
:=\left\|\sup_{t\in(0,\infty)}|\Phi_t*f|\right\|_{L^1_\upsilon}
<\infty.
\end{align*}
For any $j\in\{1,\ldots,n\}$ and $t\in(0,\infty)$, the
\textit{truncated Riesz transform} $R_j^t$ is defined by setting, for any $x\in\mathbb R^n$,
\begin{align*}
R_j^t(f)(x)
:=\int_{|x-y|>t}\frac{x_j-y_j}{|x-y|^{n+1}}f(y)\,dy.
\end{align*}
 By \eqref{eq:c-v-1} and the definition of the Hardy--Littlewood
maximal operator, we find that there exists a constant
$c_\upsilon\in(0,\infty)$ such that
\begin{align}\label{eq:hardy-weight-lower}
\upsilon(y)\geq c_\upsilon(1+|y|)^{-n}
\end{align} for almost every $y\in\mathbb R^n$.
Consequently, the integral defining $R_j^t(f)(x)$ is absolutely
convergent whenever $f\in L^1_\upsilon$.
We write $R_j(f):=\lim_{t\to0^+}R_j^t(f)$, where the limit
exists almost everywhere for $f\in L^1_\upsilon$.
We also use the classical weighted Riesz transform
characterization (see \cite[Theorem 1]{wheeden76}):
\begin{align}\label{eq:hardy-weighted-riesz-characterization}
f\in H^1_\upsilon
\ \Longleftrightarrow\
f\in L^1_\upsilon\ \text{and}\
R_j(f)\in L^1_\upsilon\ \text{for any }j\in\{1,\ldots,n\},
\end{align}
with
\begin{align}\label{eq:hardy-weighted-riesz-norm}
\|f\|_{H^1_\upsilon}
\sim\|f\|_{L^1_\upsilon}
+\sum_{j=1}^n\|R_j(f)\|_{L^1_\upsilon},
\end{align}
where the positive equivalence constants depend only on $n$ and
$[\upsilon]_{A_1}$.
For any $\gamma\in\mathbb R$ and any weight $\upsilon$, let
$\nu_{\gamma,\upsilon}$ be the measure on $\mathbb R^{n+1}_+$
defined by
$d\nu_{\gamma,\upsilon}(x,t)
:=t^{\gamma-1}\upsilon(x)\,dx\,dt$.
Recall that, for any measurable function $G$ on
$\mathbb R^{n+1}_+$, its \emph{weak $L^1$ quasi-norm} is given by
\begin{align*}
\|G\|_{L^{1,\infty}(\mathbb R^{n+1}_+,\nu_{\gamma,\upsilon})}
:&=\sup_{\lambda\in(0,\infty)}
\lambda\,\nu_{\gamma,\upsilon}
(\{(x,t)\in\mathbb R^{n+1}_+:|G(x,t)|>\lambda\})\\
&=\sup_{\lambda\in(0,\infty)}\lambda
\underset{|G(x,t)|>\lambda}
{\int_{\mathbb R^n}\int_0^\infty}
t^{\gamma-1}\,dt\,\upsilon(x)\,dx.
\end{align*}
\begin{theorem}\label{th:weighted-hardy-riesz}
Let $\upsilon\in A_1$, $\gamma\in\mathbb R\setminus\{0\}$,
and $d\nu_{\gamma,\upsilon}(x,t):=
t^{\gamma-1}\upsilon(x)\,dx\,dt$ on $\mathbb R^{n+1}_+$.
 Then $f\in H^1_\upsilon$ if and only if $f\in L^1_\upsilon$ and
\begin{align*}
\sum_{j=1}^n
\left\|\frac{R_j^t(f)(x)}{t^\gamma}\right\|
_{L^{1,\infty}(\mathbb R^{n+1}_+,\nu_{\gamma,\upsilon})}
<\infty.
\end{align*}
Moreover, for any $f\in H^1_\upsilon$,
\begin{align}\label{eq:weighted-hardy-lifted-characterization}
\|f\|_{H^1_\upsilon}
\sim\|f\|_{L^1_\upsilon}
+\sum_{j=1}^n
\left\|\frac{R_j^t(f)(x)}{t^\gamma}\right\|
_{L^{1,\infty}(\mathbb R^{n+1}_+,\nu_{\gamma,\upsilon})},
\end{align}
where the positive equivalence constants depend only on $n$,
$\gamma$, and $[\upsilon]_{A_1}$.
\end{theorem}
\begin{theorem}\label{th:weighted-hardy-range}
Let $\upsilon\in A_1$ and $\gamma\in\mathbb R$.
Then the conclusions of Theorem \ref{th:weighted-hardy-riesz}
hold if and only if
$\gamma\in\mathbb R\setminus\{0\}$.
\end{theorem}
\begin{remark}
Theorem \ref{th:weighted-hardy-riesz} extends
\cite[Theorem~5.1]{dlyyz} to $A_1$ weights and improves its
range of $\gamma$ even when $\upsilon\equiv1$.
Theorem \ref{th:weighted-hardy-range} shows that the
range $\gamma\in\mathbb R\setminus\{0\}$ is sharp for any $A_1$ weight.
\end{remark}
Recall that a function $a$ on $\mathbb{R}^n$ is called a \textit{weighted
$(1,\infty,0)$-atom} associated with a ball $B:=B(x_B,r)$ if
\begin{align*}
\operatorname{supp}a\subset B,\
\|a\|_{L^\infty}\leq\frac1{\upsilon(B)},\ \text{and} \
\int_{\mathbb R^n}a(x)\,dx=0.
\end{align*}
\begin{proof}[Proof of Theorem \ref{th:weighted-hardy-riesz}]
Fix $j\in\{1,\ldots,n\}$. We first prove
\begin{align}\label{eq:6.5x}
\left\|\frac{R_j^t(f)(x)}{t^\gamma}\right\|
_{L^{1,\infty}(\mathbb R^{n+1}_+,\nu_{\gamma,\upsilon})}
\lesssim
\|f\|_{H^1_\upsilon}
\end{align}
using the weighted atomic decomposition.

We claim that, for any weighted $(1,\infty,0)$-atom $a$,
\begin{align}\label{eq:hardy-atom-truncation-split}
\int_{\mathbb R^n}\sup_{t\in(0,r)}|R_j^t(a)(x)|
\,\upsilon(x)\,dx
+\int_{\mathbb R^n}\int_r^\infty|R_j^t(a)(x)|
\,\frac{dt}{t}\,\upsilon(x)\,dx
\lesssim1,
\end{align}
where the implicit positive constant depends only on $n$ and
$[\upsilon]_{A_1}$.

Let $R_j^*(a):=\sup_{t\in(0,\infty)}|R_j^t(a)|$.
From the weighted $L^2$ boundedness of the maximal truncated
Riesz transform (see \cite[Theorem~1.2]{hlmorsu12} with $p=2$) and the
inclusion $A_1\subset A_2$, we deduce that
\begin{align}\label{eq:hardy-atom-near-maximal}
\int_{4B}R_j^*(a)(x)\,\upsilon(x)\,dx
\leq\upsilon(4B)^{\frac12}\|R_j^*(a)\|_{L^2_\upsilon}
\lesssim\upsilon(4B)^{\frac12}\|a\|_{L^2_\upsilon}\lesssim
\left[\frac{\upsilon(4B)}{\upsilon(B)}\right]^{\frac12}
\lesssim1.
\end{align}
Using $R_j^t(a)(x)=0$ whenever $x\in4B$ and $t\in[5r,\infty)$, we conclude that
\begin{align}\label{eq:hardy-atom-near-large-scales}
\int_{4B}\int_r^\infty|R_j^t(a)(x)|
\,\frac{dt}{t}\,\upsilon(x)\,dx
\lesssim\int_{4B}R_j^*(a)(x)\,\upsilon(x)\,dx
\lesssim1.
\end{align}

For any $u\in\mathbb R^n\setminus\{{\mathbf0}\}$, let
$K_j(u):=\frac{u_j}{|u|^{n+1}}$ denote the \emph{Riesz kernel}.
We first establish a decay estimate for $R_j(a)$ outside $2B$.
For any $x\in(2B)^{\complement}$, let $d:=|x-x_B|\geq2r$.
Since $|y-x_B|\leq r\leq \frac d2$ for any $y\in B$, from the mean value theorem, we infer that
\begin{align*}
|K_j(x-y)-K_j(x-x_B)|
\lesssim \frac{r}{d^{n+1}}.
\end{align*}
By $\int_B a(y)\,dy=0$, we obtain
\begin{align}\label{eq:hardy-atom-riesz-decay}
|R_j(a)(x)|
=\left|
\int_B [K_j(x-y)-K_j(x-x_B)]a(y)\,dy
\right|\lesssim
\frac{r}{d^{n+1}}\int_B|a(y)|\,dy
=\frac{r\|a\|_{L^1}}{d^{n+1}}.
\end{align}
We now restrict to $x\in(4B)^{\complement}$.
Observe that $R_j^t(a)(x)=R_j(a)(x)$ for any
$t\in(0,d-r)$, while $R_j^t(a)(x)=0$ for any
$t\in[d+r,\infty)$. For $t\in[d-r,d+r]$, the size
estimate for the kernel gives
$|R_j^t(a)(x)|\lesssim\frac{\|a\|_{L^1}}{d^n}.$
Since $d\geq4r$, it follows that $d-r>r$.
Splitting the integral over $(r,\infty)$ at $d-r$ and $d+r$,
using \eqref{eq:hardy-atom-riesz-decay} on $(r,d-r)$
and the kernel size estimate on $[d-r,d+r)$,
and noting that $R_j^t(a)(x)=0$ for $t\geq d+r$,
we find that
\begin{align}
\int_r^\infty|R_j^t(a)(x)|\,\frac{dt}{t}
&=\int_r^{d-r}|R_j(a)(x)|\,\frac{dt}{t}
+\int_{d-r}^{d+r}\cdots
\notag\\
&\lesssim\frac{r\|a\|_{L^1}}{d^{n+1}}
\log\frac{d-r}{r}
+\frac{\|a\|_{L^1}}{d^n}\log\frac{d+r}{d-r}
\lesssim\frac{r\|a\|_{L^1}}{d^{n+1}}
\left(1+\log\frac{d}{r}\right).
\label{eq:hardy-atom-large-scales}
\end{align}
Moreover, using $|x-y|\geq d-r>r$ for any $y\in B$, we find that $R_j^t(a)(x)=R_j(a)(x)$ for any $t\in(0,r)$.
Thus, since \eqref{eq:hardy-atom-riesz-decay} holds, we deduce that
\begin{align}\label{eq:hardy-sup-Rt-bound}
\sup_{t\in(0,r)}|R_j^t(a)(x)|
\lesssim\frac{r\|a\|_{L^1}}{d^{n+1}}.
\end{align}
By the size condition of $a$ and by the $A_1$ condition, we conclude that, for any $k\in\mathbb Z_+$,
\begin{align*}
\|a\|_{L^1}\leq\frac{|B|}{\upsilon(B)}
\ \text{and}\
\upsilon(2^kB)\leq[\upsilon]_{A_1}2^{kn}\upsilon(B).
\end{align*}
For any integer $k\geq2$, let
$A_k:=2^{k+1}B\setminus2^kB$.
Then $(4B)^{\complement}=\bigcup_{k=2}^\infty A_k$, and, for any
$x\in A_k$, one has
\begin{align*}
2^kr\leq |x-x_B|<2^{k+1}r
\ \text{and} \
1+\log\frac{|x-x_B|}{r}\lesssim1+k.
\end{align*}
From these, \eqref{eq:hardy-sup-Rt-bound}, and
\eqref{eq:hardy-atom-large-scales}, we infer that
\begin{align}
&\int_{(4B)^{\complement}}\left[
\sup_{t\in(0,r)}|R_j^t(a)(x)|
+\int_r^\infty|R_j^t(a)(x)|\,\frac{dt}{t}
\right]\upsilon(x)\,dx\notag\\
&\quad=
\sum_{k=2}^\infty
\int_{A_k}\left[
\sup_{t\in(0,r)}|R_j^t(a)(x)|
+\int_r^\infty|R_j^t(a)(x)|\,\frac{dt}{t}
\right]\upsilon(x)\,dx\notag\\
&\quad\lesssim
\|a\|_{L^1}
\sum_{k=2}^\infty
\frac{r(1+k)}{(2^kr)^{n+1}}
\upsilon(2^{k+1}B)\notag\\
&\quad\lesssim
\frac{|B|}{\upsilon(B)}
\sum_{k=2}^\infty
\frac{r(1+k)}{(2^kr)^{n+1}}
[\upsilon]_{A_1}2^{(k+1)n}\upsilon(B)
\lesssim\sum_{k=2}^\infty(1+k)2^{-k}
\lesssim1.
\label{eq:hardy-atom-far-integral}
\end{align}
Combining \eqref{eq:hardy-atom-near-maximal},
\eqref{eq:hardy-atom-near-large-scales}, and
\eqref{eq:hardy-atom-far-integral} and using
$\sup_{t\in(0,r)}|R_j^t(a)|\leq R_j^*(a)$ on $4B$, we obtain
\eqref{eq:hardy-atom-truncation-split}.

We next consider $f\in H^1_\upsilon$. Using the atomic
decomposition for $H^1_\upsilon$ (see \cite[Theorem~A]{ky11}), we find that
\begin{align}\label{eq:hardy-atomic-expansion}
f=\sum_{k=1}^\infty c_ka_k\ \text{in }\mathcal S'
\ \mathrm{and}
\sum_{k=1}^\infty|c_k|\lesssim\|f\|_{H^1_\upsilon},
\end{align}
where $a_k$ is a weighted $(1,\infty,0)$-atom associated
with a ball of radius $r_k$. Since
$\|a_k\|_{L^1_\upsilon}\leq1$ and
$\sum_{k=1}^\infty|c_k|<\infty$, it follows that the series also converges
in $L^1_\upsilon$, with the same limit $f$ by
\eqref{eq:hardy-weight-lower} and the uniqueness of
distributional limits.
Using \eqref{eq:hardy-weight-lower}, we conclude that, for any fixed
$x\in\mathbb R^n$ and $t\in(0,\infty)$, there exists a
positive constant $c_{x,t,\upsilon}$ such that
\begin{align}\label{eq:hardy-truncated-kernel-weight-control}
\mathbf1_{\{|x-y|>t\}}
\frac{|x_j-y_j|}{|x-y|^{n+1}}
\leq c_{x,t,\upsilon}\upsilon(y)
\quad\text{for almost every }y\in\mathbb R^n.
\end{align}
Consequently, the $L^1_\upsilon$ convergence in
\eqref{eq:hardy-atomic-expansion} implies that, for any $t\in(0,\infty)$ and $x\in\mathbb{R}^n$,
$$R_j^t(f)(x)=\sum_{k=1}^\infty c_kR_j^t(a_k)(x).$$

For any $x\in\mathbb R^n$ and $t\in(0,\infty)$, let
\begin{align*}
F(x):=\sum_{k=1}^\infty|c_k|
\sup_{s\in(0,r_k)}|R_j^s(a_k)(x)|\
\mathrm{and}\
G(x,t):=\sum_{k=1}^\infty|c_k|
\mathbf1_{[r_k,\infty)}(t)|R_j^t(a_k)(x)|.
\end{align*}
From Tonelli's theorem and
\eqref{eq:hardy-atom-truncation-split}, it follows that
\begin{align}
&\|F\|_{L^1_\upsilon}
+\int_{\mathbb R^n}\int_0^\infty G(x,t)
\,\frac{dt}{t}\,\upsilon(x)\,dx\notag\\
&\quad=\sum_{k=1}^\infty|c_k|
\left[
\int_{\mathbb R^n}\sup_{s\in(0,r_k)}
|R_j^s(a_k)(x)|\upsilon(x)\,dx
+\int_{\mathbb R^n}\int_{r_k}^\infty
|R_j^t(a_k)(x)|\,\frac{dt}{t}\,\upsilon(x)\,dx
\right]\notag\\
&\quad\lesssim\sum_{k=1}^\infty|c_k|
\lesssim\|f\|_{H^1_\upsilon}.
\label{eq:hardy-summed-controls}
\end{align}
Moreover, the definitions of $F$ and $G$ give
$|R_j^t(f)(x)|\leq F(x)+G(x,t)$.
Fix $\lambda\in(0,\infty)$. By the fact that $F$ is independent of $t$
and finite almost everywhere, we find that, for almost every $x\in\mathbb R^n$,
\begin{align}\label{eq:hardy-F-integral}
\lambda\underset{F(x)>\frac{\lambda}{2}t^\gamma}
{\int_0^\infty}
t^{\gamma-1}\,dt
&=\frac{\lambda}{|\gamma|}
\int_0^{\frac{2F(x)}{\lambda}}du
=\frac{2F(x)}{|\gamma|}.
\end{align}
For the term involving $G$, Chebyshev's inequality takes
the pointwise form
\begin{align}\label{eq:hardy-G-level-control}
\lambda\mathbf1_{\{G(x,t)>\frac{\lambda}{2}t^\gamma\}}
t^{\gamma-1}
\leq 2G(x,t)t^{-\gamma}t^{\gamma-1}
=\frac{2G(x,t)}{t}.
\end{align}
Since $F(x)+G(x,t)>\lambda t^\gamma$ implies $F(x)>\frac{\lambda}{2}t^\gamma$
or $G(x,t)>\frac{\lambda}{2}t^\gamma$, from \eqref{eq:hardy-F-integral}
and \eqref{eq:hardy-G-level-control}, we deduce that
\begin{align*}
&\lambda\underset{|R_j^t(f)(x)|>\lambda t^\gamma}
{\int_{\mathbb R^n}\int_0^\infty}
t^{\gamma-1}\,dt\,\upsilon(x)\,dx\\
&\quad\leq
\lambda\underset{F(x)>\frac{\lambda}{2}t^\gamma}
{\int_{\mathbb R^n}\int_0^\infty}
t^{\gamma-1}\,dt\,\upsilon(x)\,dx
+\lambda\underset{G(x,t)>\frac{\lambda}{2}t^\gamma}
{\int_{\mathbb R^n}\int_0^\infty}
t^{\gamma-1}\,dt\,\upsilon(x)\,dx\\
&\quad\leq\frac2{|\gamma|}\|F\|_{L^1_\upsilon}
+2\int_{\mathbb R^n}\int_0^\infty G(x,t)
\,\frac{dt}{t}\,\upsilon(x)\,dx\lesssim\|f\|_{H^1_\upsilon},
\end{align*}
where the last inequality follows from
\eqref{eq:hardy-summed-controls}. Taking the supremum over
$\lambda\in(0,\infty)$ shows \eqref{eq:6.5x}.

For any $f\in L^1_\upsilon$, since $R_j^t(f)\to R_j(f)$
almost everywhere as $t\to0^+$, it follows from
\cite[Theorem~9.1(ii)]{dm22} (see also \cite[Lemma~3.3]{dlyyz})
with $p=1$, $d\mu(x)=\upsilon(x)\,dx$, and $T_t=R_j^t$ that
\begin{align}\label{eq:hardy-riesz-lower-L1}
\frac1{|\gamma|}\|R_j(f)\|_{L^1_\upsilon}
\leq
\left\|\frac{R_j^t(f)(x)}{t^\gamma}\right\|
_{L^{1,\infty}(\mathbb R^{n+1}_+,\nu_{\gamma,\upsilon})}.
\end{align}
Combining \eqref{eq:hardy-weighted-riesz-characterization},
\eqref{eq:6.5x}, and \eqref{eq:hardy-riesz-lower-L1},
we obtain the asserted characterization in
Theorem \ref{th:weighted-hardy-riesz}.
Moreover, \eqref{eq:hardy-weighted-riesz-norm}, together
with \eqref{eq:6.5x} and \eqref{eq:hardy-riesz-lower-L1},
yields \eqref{eq:weighted-hardy-lifted-characterization}. This then completes the proof of
Theorem \ref{th:weighted-hardy-riesz}.
\end{proof}

The following lemma proves that the range of $\gamma$ in
Theorem \ref{th:weighted-hardy-riesz} is sharp for any fixed
$A_1$ weight.
\begin{lemma}\label{prop:weighted-hardy-gamma-zero}
Let $\upsilon\in A_1$, and let $a\in C_{\mathrm{c}}^\infty$ be
nonzero with $\int_{\mathbb R^n}a(x)\,dx=0$. Then
$a\in H^1_\upsilon$ and there exist
$j\in\{1,\ldots,n\}$ and $\lambda_0\in(0,\infty)$ such that
\begin{align}\label{eq:weighted-hardy-gamma-zero}
  \underset{|R_j^t(a)(x)|>\lambda_0}
{\int_{\mathbb R^n}\int_0^\infty}
\frac{dt}{t}\,\upsilon(x)\,dx=\infty.
\end{align}
\end{lemma}
\begin{proof}
Using the fact that $a$ is a constant multiple of a weighted $(1,\infty,0)$-atom, we conclude that $a\in H^1_\upsilon$.
From the identity $\sum_{j=1}^nR_j^2(a)=-c_na$ and $a\not\equiv0$,
we infer that $R_j(a)\not\equiv0$ for some $j\in\{1,\ldots,n\}$.
Thus, there exists $\lambda_0\in(0,\infty)$ such that
$E:=\{x\in\mathbb R^n:|R_j(a)(x)|>2\lambda_0\}$ has positive and finite measure with respect to $\upsilon(x)\,dx$.
By the fact that $R_j^t(a)\to R_j(a)$ as $t\to0^+$ almost everywhere, we find that, for almost every
$x\in E$, there exists $t_x\in(0,\infty)$ such that
$|R_j^t(a)(x)|>\lambda_0$ for any $t\in(0,t_x)$.
Therefore, for almost every $x\in E$,
\begin{align*}
\underset{|R_j^t(a)(x)|>\lambda_0}
{\int_0^\infty}\,\frac{dt}{t}
\geq\int_0^{t_x}\frac{dt}{t}=\infty.
\end{align*}
From Tonelli's theorem, it follows that \eqref{eq:weighted-hardy-gamma-zero} holds. This completes the proof of
Lemma \ref{prop:weighted-hardy-gamma-zero}.
\end{proof}
\begin{proof}[Proof of Theorem \ref{th:weighted-hardy-range}]
The sufficiency follows from Theorem \ref{th:weighted-hardy-riesz}.
By Lemma \ref{prop:weighted-hardy-gamma-zero}, we obtain a function
$a\in H^1_\upsilon$ for which the weak norm in
\eqref{eq:weighted-hardy-lifted-characterization} is infinite when
$\gamma=0$. Thus, the characterization and the norm equivalence
cannot hold in this case, which proves the necessity.
This completes the proof of Theorem \ref{th:weighted-hardy-range}.
\end{proof}
We next establish the corresponding weighted Hardy--Sobolev
estimate. To avoid technicalities, we assume that $f\in\mathcal S$.
For any distinct $x,y\in\mathbb R^n$, let
\begin{align}\label{6.19x}
\mathcal R(f)(x,y)
:=|f(x)-f(y)|
+\sum_{j=1}^n
\left|R_j^{|x-y|}(f)(x)-R_j^{|x-y|}(f)(y)\right|,
\end{align} and let $\mathcal R(f)(x,x)=0$.
\begin{theorem}\label{th:weighted-hardy-sobolev}
Let $\upsilon\in A_1$, $\gamma\in\mathbb R\setminus\{0\}$,
and $d\mu_{\gamma,\upsilon}(x,y):=
|x-y|^{\gamma-n}\upsilon(x)\,dx\,dy$ on
$\mathbb R^n\times\mathbb R^n$.
Then, for any $f\in\mathcal S$,
\begin{align}\label{eq:weighted-hardy-sobolev}
\|\nabla f\|_{H^1_\upsilon}
\sim
\left\|\frac{\mathcal R(f)(x,y)}{|x-y|^{1+\gamma}}\right\|
_{L^{1,\infty}(\mathbb R^n\times\mathbb R^n,\mu_{\gamma,\upsilon})},
\end{align}
where the positive equivalence constants depend only on $n$,
$\gamma$, and $[\upsilon]_{A_1}$.
\end{theorem}
\begin{theorem}\label{th:weighted-hardy-sobolev-range}
Let $\upsilon\in A_1$ and $\gamma\in\mathbb R$.
Then the conclusion of Theorem \ref{th:weighted-hardy-sobolev}
holds if and only if
$\gamma\in\mathbb R\setminus\{0\}$.
\end{theorem}
\begin{remark}
Theorem \ref{th:weighted-hardy-sobolev} extends
\cite[Theorem~5.3]{dlyyz} to $A_1$ weights and improves its
range of $\gamma$ from $\mathbb R\setminus[-n,0]$ to $\mathbb R\setminus\{0\}$.
Theorem \ref{th:weighted-hardy-sobolev-range}
shows that this range is sharp for any fixed $A_1$ weight.
In particular, the unweighted case of Theorem \ref{th:weighted-hardy-sobolev}
answers the question in
\cite[Remark~5.4]{dlyyz}.
\end{remark}

To prove Theorem \ref{th:weighted-hardy-sobolev}, we first
establish an auxiliary estimate for the truncated Riesz transforms.
\begin{lemma}\label{lem:hardy-truncated-potential}
Let $\upsilon\in A_1$ and $j\in\{1,\ldots,n\}$. For any
$g\in H^1_\upsilon$ and $x\in\mathbb R^n$, let
\begin{align*}
\mathcal P_j(g)(x)
:=\sup_{t\in(0,\infty)}
\frac1t\int_{B(x,2t)}
\frac{|R_j^t(g)(z)|}{|x-z|^{n-1}}\,dz.
\end{align*}
Then
\begin{align}\label{eq:hardy-truncated-potential}
\|\mathcal P_j(g)\|_{L^1_\upsilon}
\lesssim\|g\|_{H^1_\upsilon},
\end{align}
where the implicit positive constant depends only on $n$ and
$[\upsilon]_{A_1}$.
\end{lemma}

\begin{proof}
We first consider a weighted $(1,\infty,0)$-atom $a$ supported
in $B:=B(z_B,\rho)$.
For any $k\in \mathbb Z_+$, $x\in\mathbb R^n$, and $t\in(0,\infty)$, let
$A_k(x,t):=B(x,2^{1-k}t)\setminus B(x,2^{-k}t)$.
Using $|x-z|\geq2^{-k}t$ on $A_k(x,t)$, we conclude that, for
any nonnegative locally integrable function $h$,
\begin{align*}
\frac1t\int_{B(x,2t)}\frac{h(z)}{|x-z|^{n-1}}\,dz
&=\sum_{k=0}^\infty\frac1t
\int_{A_k(x,t)}\frac{h(z)}{|x-z|^{n-1}}\,dz\\
&\leq\sum_{k=0}^\infty
\frac{(2^{-k}t)^{1-n}}{t}
\int_{B(x,2^{1-k}t)}h(z)\,dz\\
&\lesssim\sum_{k=0}^\infty 2^{-k}
\fint_{B(x,2^{1-k}t)}h(z)\,dz
\lesssim M(h)(x).
\end{align*}
Recall that $R_j^*(a)$ is defined in the proof of
Theorem \ref{th:weighted-hardy-riesz}.
Applying the above estimate with $h=R_j^*(a)$, using $|R_j^t(a)|\leq R_j^*(a)$,
and taking the supremum over $t\in(0,\infty)$, we find that
$\mathcal P_j(a)\lesssim M(R_j^*(a))$.
By H\"older's inequality, the known boundedness on $L^2_\upsilon$ of both $M$ and $R_j^*$,
and $\upsilon\in A_1\subset A_2$, we obtain
\begin{align}
\int_{8B}\mathcal P_j(a)(x)\upsilon(x)\,dx
\lesssim[\upsilon(8B)]^{\frac12}
\|M(R_j^*(a))\|_{L^2_\upsilon}
\lesssim[\upsilon(8B)]^{\frac12}\|a\|_{L^2_\upsilon}
\leq\left[\frac{\upsilon(8B)}{\upsilon(B)}\right]^{\frac12}
\lesssim1,
\label{eq:hardy-potential-atom-near}
\end{align}
where we used $\|a\|_{L^2_\upsilon}^2
\leq\|a\|_{L^\infty}^2\upsilon(B)\leq[\upsilon(B)]^{-1}$
and
$\upsilon(8B)\lesssim8^n[\upsilon]_{A_1}\upsilon(B)$.

We next consider $x\in\mathbb R^n\setminus8B$. Let
$d:=|x-z_B|\geq8\rho$. For any $t\in(0,\frac d4)$,
$z\in B(x,2t)$, and $y\in B$, we obviously have
\begin{align*}
|z-y|\geq d-|z-x|-|y-z_B|
>d-2t-\rho>\frac{3d}{8}>t.
\end{align*}
Since $a$ is supported in $B$, we deduce that
$R_j^t(a)(z)=R_j(a)(z)$. Using $|z-z_B|>\frac d2\geq4\rho$ and \eqref{eq:hardy-atom-riesz-decay}, we find that
\begin{align*}
|R_j^t(a)(z)|
=|R_j(a)(z)|
\lesssim\frac{\rho\|a\|_{L^1}}{|z-z_B|^{n+1}}
\lesssim\frac{\rho\|a\|_{L^1}}{d^{n+1}},
\end{align*}
which further implies that
\begin{align}\label{eq:hardy-potential-small-t}
\frac1t\int_{B(x,2t)}
\frac{|R_j^t(a)(z)|}{|x-z|^{n-1}}\,dz
&\lesssim\frac{\rho\|a\|_{L^1}}{td^{n+1}}
\int_{B(x,2t)}\frac{dz}{|x-z|^{n-1}}
\lesssim\frac{\rho\|a\|_{L^1}}{d^{n+1}}{\leq
\frac{\rho^{\frac1n}\|a\|_{L^1}}{d^{n+\frac1n}},}
\end{align}
where the last inequality follows from
$\frac{\rho}{d}\leq1$.

Suppose now that $t\in[\frac d4,\infty)$, so that $t\geq2\rho$.
We claim that, for any $z\in\mathbb R^n$,
\begin{align}\label{eq:hardy-atom-truncation-shell}
|R_j^t(a)(z)|
\lesssim\frac{\rho\|a\|_{L^1}}{t^{n+1}}
+\frac{\|a\|_{L^1}}{t^n}
\mathbf1_{\{w\in\mathbb R^n:\,||w-z_B|-t|\leq\rho\}}(z).
\end{align}
Indeed, we distinguish three cases.
If $|z-z_B|<t-\rho$, then $B\subset B(z,t)$ and
$R_j^t(a)(z)=0$. If $|z-z_B|>t+\rho$, then
$|z-y|>t$ for any $y\in B$, and hence
$R_j^t(a)(z)=R_j(a)(z)$. Since $|z-z_B|>t+\rho\geq3\rho$, by \eqref{eq:hardy-atom-riesz-decay}, we infer that
\begin{align*}
|R_j^t(a)(z)|
\lesssim\frac{\rho\|a\|_{L^1}}{|z-z_B|^{n+1}}
\leq\frac{\rho\|a\|_{L^1}}{t^{n+1}}.
\end{align*}
Finally, if $||z-z_B|-t|\leq\rho$, then
\begin{align*}
|R_j^t(a)(z)|
\leq\int_{B\cap\{|z-y|>t\}}\frac{|a(y)|}{|z-y|^n}\,dy
\leq\frac{\|a\|_{L^1}}{t^n}.
\end{align*}
These three cases show \eqref{eq:hardy-atom-truncation-shell}.

{
For $S_t:=\{z\in\mathbb R^n:||z-z_B|-t|\leq\rho\}$,
we next prove that
\begin{align}\label{eq:hardy-thin-shell-potential}
\int_{B(x,2t)\cap S_t}\frac{dz}{|x-z|^{n-1}}
\lesssim\rho^{\frac1n}t^{1-\frac1n}.
\end{align}
By $t\geq2\rho$, we conclude that
\begin{align}\label{eq:hardy-shell-volume}
|S_t|
&=\omega_n[(t+\rho)^n-(t-\rho)^n]=n\omega_n\int_{t-\rho}^{t+\rho}s^{n-1}\,ds
\lesssim\rho t^{n-1}.
\end{align}
Let $n\geq1$ and $E\subset\mathbb R^n$ be a measurable
set with $|E|\in(0,\infty)$.
For any $x\in\mathbb R^n$, the decreasing rearrangements
on $(0,\infty)$ of $\mathbf1_E$ and
$z\mapsto|x-z|^{1-n}$ are $\mathbf1_{(0,|E|)}$ and
$s\mapsto(\frac{\omega_n}{s})^{1-\frac1n}$.
From the Hardy--Littlewood rearrangement inequality
\cite[Theorem~2.2, p.\,44]{bennett-sharpley88}, it follows that
\begin{align}\label{eq:hardy-integral-E}
\int_E\frac{dz}{|x-z|^{n-1}}
&\leq\int_0^{|E|}\left(\frac{\omega_n}{s}\right)^{1-\frac1n}\,ds
=n\omega_n^{1-\frac1n}|E|^{\frac1n}.
\end{align}
The same estimate is immediate when $|E|=0$.
Using \eqref{eq:hardy-integral-E} with
$E=B(x,2t)\cap S_t$ and using \eqref{eq:hardy-shell-volume},
we find that
\begin{align*}
\int_{B(x,2t)\cap S_t}\frac{dz}{|x-z|^{n-1}}
\lesssim|B(x,2t)\cap S_t|^{\frac1n}
\leq|S_t|^{\frac1n}
\lesssim\rho^{\frac1n}t^{1-\frac1n},
\end{align*}
which shows \eqref{eq:hardy-thin-shell-potential}.

From \eqref{eq:hardy-atom-truncation-shell} and
\eqref{eq:hardy-thin-shell-potential}, we deduce that
\begin{align}
\frac1t\int_{B(x,2t)}
\frac{|R_j^t(a)(z)|}{|x-z|^{n-1}}\,dz
&\lesssim\frac{\rho\|a\|_{L^1}}{t^{n+2}}
\int_{B(x,2t)}|x-z|^{1-n}\,dz +\frac{\|a\|_{L^1}}{t^{n+1}}
\int_{B(x,2t)\cap S_t}\cdots\notag\\
&\lesssim\|a\|_{L^1}
\left[\frac{\rho}{t^{n+1}}
+\frac{\rho^{\frac1n}}{t^{n+\frac1n}}\right]
\lesssim\frac{\rho^{\frac1n}\|a\|_{L^1}}
{t^{n+\frac1n}}
\lesssim\frac{\rho^{\frac1n}\|a\|_{L^1}}
{d^{n+\frac1n}},
\label{eq:hardy-potential-large-t}
\end{align}
where we used $\frac{\rho}{t}\leq1$ and
$t\geq\frac d4$ in the last two inequalities.
Splitting the supremum in the definition of $\mathcal P_j(a)(x)$
into the ranges $t\in(0,\frac d4)$ and
$t\in[\frac d4,\infty)$, and using
\eqref{eq:hardy-potential-small-t} and
\eqref{eq:hardy-potential-large-t}, we obtain,
for any $x\in\mathbb R^n\setminus8B$,
\begin{align}\label{eq:hardy-potential-atom-tail}
\mathcal P_j(a)(x)
\lesssim\frac{\rho^{\frac1n}\|a\|_{L^1}}
{|x-z_B|^{n+\frac1n}}
\lesssim\frac1{\upsilon(B)}
\left(\frac{\rho}{|x-z_B|}\right)^{n+\frac1n},
\end{align}
where we used $\|a\|_{L^1}\leq\frac{|B|}{\upsilon(B)}$.
By the decomposition $\mathbb R^n\setminus8B
=\bigcup_{k=3}^{\infty}(2^{k+1}B\setminus2^kB)$ and
\eqref{eq:hardy-potential-atom-tail}, we find that
\begin{align}
\int_{\mathbb R^n\setminus8B}\mathcal P_j(a)(x)\upsilon(x)\,dx
&\lesssim\sum_{k=3}^\infty
2^{-k(n+\frac1n)}
\frac{\upsilon(2^{k+1}B)}{\upsilon(B)}
\lesssim
\sum_{k=3}^\infty
2^{-k(n+\frac1n)}2^{(k+1)n}\lesssim 1,
\label{eq:hardy-potential-atom-far}
\end{align}
where we used
$\upsilon(2^{k+1}B)\lesssim
2^{(k+1)n}[\upsilon]_{A_1}\upsilon(B)$.}
From \eqref{eq:hardy-potential-atom-near} and
\eqref{eq:hardy-potential-atom-far}, we infer that
\begin{align}\label{eq:hardy-potential-atom-bound}
\|\mathcal P_j(a)\|_{L^1_\upsilon}
{
=\int_{8B}\mathcal P_j(a)(x)\upsilon(x)\,dx
+\int_{\mathbb R^n\setminus8B}\mathcal P_j(a)(x)\upsilon(x)\,dx}
\lesssim1.
\end{align}

For any $g\in H^1_\upsilon$, take an atomic decomposition
$g=\sum_{k=1}^\infty c_ka_k$ as in
\eqref{eq:hardy-atomic-expansion}, and let
$g_N:=\sum_{k=1}^Nc_ka_k$ for any $N\in\mathbb N$. As proved after
\eqref{eq:hardy-atomic-expansion}, these partial sums converge
to $g$ in $L^1_\upsilon$. Using \eqref{eq:hardy-truncated-kernel-weight-control}, we conclude that, for any fixed $z\in\mathbb R^n$ and $t\in(0,\infty)$,
\begin{align*}
\left|R_j^t(g_N)(z)-R_j^t(g)(z)\right|
&\leq\underset{|z-y|>t}
{\int_{\mathbb R^n}}
\frac{|z_j-y_j|}{|z-y|^{n+1}}|g_N(y)-g(y)|\,dy\\
&\lesssim\int_{\mathbb R^n}|g_N(y)-g(y)|\upsilon(y)\,dy
=\|g_N-g\|_{L^1_\upsilon}\to0
\end{align*}
as $N\to\infty.$ For any fixed $x\in\mathbb R^n$
and $t\in(0,\infty)$, from Fatou's lemma, it follows that
\begin{align*}
\frac1t\int_{B(x,2t)}
\frac{|R_j^t(g)(z)|}{|x-z|^{n-1}}\,dz
\leq\liminf_{N\to\infty}
\frac1t\int_{B(x,2t)}
\frac{|R_j^t(g_N)(z)|}{|x-z|^{n-1}}\,dz\leq\liminf_{N\to\infty}\mathcal P_j(g_N)(x).
\end{align*}
Taking
the supremum over $t\in(0,\infty)$, we find that
$\mathcal P_j(g)(x)\leq\liminf_{N\to\infty}\mathcal P_j(g_N)(x)$.
By Fatou's lemma, \eqref{eq:hardy-potential-atom-bound},
and \eqref{eq:hardy-atomic-expansion},
we obtain
\begin{align*}
\|\mathcal P_j(g)\|_{L^1_\upsilon}
\leq\liminf_{N\to\infty}
\|\mathcal P_j(g_N)\|_{L^1_\upsilon}\leq\sum_{k=1}^\infty|c_k|\|\mathcal P_j(a_k)\|_{L^1_\upsilon}
\lesssim\sum_{k=1}^\infty|c_k|
\lesssim\|g\|_{H^1_\upsilon}.
\end{align*}
This completes the proof of Lemma \ref{lem:hardy-truncated-potential}.
\end{proof}

\begin{proof}[Proof of Theorem \ref{th:weighted-hardy-sobolev}]
We first show the upper estimate. Let
$$\mathcal A:=\{\varphi\in C_{\mathrm{c}}^\infty(B({\mathbf0},1)):
\|\varphi\|_{L^\infty}+\|\nabla\varphi\|_{L^\infty}\leq1\}$$
and
$$\mathcal G(g)(x):=
\sup_{s\in(0,\infty), \varphi\in\mathcal A}
|\varphi_s*g(x)|,$$
where $\varphi_s(x):=s^{-n}\varphi(\frac xs)$ for any $s\in(0,\infty)$.
Let $\mathcal M_{r,1}$ be the \textit{local grand maximal function}
defined in \cite[Section~2]{koskela-saksman08}. From the definitions of $\mathcal M_{r,1}$ and $\mathcal G$, we infer that $\mathcal M_{r,1}(h)(x)\lesssim\mathcal G(h)(x)$.
Using $f\in\mathcal S$ and
\cite[Theorem~7]{koskela-saksman08} with $N=1$, we find that, for any $x,y\in\mathbb R^n$,
\begin{align}\label{eq:hardy-sobolev-grand-difference}
|f(x)-f(y)|
\lesssim|x-y|\sum_{\ell=1}^n
[\mathcal G(\partial_\ell f)(x)+\mathcal G(\partial_\ell f)(y)].
\end{align}
To control the weighted $L^1$ norms of $\{\mathcal G(\partial_\ell f)\}_{l=1}^{n}$ in \eqref{eq:hardy-sobolev-grand-difference},
we next prove that, for any $g\in H^1_\upsilon$,
\begin{align}\label{eq:hardy-weighted-grand-bound}
\|\mathcal G(g)\|_{L^1_\upsilon}
\lesssim \|g\|_{H^1_\upsilon}.
\end{align}
Indeed, let $\Theta(x,u):=\upsilon(x)u$. For any
$\varphi\in\mathcal A$, we obviously have
\begin{align*}
\sup_{\genfrac{}{}{0pt}{}{z\in\mathbb R^n}{|\alpha|\leq1}}
(1+|z|)^{2(n+1)}
\left|\partial^\alpha\left(\frac{\varphi}{C_n}\right)(z)\right|
\leq1,
\end{align*}
where $C_n:=2^{2(n+1)}$. Thus, $\varphi/C_n$ belongs to the
test-function family defining the grand maximal function $g_0^*$
in \cite[Definition~2.2]{ky14}, and hence
$\mathcal G(g)(x)\leq C_n g_0^*(x)$.
Since $\upsilon\in A_1$, the triplet $(\Theta,\infty,0)$ is
admissible in the sense of \cite[Definition~2.4]{ky14}.
Moreover, $L^\Theta=L^1_\upsilon$ with equal norms, and
$\|\mathbf 1_B\|_{L^\Theta}=\upsilon(B)$. Hence, the
$(\Theta,\infty,0)$-atoms in \cite[Definition~2.4]{ky14}
are precisely the weighted $(1,\infty,0)$-atoms used here.
By the weighted atomic decomposition
\cite[Theorem~A]{ky11}, choose such atoms $a_k$ and coefficients
$c_k$ so that
\begin{align*}
g=\sum_{k=1}^\infty c_ka_k\ \text{in }\mathcal S'
\ \mathrm{and}\
\sum_{k=1}^\infty|c_k|
\lesssim\|g\|_{H^1_\upsilon}.
\end{align*}
From \cite[Theorem~3.1]{ky14}, it follows that
\begin{align*}
\|\mathcal G(g)\|_{L^1_\upsilon}
\lesssim \|g_0^*\|_{L^\Theta}
\lesssim \sum_{k=1}^\infty|c_k|
\lesssim \|g\|_{H^1_\upsilon}.
\end{align*}
This shows \eqref{eq:hardy-weighted-grand-bound}.

For any distinct $x,y\in\mathbb R^n$, let $t:=|x-y|$. Applying \cite[Lemma~7.16]{gilbarg-trudinger01}
to the function $R_j^t(f)$ on the ball
$B(\frac{x+y}{2},t)$ at $x$ and $y$ and
using $\partial_\ell R_j^t(f)=R_j^t(\partial_\ell f)$,
we conclude that
\begin{align}
\frac{|R_j^t(f)(x)-R_j^t(f)(y)|}{t}
&\lesssim\frac1t\int_{B(\frac{x+y}{2},t)}
\frac{|R_j^t(\nabla f)(z)|}{|x-z|^{n-1}}\,dz
+\frac1t\int_{B(\frac{x+y}{2},t)}
\frac{|R_j^t(\nabla f)(z)|}{|y-z|^{n-1}}\,dz\notag\\
&\leq\frac1t\int_{B(x,2t)}
\frac{|R_j^t(\nabla f)(z)|}{|x-z|^{n-1}}\,dz
+\frac1t\int_{B(y,2t)}
\frac{|R_j^t(\nabla f)(z)|}{|y-z|^{n-1}}\,dz\notag\\
&\leq\sum_{\ell=1}^n
[\mathcal P_j(\partial_\ell f)(x)+\mathcal P_j(\partial_\ell f)(y)],
\label{eq:hardy-sobolev-riesz-difference}
\end{align}
where the last inequality follows from the definition of
$\mathcal P_j$.
Let
\begin{align*}
H:=c\sum_{\ell=1}^n
\left[\mathcal G(\partial_\ell f)
+\sum_{j=1}^n\mathcal P_j(\partial_\ell f)\right],
\end{align*}
where $c\in(0,\infty)$ is a  constant.
Since \eqref{eq:hardy-sobolev-riesz-difference} and
\eqref{eq:hardy-sobolev-grand-difference} hold, we deduce that
\begin{align}\label{eq:hardy-sobolev-pointwise-control}
\mathcal R(f)(x,y)\leq|x-y|[H(x)+H(y)].
\end{align}
Using \eqref{eq:hardy-weighted-grand-bound}
and \eqref{eq:hardy-truncated-potential}, we obtain
\begin{align}\label{eq:hardy-sobolev-H-leq-Hv}
\|H\|_{L^1_\upsilon}\lesssim\|\nabla f\|_{H^1_\upsilon}.
\end{align}
Next, by \eqref{eq:hardy-sobolev-pointwise-control}, to prove the
upper estimate in \eqref{eq:weighted-hardy-sobolev}, it suffices
to estimate the left-most hand sides of
\eqref{eq:hardy-sobolev-first-endpoint} and
\eqref{eq:hardy-sobolev-second-endpoint}.
For any $\lambda\in(0,\infty)$, from the changes of variables $y=x+r\omega$ and $u=r^\gamma$, we infer that
\begin{align}
&\lambda\iint_{H(x)>\lambda|x-y|^\gamma}
|x-y|^{\gamma-n}\upsilon(x)\,dx\,dy\notag\\
&\quad=\lambda{n\omega_n}
\underset{H(x)>\lambda r^\gamma}
{\int_{\mathbb R^n}\int_0^\infty}
r^{\gamma-1}\,dr\,\upsilon(x)\,dx
=\frac{{n\omega_n}}{|\gamma|}
\|H\|_{L^1_\upsilon}.
\label{eq:hardy-sobolev-first-endpoint}
\end{align}
To estimate the term involving $H(y)$, recall that, for almost
every $y\in\mathbb R^n$ and any $s\in(0,\infty)$,
\begin{align}\label{eq:hardy-Vys-upsilon-y}
\upsilon(B(y,s))
\lesssim[\upsilon]_{A_1}s^n\upsilon(y).
\end{align}
For $y$ satisfying \eqref{eq:hardy-Vys-upsilon-y} and
$H(y)\in(0,\infty)$, let
$R:=[\frac{H(y)}\lambda]^{\frac1\gamma}$.
If $\gamma\in(0,\infty)$, by the decomposition
$$
B(y,R)\setminus\{y\}
=
\bigcup_{k=0}^{\infty}
\{x\in\mathbb R^n:
2^{-k-1}R\leq|x-y|<2^{-k}R\}$$
and \eqref{eq:hardy-Vys-upsilon-y}, we find that
\begin{align}
\int_{|x-y|<R}|x-y|^{\gamma-n}\upsilon(x)\,dx
&=\sum_{k=0}^\infty
\int_{2^{-k-1}R\leq|x-y|<2^{-k}R}
|x-y|^{\gamma-n}\upsilon(x)\,dx\notag\\
&\lesssim\sum_{k=0}^\infty
(2^{-k}R)^{\gamma-n}\upsilon\left(B\left(y,2^{-k}R\right)\right)
\lesssim[\upsilon]_{A_1}R^\gamma\upsilon(y)
\sum_{k=0}^\infty2^{-k\gamma}.
\label{eq:hardy-xyR-0n}
\end{align}
Similarly,
if $\gamma\in(-\infty,0)$, from the dyadic decomposition of $[B(y,R)]^{\complement}$
and \eqref{eq:hardy-Vys-upsilon-y}, it follows that
\begin{align}
\int_{|x-y|>R}|x-y|^{\gamma-n}\upsilon(x)\,dx
\lesssim[\upsilon]_{A_1}R^\gamma\upsilon(y)
\sum_{k=0}^\infty2^{k\gamma}.
\label{eq:hardy-xyR-negative}
\end{align}
Using \eqref{eq:hardy-xyR-0n}, \eqref{eq:hardy-xyR-negative},
$\lambda R^\gamma=H(y)$,
$\sum_{k=0}^\infty2^{-k|\gamma|}\lesssim1+\frac1{|\gamma|}$,
and Tonelli's theorem, we conclude that
\begin{align}\label{eq:hardy-sobolev-second-endpoint}
\lambda\iint_{H(y)>\lambda|x-y|^\gamma}
|x-y|^{\gamma-n}\upsilon(x)\,dx\,dy
\lesssim[\upsilon]_{A_1}
\left(1+\frac1{|\gamma|}\right)\|H\|_{L^1_\upsilon}.
\end{align}
Combining \eqref{eq:hardy-sobolev-pointwise-control},
\eqref{eq:hardy-sobolev-first-endpoint}, and
\eqref{eq:hardy-sobolev-second-endpoint}, with
$\lambda$ replaced by $\frac\lambda2$ in the last two
estimates, and then using \eqref{eq:hardy-sobolev-H-leq-Hv},
we conclude that the upper estimate in \eqref{eq:weighted-hardy-sobolev} holds.

From the pointwise limit for the truncated Riesz differences established in
\cite[p.\,51]{dlyyz}, the argument used to show \cite[(5.20)]{dlyyz}
with $dx$ replaced by $\upsilon(x)\,dx$, and \eqref{eq:hardy-weighted-riesz-norm}, we deduce that the lower
estimate in \eqref{eq:weighted-hardy-sobolev} holds.
This completes the proof of Theorem
\ref{th:weighted-hardy-sobolev}.
\end{proof}
The next lemma proves that the assumption $\gamma\ne0$
in Theorem \ref{th:weighted-hardy-sobolev} is necessary even
for any $A_1$ weight.
\begin{lemma}\label{prop:weighted-hardy-sobolev-gamma-zero}
Let $\upsilon\in A_1$ and
$d\mu_{\upsilon}(x,y):=|x-y|^{-n}\upsilon(x)\,dx\,dy$.
Then, for any nonzero $f\in C_{\mathrm{c}}^\infty$,
$\|\nabla f\|_{H^1_\upsilon}<\infty$, but
\begin{align}\label{6.41x}
\left\|\frac{\mathcal R(f)(x,y)}{|x-y|}\right\|
_{L^{1,\infty}(\mathbb R^n\times\mathbb R^n,\mu_{\upsilon})}
=\infty.
\end{align}
\end{lemma}
\begin{proof}
For any $\ell\in\{1,\ldots,n\}$, the function
$\partial_\ell f$ belongs to $C_{\rm c}^\infty$ and has integral zero. Thus, whenever it is nonzero, it is a constant multiple of a
weighted $(1,\infty,0)$-atom, and hence
$\|\nabla f\|_{H^1_\upsilon}<\infty$.
By $f\not\equiv0$ and $f\in C_{\mathrm{c}}^\infty$, we conclude that
$\nabla f\not\equiv0$. Using the continuity  from $\mathbb R^n\times  \mathbb S^{n-1}$ to $\mathbb R$ of the map
$(x,\omega)\mapsto|\nabla f(x)\cdot\omega|$, we may therefore choose
$\lambda_0\in(0,\infty)$ such that
$E:=\{(x,\omega)\in\mathbb R^n\times\mathbb S^{n-1}:
|\nabla f(x)\cdot\omega|>2\lambda_0\}$
has positive measure with respect to
$\upsilon(x)\,dx\,d\sigma(\omega)$.
For any $(x,\omega)\in E$, by the mean value theorem, we have
\begin{align*}
\lim_{t\to0^+}\frac{|f(x+t\omega)-f(x)|}t
=|\nabla f(x)\cdot\omega|.
\end{align*}
Since $\mathcal R(f)(x,x+t\omega)\geq|f(x+t\omega)-f(x)|$ due to the definition \eqref{6.19x}, we infer that there exists $t_{x,\omega}\in(0,\infty)$ such that
$\mathcal R(f)(x,x+t\omega)>\lambda_0t$ whenever
$t\in(0,t_{x,\omega})$ for any $(x,\omega)\in E$.
Consequently,
\begin{align*}
\underset{\mathcal R(f)(x,x+t\omega)>\lambda_0t}
{\int_0^\infty}
\frac{dt}{t}
\geq\int_0^{t_{x,\omega}}\frac{dt}{t}
=\infty.
\end{align*}
Using Tonelli's theorem, we find that
\begin{align*}
\iint_{\frac{\mathcal R(f)(x,y)}{|x-y|}>\lambda_0}
\frac{\upsilon(x)}{|x-y|^n}\,dx\,dy
&\geq\underset{\mathcal R(f)(x,x+t\omega)>\lambda_0t}
{\int_E\int_0^\infty}
\frac{dt}t\,\upsilon(x)\,dx\,d\sigma(\omega)
=\infty,
\end{align*}
which further implies \eqref{6.41x}. This completes the proof of Lemma \ref{prop:weighted-hardy-sobolev-gamma-zero}.
\end{proof}
\begin{proof}[Proof of Theorem \ref{th:weighted-hardy-sobolev-range}]
Theorem \ref{th:weighted-hardy-sobolev} gives the sufficiency.
By Lemma \ref{prop:weighted-hardy-sobolev-gamma-zero}, we find that,
when $\gamma=0$, any nonzero $f\in C_{\mathrm{c}}^\infty$ has a finite
left-hand side and an infinite right-hand side in
\eqref{eq:weighted-hardy-sobolev}, which proves the necessity.
This completes the proof of Theorem \ref{th:weighted-hardy-sobolev-range}.
\end{proof}
\subsection{The Caffarelli--Silvestre Extension}\label{sub:6-2}
We next apply Theorem \ref{th:intro-m} with $\tau=\mathrm{av}$
to obtain a weighted norm equivalence for the Caffarelli--Silvestre extension,
extending \cite[Theorem~3.2 and Example~3.6]{dlyyz}.
The Ahlfors regular fractal sets constructed in Subsection \ref{subsec:sharp-m}
are then used to find those counterexamples for the excluded negative parameters
at $p=1$.

Let $s\in(0,1)$. For any $t\in(0,\infty)$, define the \textit{fractional
Poisson kernel} $p_{s,t}$ and the associated \textit{fractional Poisson operator} $P_{s,t}$ by setting, for any $x\in\mathbb R^n$,
\begin{align*}
p_{s,t}(x)
&:=
c_{n,s}\frac{t^{2s}}
{(t^2+|x|^2)^{\frac{n+2s}{2}}}
\ \text{and}\ P_{s,t}(f)(x)
:=
(p_{s,t}*f)(x),
\end{align*}
where
$
c_{n,s}
:=
\frac{\Gamma(\frac{n+2s}{2})}
{\pi^{\frac n2}\Gamma(s)}.$
For $f\in\mathcal S$, the function
$
(x,t)\mapsto P_{s,t}(f)(x)$
gives the Caffarelli--Silvestre extension from $\mathbb R^{n+1}_{+}$ to $\mathbb R$ of $f$ of order $s$; see
\cite[(1.4)--(1.6) and Section~2.4]{cs07}, particularly (2.3) therein.

We now establish the following weighted estimates.

\begin{theorem}\label{th:weighted-CS}
Let $s\in(0,1)$, $p\in[1,\infty)$, and $\upsilon\in A_p$.
Assume that $\gamma\in\mathbb R\setminus[-n,0]$ when $p=1$
and $\gamma\in\mathbb R\setminus\{0\}$ when $p\in(1,\infty)$.
Then, for any $f\in\mathcal S$ satisfying
$(-\Delta)^sf\in L^p_\upsilon$,
\begin{align}\label{eq:weighted-CS-Ap}
\sup_{\lambda\in(0,\infty)}
\lambda^p
\underset{
|P_{s,t}(f)(x)-f(x)|>
\lambda t^{2s+\gamma}}
{\int_{\mathbb R^n}\int_0^\infty}
t^{\gamma p-1}\,dt\,\upsilon(x)\,dx
\sim
\|(-\Delta)^sf\|_{L^p_\upsilon}^p,
\end{align}
where the positive equivalence constants depend only on
$n$, $p$, $s$, $\gamma$, and $[\upsilon]_{A_p}$.
\end{theorem}
\begin{theorem}\label{th:weighted-CS-range}
Let $s\in(0,1)$, $p\in[1,\infty)$, and $\gamma\in\mathbb R$.
Then the conclusion of Theorem \ref{th:weighted-CS}
holds for any $\upsilon\in A_p$ if and only if
$\gamma\in\Gamma_p^{\mathrm{av}}$.
\end{theorem}
\begin{remark}
Theorem \ref{th:weighted-CS} extends the fractional Poisson
characterization in \cite[Theorem~3.2 and Example~3.6]{dlyyz}
to Muckenhoupt weights, with the parameter therein replaced by
$p\gamma$. Theorem \ref{th:weighted-CS-range}
shows that the range of
$\gamma$ in Theorem \ref{th:weighted-CS} is sharp for any $p\in[1,\infty)$.
\end{remark}

\begin{proof}[Proof of Theorem \ref{th:weighted-CS}]
Let $g:=(-\Delta)^sf$. By the proof of
\cite[Example~3.6]{dlyyz},
there exists $c_s\in(0,\infty)$, depending only on $s$, such that
\begin{align}\label{eq:CS-boundary-limit}
\lim_{t\to0^+}\frac{f(x)-P_{s,t}(f)(x)}{t^{2s}}
=c_s g(x),\quad\forall\, x\in\mathbb R^n.
\end{align}
From this and \cite[Lemma~3.3]{dlyyz} with
$d\mu(x)=\upsilon(x)\,dx$ and $\gamma$ therein replaced by
$p\gamma$, it follows that the lower estimate in
\eqref{eq:weighted-CS-Ap} holds.

We next show that the upper estimate holds. By the pointwise estimate in the proof of \cite[Example~3.6]{dlyyz},
the Aoki--Rolewicz theorem (see, for example, \cite{aoki42,rolewicz57}),
and the geometric decay of the coefficients, we conclude that
\begin{align*}
\left\|\frac{P_{s,t}(f)(x)-f(x)}{t^{2s+\gamma}}\right\|
_{L^{p,\infty}(\mathbb R^{n+1}_+,\nu_{\gamma p,\upsilon})}
\lesssim
\sup_{\theta\in(0,\infty)}
\left\|\frac{\mathcal M_\theta(g)(x,t)}{t^\gamma}\right\|
_{L^{p,\infty}(\mathbb R^{n+1}_+,\nu_{\gamma p,\upsilon})}
\lesssim\|g\|_{L^p_\upsilon},
\end{align*}
where the last inequality follows from \eqref{eq:intro-m-1} when
$p=1$ and \eqref{eq:intro-m-p} when $p\in(1,\infty)$.
This proves the upper estimate in \eqref{eq:weighted-CS-Ap}
and hence completes the proof of Theorem \ref{th:weighted-CS}.
\end{proof}
To prove Theorem \ref{th:weighted-CS-range}, we first exclude
$\gamma=0$ for any $p\in[1,\infty)$, even in the unweighted case.

\begin{lemma}\label{prop:CS-gamma-zero}
Let $s\in(0,1)$, $p\in[1,\infty)$, and $f\in\mathcal S$ satisfy
$(-\Delta)^sf\not\equiv0$. Then there exists $\lambda_0\in(0,\infty)$ such that,
for any $\lambda\in(0,\lambda_0)$,
\begin{align*}
\lambda^p
\underset{
|P_{s,t}(f)(x)-f(x)|>\lambda t^{2s}}
{\int_{\mathbb R^n}\int_0^\infty}
\frac{dt}{t}\,dx
=\infty.
\end{align*}
\end{lemma}
\begin{proof}
Let $g:=(-\Delta)^sf$. Since $g$ is continuous and $g\not\equiv0$, we deduce that there exist a ball $B\subset\mathbb R^n$
and $\lambda_0\in(0,\infty)$ such that
$c_s|g(x)|>2\lambda_0$ for any $x\in B$, where $c_s$ is as in
\eqref{eq:CS-boundary-limit}.
Using \eqref{eq:CS-boundary-limit}, we obtain, for any
$x\in B$ and $\lambda\in(0,\lambda_0)$,
\begin{align*}
\underset{|P_{s,t}(f)(x)-f(x)|>\lambda t^{2s}}
{\int_0^\infty}
\frac{dt}{t}=\infty.
\end{align*}
From Tonelli's theorem, we infer that the desired conclusion holds.
This completes the proof of Lemma \ref{prop:CS-gamma-zero}.
\end{proof}

Lemma \ref{prop:CS-gamma-zero} settles the optimality
when $p\in(1,\infty)$. At the endpoint case $p=1$, it remains
to exclude $\gamma\in[-n,0)$. The following lemma
provides a counterexample for any such $\gamma$.
\begin{lemma}\label{prop:CS-weighted-negative-gamma}
Let $s\in(0,1)$, $\gamma\in[-n,0)$, and
$\upsilon(x):=|x|^{-\frac n2}$.
Then $\upsilon\in A_1$ and there exist a positive constant
$\lambda_0$, depending only on $n$, $s$, and $\gamma$, and a sequence
$\{f_k\}_{k\in\mathbb N}\subset\mathcal S$ such that
\begin{align}\label{eq:CS-ce-bounded-data}
\sup_{k\in\mathbb N}
\|(-\Delta)^sf_k\|_{L^1_\upsilon}<\infty
\end{align}
and
\begin{align}\label{eq:CS-ce-divergence}
\lim_{k\to\infty}
\lambda_0
\underset{
|P_{s,t}(f_k)(x)-f_k(x)|>\lambda_0t^{2s+\gamma}}
{\int_{\mathbb R^n}\int_0^\infty}
t^{\gamma-1}\,dt\,\upsilon(x)\,dx
=\infty.
\end{align}
\end{lemma}
\begin{proof}
By \cite[Lemma~1.4]{b93} with $\delta:=\frac12$, we find that
$\upsilon\in A_1$.

Let $d:=n+\gamma$.
Fix $a\in\mathbb R^n$ with $|a|=3$.
When $d\in(0,n)$, let $E\subset[0,1]^n$ denote the compact
set constructed in \eqref{eq:4.3x}, and define
\begin{align*}
E_1:=a+\frac{1}{8\sqrt n}E\subset B\left(a,\frac14\right)
\ \text{and} \
\mu(A):= \frac{\mathcal H^d(A\cap E_1)}{\mathcal H^d(E_1)}
\end{align*}
for any Borel set $A\subset\mathbb R^n$.
When $d=0$, let $E_1:=\{a\}$ and $\mu:=\delta_a$, where $\delta_a$ denotes the Dirac measure at $a$.
Fix $r_0\in(0,\frac14)$ sufficiently small such that
\eqref{eq:fixed-gamma-ahlfors} and
\eqref{eq:fixed-gamma-global-tube} hold for $E_1$ by translation
and scaling when $d\in(0,n)$, while, when $d=0$, the corresponding estimates
follow directly from $E_1=\{a\}$ and $\mu=\delta_a$.
Since $\mathcal T_{r_0}(E_1)\subset\{x\in\mathbb R^n:2<|x|<4\}$, it follows that
$\upsilon\sim1$ on $\mathcal T_{r_0}(E_1)$.
Thus, using \eqref{eq:fixed-gamma-global-tube}, we conclude that
\begin{align}\label{eq:CS-ce-weighted-tube}
\int_{\mathcal T_t(E_1)}\upsilon(x)\,dx
\sim|\mathcal T_t(E_1)|
\sim t^{-\gamma},
\quad\forall\, t\in(0,r_0).
\end{align}
Choose a nonnegative function
$\eta\in C_{\mathrm{c}}^\infty$ such that $ \operatorname{supp} \eta\subset B(\mathbf0,1)$ and
$\int_{\mathbb R^n}\eta(x)\,dx=1$.
For any $\varepsilon\in(0,\frac{r_0}{8})$, let
\begin{align*}
\eta_\varepsilon(x)
:=\varepsilon^{-n}\eta\left(\frac{x}{\varepsilon}\right)
\ \text{and}\
h_\varepsilon:=\mu*\eta_\varepsilon.
\end{align*}
Then $h_\varepsilon\in C_{\mathrm{c}}^\infty$,
$h_\varepsilon\geq0$,
$\operatorname{supp}h_\varepsilon\subset\mathcal T_{r_0}(E_1)$,
and $\int_{\mathbb R^n}h_\varepsilon(x)\,dx=1$.
Since $\upsilon\sim1$ on $\mathcal T_{r_0}(E_1)$, we deduce that
\begin{align}\label{eq:CS-ce-h-norm}
\|h_\varepsilon\|_{L^1_\upsilon}\lesssim1.
\end{align}
Fix $t\in(4\varepsilon,r_0)$ and
$x\in\mathcal T_t(E_1)$, and choose $z\in E_1$ such that
$|x-z|<t$.
For any $u\in E_1\cap B(z,\frac t2)$, one has
$B(u,\varepsilon)\subset B(z,t)$ and hence
$\int_{B(z,t)}\eta_\varepsilon(y-u)\,dy=1$.
By Tonelli's theorem and \eqref{eq:fixed-gamma-ahlfors}, we obtain
\begin{align}\label{eq:CS-ce-positive-part-ball}
\int_{B(z,t)}h_\varepsilon(y)\,dy
=\int_{E_1}\left[
\int_{B(z,t)}\eta_\varepsilon(y-u)\,dy
\right]d\mu(u)
\geq\mu\left(B\left(z,\frac t2\right)\right)
\gtrsim t^d.
\end{align}
From \eqref{eq:CS-ce-positive-part-ball},  we infer that,
for any $u\in[t^2,2t^2]$,
\begin{align}\label{eq:CS-ce-positive-part}
e^{u\Delta}h_\varepsilon(x)
&=\frac{1}{(4\pi u)^{\frac n2}}
\int_{\mathbb R^n}e^{-\frac{|x-y|^2}{4u}}h_\varepsilon(y)\,dy\notag\\
&\geq\frac{1}{(4\pi u)^{\frac n2}}
\int_{B(z,t)}e^{-\frac{|x-y|^2}{4u}}h_\varepsilon(y)\,dy
\gtrsim t^{-n}\int_{B(z,t)}h_\varepsilon(y)\,dy
\gtrsim t^{\gamma}.
\end{align}

Choose a fixed real-valued even
function $\chi\in C_{\mathrm{c}}^\infty$ which equals $1$ on
$B(\mathbf0,1)$, and let $\phi:=\mathcal F^{-1}\chi$.
Define $f_\varepsilon$ by setting, for any $\xi\in\mathbb{R}^n$,
\begin{align*}
\widehat f_\varepsilon(\xi)
:=|\xi|^{-2s}[1-\chi(\xi)]\widehat h_\varepsilon(\xi).
\end{align*}
Using the fact that $h_\varepsilon\in C_{\mathrm{c}}^\infty$ and $1-\chi$
vanishes near the origin, we find that
$f_\varepsilon\in\mathcal S$ and
\begin{align*}
g_\varepsilon:=(-\Delta)^sf_\varepsilon
=h_\varepsilon-\phi*h_\varepsilon.
\end{align*}
Since $\phi\in\mathcal S$, $\operatorname{supp}h_\varepsilon\subset B(\mathbf0,4)$,
and $\|h_\varepsilon\|_{L^1}=1$, it follows that, for any $x\in \mathbb R^n$,
\begin{align}\label{eq:CS-ce-correction-decay}
|\phi*h_\varepsilon(x)|\lesssim(1+|x|)^{-n-1}.
\end{align}
By the integrability of $(1+|x|)^{-n-1}|x|^{-\frac n2}$ on $\mathbb R^n$,
\eqref{eq:CS-ce-h-norm}, and \eqref{eq:CS-ce-correction-decay}, we conclude that
\begin{align}\label{eq:CS-ce-f-norm}
\|(-\Delta)^sf_\varepsilon\|_{L^1_\upsilon}
\leq\|h_\varepsilon\|_{L^1_\upsilon}
+\|\phi*h_\varepsilon\|_{L^1_\upsilon}\lesssim1,
\end{align}
where the implicit positive constant is independent of $\varepsilon$.
From \cite[(3.13)]{dlyyz}, we deduce that
\begin{align}\label{eq:CS-ce-heat-representation}
\frac{f_\varepsilon(x)-P_{s,t}(f_\varepsilon)(x)}{t^{2s}}
=\frac{1}{\Gamma(s)t^{2s}}
\int_0^\infty e^{u\Delta}g_\varepsilon(x)
\left(1-e^{-\frac{t^2}{4u}}\right)\,\frac{du}{u^{1-s}}.
\end{align}
Using \eqref{eq:CS-ce-positive-part} and
$e^{u\Delta}h_\varepsilon\geq0$, we obtain
\begin{align*}
&\frac{1}{\Gamma(s)t^{2s}}
\int_0^\infty e^{u\Delta}h_\varepsilon(x)
\left(1-e^{-\frac{t^2}{4u}}\right)\,\frac{du}{u^{1-s}}\\
&\quad\geq\frac{1}{\Gamma(s)t^{2s}}
\int_{t^2}^{2t^2}e^{u\Delta}h_\varepsilon(x)
\left(1-e^{-\frac{t^2}{4u}}\right)\,\frac{du}{u^{1-s}}
\geq c_0t^{\gamma},
\end{align*}
where $c_0$ is a positive constant independent of $\varepsilon$.
Since $\|e^{u\Delta}(\phi*h_\varepsilon)\|_{L^\infty}\leq\|\phi\|_{L^\infty}$, it follows from the change of variables $u=t^2r$ that
\begin{align*}
&\left|\frac{1}{\Gamma(s)t^{2s}}
\int_0^\infty e^{u\Delta}(\phi*h_\varepsilon)(x)
\left(1-e^{-\frac{t^2}{4u}}\right)\,\frac{du}{u^{1-s}}\right|\\
&\quad\leq\frac{\|\phi\|_{L^\infty}}{\Gamma(s)}
\int_0^\infty\left(1-e^{-\frac{1}{4r}}\right)\,
\frac{dr}{r^{1-s}}=:c_1<\infty.
\end{align*}
By $g_\varepsilon=h_\varepsilon-\phi*h_\varepsilon$, the choice of $r_0$
with $c_1\leq\frac{c_0}{2}r_0^{\gamma}$, and \eqref{eq:CS-ce-heat-representation}, we find that
\begin{align}\label{eq:CS-ce-g-lower}
\frac{f_\varepsilon(x)-P_{s,t}(f_\varepsilon)(x)}{t^{2s}}
\geq c_0t^{\gamma}-c_1\geq\frac{c_0}{2}t^{\gamma}
\end{align}
for any $\varepsilon\in(0,\frac{r_0}{8})$,
$t\in(4\varepsilon,r_0)$, and $x\in\mathcal T_t(E)$.

Let $\lambda_0:=\frac{c_0}{4}$.
Since \eqref{eq:CS-ce-g-lower} and \eqref{eq:CS-ce-weighted-tube} hold, we infer that, as $ \varepsilon\to0^+$,
\begin{align}
&\lambda_0
\underset{|P_{s,t}(f_\varepsilon)(x)-f_\varepsilon(x)|
>\lambda_0t^{2s+\gamma}}
{\int_{\mathbb R^n}\int_0^\infty}
t^{\gamma-1}\,dt\,\upsilon(x)\,dx\notag\\
&\quad\geq\lambda_0\int_{4\varepsilon}^{r_0}
t^{\gamma-1}\upsilon(\mathcal T_t(E))\,dt
\gtrsim\int_{4\varepsilon}^{r_0}\,\frac{dt}{t}
=\log\frac{r_0}{4\varepsilon}\to\infty,
\label{eq:CS-ce-log-divergence}
\end{align}
where the implicit positive constant is independent of $\varepsilon$.
Using the choices $\varepsilon_k:=r_0 2^{-k-4}$ and
$f_k:=f_{\varepsilon_k}$, together with
\eqref{eq:CS-ce-f-norm} and \eqref{eq:CS-ce-log-divergence},
we conclude that \eqref{eq:CS-ce-bounded-data} and
\eqref{eq:CS-ce-divergence} hold. This completes the proof of Lemma \ref{prop:CS-weighted-negative-gamma}.
\end{proof}
\begin{proof}[Proof of Theorem \ref{th:weighted-CS-range}]
The sufficiency follows from Theorem \ref{th:weighted-CS}.
For the necessity, Lemma \ref{prop:CS-gamma-zero} excludes
$\gamma=0$ for any $p\in[1,\infty)$, while Lemma
\ref{prop:CS-weighted-negative-gamma} excludes $\gamma\in[-n,0)$
when $p=1$. This completes the proof of Theorem \ref{th:weighted-CS-range}.
\end{proof}
\subsection{Ball Banach Function Spaces}\label{sub:6-3}
In this subsection, we use the Rubio de Francia extrapolation to connect the $A_1$-weighted
estimates with ball Banach function spaces. First, we recall
some preliminaries on ball quasi-Banach
function spaces introduced in \cite[Definition~2.2]{shyy17}.

\begin{definition}
A quasi-Banach space $X\subset\mathscr{M}$, equipped with a quasi-norm
$\|\cdot\|_X$ which makes sense for all functions in $\mathscr{M}$,
is called a \emph{ball quasi-Banach function space}
 if $X$ satisfies that
\begin{enumerate}
\item[\textup{(i)}]
for any $f\in\mathscr{M}$, if $\|f\|_X=0$, then $f=0$
almost everywhere,
\item[\textup{(ii)}]
if $f,g\in\mathscr{M}$ satisfy that $|g|\leq|f|$ almost everywhere,
then $\|g\|_X\leq\|f\|_X$,
\item[\textup{(iii)}]
if a sequence $\{f_m\}_{m\in\mathbb N}$ in $\mathscr{M}$ satisfies that
$0\leq f_m\uparrow f$ almost everywhere as $m\to\infty$,
then $\|f_m\|_X\uparrow\|f\|_X$ as $m\to\infty$,
\item[\textup{(iv)}]
for any ball $B:=B(x,r)$ with both $x\in\mathbb R^n$ and
$r\in(0,\infty)$, $\mathbf{1}_B\in X$.
\end{enumerate}
Moreover, a ball quasi-Banach function space $X$ is called a \emph{ball Banach function space}
 if $X$ satisfies the following additional conditions:
\begin{enumerate}
\item[\textup{(v)}]
for any $f,g\in X$,
$\|f+g\|_X\leq\|f\|_X+\|g\|_X$,
\item[\textup{(vi)}]
for any ball $B$,
there exists a positive constant $C_{(B)}$, depending on $B$,
such that, for any $f\in X$,
\begin{align*}
\int_B|f(x)|\,dx\leq C_{(B)}\|f\|_X.
\end{align*}
\end{enumerate}
\end{definition}

For any ball Banach function space $X$ and $p\in[1,\infty)$, let
\begin{align*}
X^{\frac{1}{p}}
:=
\left\{f\in\mathscr M:|f|^{\frac{1}{p}}\in X\right\} \ \text{and} \
\|f\|_{X^{\frac{1}{p}}}
:=
\left\||f|^{\frac{1}{p}}\right\|_X^p.
\end{align*}
For any ball Banach function space $X$, its \textit{associate
space} $X'$ consists of all $g\in\mathscr M$ for which
the following norm is finite:
\begin{align*}
\|g\|_{X'}
:=\sup\left\{
\int_{\mathbb R^n}|g(x)h(x)|\,dx:
 h\in X,\ \|h\|_X\leq1
\right\}.
\end{align*}
We also require the following definition. The \textit{homogeneous ball Banach Sobolev space} $\dot W^{1,X}$ consists of
all $f\in L^1_{\mathrm{loc}}$ whose weak gradient exists and satisfies
$|\nabla f|\in X$, equipped with the seminorm
$\|f\|_{\dot W^{1,X}}:=\|\nabla f\|_X$.

From Theorem \ref{th:intro-m}\textup{(i)} and the standard Rubio de Francia
extrapolation argument used to derive
\cite[(4.10)]{dly} (see also \cite[Lemmas~4.6 and 4.7 and Remark~4.8]{dly}),
it follows that the following theorem holds.

\begin{theorem}\label{th:m-x-sim}
Let $\tau\in\{\mathrm{av},\mathrm{osc}\}$, $p\in[1,\infty)$, and
$\gamma\in\Gamma_p^\tau$. Let $X$ be a ball Banach function
space such that $X^{\frac1p}$ is a ball Banach function space
and $M$ is bounded on $(X^{\frac1p})'$.
Then, for any $f\in\mathscr{M}$ with $|D^\tau f|\in X$,
\begin{align}\label{eq:ball-banach-lifted-characterization}
\sup_{\theta,\lambda\in(0,\infty)}\lambda
\left\|\left[\underset{\mathcal M_\theta^\tau(f)(\cdot,t)>\lambda t^\gamma}
{\int_0^\infty}
t^{\gamma p-1}\,dt\right]^{\frac1p}\right\|_X
\sim\|\,|D^\tau f|\,\|_X,
\end{align}
where the positive equivalence constants depend only
on $n$, $p$, $\gamma$, and
$\|M\|_{(X^{\frac1p})'\to(X^{\frac1p})'}$.
\end{theorem}

In the following remark, we give several concrete
examples of function spaces to which
Theorem \ref{th:m-x-sim} applies.
\begin{remark}
The following spaces satisfy all the assumptions of
Theorem \ref{th:m-x-sim}.
\begin{enumerate}
\item[\textup{(i)}]
\emph{Weighted Lebesgue spaces.}
Let $X:=L^r_\upsilon$ be as in Definition \ref{df:c-v}\textup{(ii)},
where $r\in[1,\infty)$ and $\upsilon\in A_r$, and let
$r_\upsilon:=\inf\{q\in[1,\infty):\upsilon\in A_q\}.$
As in the proof of \cite[Theorem~7.1]{scm}, when $r\in(1,\infty)$,
the open property of Muckenhoupt classes gives $r_\upsilon<r$ and
$\upsilon\in A_{\frac rp}$ for any $p\in[1,\frac r{r_\upsilon})$.
Applying \cite[Lemmas~2.4 and 4.16]{scm}, we find that
$X^{\frac1p}=L^{\frac rp}_\upsilon$ is a ball Banach function space
and $M$ is bounded on $(X^{\frac1p})'$ for these choices of $p$.
If $\upsilon\in A_1$,
since
$\|g\|_{(L^1_\upsilon)'}=\|g/\upsilon\|_{L^\infty}$ and
$$M(g)\leq\|g/\upsilon\|_{L^\infty}M(\upsilon)
\leq[\upsilon]_{A_1}\|g\|_{(L^1_\upsilon)'}\upsilon$$ almost everywhere for any $g\in(L^1_\upsilon)'$, we deduce that
\eqref{eq:ball-banach-lifted-characterization} also holds for $p=r$.
For further study of weighted Lebesgue spaces, we refer to
\cite{b93,bz26sparse,hlmorsu12,inns20}.

\item[\textup{(ii)}]
\emph{Morrey spaces.}
Let $X:=M^\alpha_r$, where $r\in[p,\infty)$ and
$\alpha\in[r,\infty)$, with
\begin{align*}
\|h\|_{M^\alpha_r}
:=\sup\left\{|B|^{\frac1\alpha-\frac1r}
\|h\|_{L^r(B)}: B\text{ is a ball in }\mathbb R^n\right\}.
\end{align*}
By \cite[Remark~7.4\textup{(iii)}]{scm}, both $X$ and
$X^{\frac1p}=M^{\frac\alpha p}_{\frac r p}$ are ball Banach
function spaces. When $p<r\le\alpha$, $M$ is bounded on
$(X^{\frac{1}{p}})'$; see Case 1 in the proof of \cite[Theorem 5.1]{dly}.
When $r=p=\alpha$, this follows from $(X^{\frac{1}{p}})'=L^\infty$.
When $r=p<\alpha$, let $\beta:=\frac{\alpha}{p}$. Using \cite[Theorem 3.1]{mst18} and \cite[Lemma 5.7]{Ho2015}
with $p(\cdot)\equiv1$ and $u(x,t)=t^{n(1-\frac{1}{\beta})}$, 
we conclude that $M$ is bounded on $(M_1^\beta)'=(X^{\frac{1}{p}})'$.
See also \cite{ho22morrey,ho24bilinear,ms19,ms20,noss21} for some
recent developments on Morrey spaces.

\item[\textup{(iii)}]
\emph{Mixed-norm Lebesgue spaces.}
Let $X:=L^{\vec r}$ (see \cite[p.\,301]{bp1961}), where
$\vec r:=(r_1,\ldots,r_n)\in{(p,\infty)^n}$.
Then $X^{\frac1p}=L^{\frac{\vec r}{p}}$ is a ball Banach function
space. Its associate space is $L^{\vec s}$, where
$\vec s:=(s_1,\ldots,s_n)\in(1,\infty)^n$ is given by
$\frac{1}{s_i}=1-\frac{p}{r_i}$ for any $i\in\{1,\ldots,n\}$.
The boundedness of $M$ on $(X^{\frac1p})'=L^{\vec s}$
follows from \cite[Lemma~3.5]{hlyy19} with
$\vec a=(1,\ldots,1)$ therein. The equivalence
\eqref{eq:ball-banach-lifted-characterization}
also holds for $\vec r=(p,\ldots,p)$ because $(X^{\frac1p})'=L^\infty$ on which $M$ is bounded.
For further progress of mixed-norm Lebesgue spaces, we refer to
\cite{cgn17,cgn19,hlyy19}.

\item[\textup{(iv)}]
\emph{Variable Lebesgue spaces.}
Let $X:=L^{r(\cdot)}$
(see \cite[Definition 2.16]{cf2013}),
where $r(\cdot)$ is globally log-H\"older continuous in the sense of
\cite[Definition~4.1.1]{dhhr11},
$r_-:=\operatorname*{ess\,inf}r(\cdot)\in(p,\infty)$,
and $\operatorname*{ess\,sup}r(\cdot)<\infty$.
Then $X^{\frac1p}=L^{\frac{r(\cdot)}p}$ is a ball Banach function
space, and $(X^{\frac1p})'=L^{\frac{r(\cdot)}{r(\cdot)-p}}$ with equivalent norms
(see \cite[Theorem~3.2.13]{dhhr11}). The boundedness of $M$ on $(X^{\frac1p})'$ follows from the argument in the proof of \cite[Theorem 4.1]{zyy23}.
See also \cite{cf2013,dhhr11,inns20,inns23} for some recent
results on variable Lebesgue spaces.

\item[\textup{(v)}]
\emph{Orlicz spaces.}
Let $X:=L^\Phi$ (see \cite[Section~1.1]{kk1991}), where
$\Phi$ is a finite-valued Young function such that
$\Phi_p(t):=\Phi(t^{\frac1p})$ for any $t\in[0,\infty)$ is a Young function of lower type $a\in(1,\infty)$ and upper type
$b\in[a,\infty)$ (see \cite[Definition~2.1]{zyyw2019}).
Applying \cite[Chapter~4, Theorem~8.9]{bennett-sharpley88} to
$\Phi$ and $\Phi_p$, we conclude that $X$ and $X^{\frac1p}$ are
ball Banach function spaces.
Let $\Psi_p$ denote the complementary function as in \cite[Chapter~4, Definition~8.11]{bennett-sharpley88} of
$\Phi_p$. Then $(X^{\frac1p})'=L^{\Psi_p}$ with equivalent norms,
and $M$ is bounded on this associate space (see Case~1 in the proof
of \cite[Theorem~5.20]{dly}). Some recent developments can also
be found in \cite{ho23orliczmorrey,kns23,san19}.

\item[\textup{(vi)}]
\emph{Orlicz-slice spaces.}
Let $X:=(E^r_\Phi)_{t_0}$ (see \cite[Definition~2.8]{zyyw2019}),
where $t_0\in(0,\infty)$, $r\in(p,\infty)$,
and $\Phi$ is as in \textup{(v)}. By the definitions, we find that
$X^{\frac1p}=(E^{\frac r p}_{\Phi_p})_{t_0}$
with equal norms (see the proof of \cite[Lemma~2.31]{zyyw2019}), where $\Phi_p$ is as in \textup{(v)}.
By \cite[Lemma~2.28 and Remark~2.17\textup{(ii)}]{zyyw2019},
both $X$ and $X^{\frac1p}$ are ball Banach function spaces.
Moreover, $M$ is bounded on $(X^{\frac1p})'$ [see Case (2)
in the proof of \cite[Theorem 5.28]{zyy}].
We refer also to \cite{ho21slice,ho22slice,zyyw2019,zyy}
for further study of Orlicz-slice spaces.
\end{enumerate}
\end{remark}

\smallskip
\noindent\textbf{Acknowledgements}\quad
The authors acknowledge the use of AI tools for exploratory
discussions and assistance with the preparation and revision of
this manuscript. The authors take full responsibility for the
mathematical content and the final text.

Yirui Zhao would like to thank Professor Ji Li for some
helpful discussions on the sharp range of $\gamma$ in Theorem
\ref{th:weighted-hardy-riesz}.

\bigskip

\noindent Meng Geng, Dachun Yang, Wen Yuan and Yirui Zhao

\medskip

\noindent  Laboratory of Mathematics and Complex Systems
(Ministry of Education of China),
School of Mathematical Sciences, Institute for Advanced Study, Beijing Normal University,
Beijing 100875, The People's Republic of China

\smallskip

\noindent{\it E-mails:} \texttt{menggeng@mail.bnu.edu.cn} (M. Geng)

\noindent\phantom{\it E-mails:} \texttt{dcyang@bnu.edu.cn} (D. Yang)

\noindent\phantom{\it E-mails:} \texttt{wenyuan@bnu.edu.cn} (W. Yuan)

\noindent\phantom{\it E-mails:} \texttt{yiruizhao@mail.bnu.edu.cn} (Y. Zhao)


\begin{thebibliography}{99}

\bibitem{agmp20}
C. Antonucci, M. Gobbino, M. Migliorini and N. Picenni,
Optimal constants for a nonlocal approximation of Sobolev norms and total variation,
Anal. PDE 13 (2020), 595--625.

\vspace{-0.3cm}

\bibitem{aoki42}
T. Aoki,
Locally bounded linear topological spaces,
Proc. Imp. Acad. (Tokyo) 18 (1942), 588--594.

\vspace{-0.3cm}

\bibitem{bp1961}
A. Benedek and R. Panzone,
The space $L^P$, with mixed norm,
Duke Math. J. 28 (1961), 301--324.

\vspace{-0.3cm}

\bibitem{bennett-sharpley88}
C. Bennett and R. Sharpley,
Interpolation of Operators,
Pure and Applied Mathematics 129,
Academic Press, Orlando, FL, 1988.

\vspace{-0.3cm}

\bibitem{bssvy24}
H. Brezis, A. Seeger, J. Van Schaftingen and P.-L. Yung,
Families of functionals representing Sobolev norms,
Anal. PDE 17 (2024), 943--979.

\vspace{-0.3cm}

\bibitem{bssvy22}
H. Brezis, A. Seeger, J. Van Schaftingen and P.-L. Yung,
Sobolev spaces revisited,
Atti Accad. Naz. Lincei Rend. Lincei Mat. Appl. 33 (2022), 413--437.

\vspace{-0.3cm}

\bibitem{bsvy21}
H. Brezis, J. Van Schaftingen and P.-L. Yung,
A surprising formula for Sobolev norms,
Proc. Natl. Acad. Sci. USA 118 (2021),
Paper No. e2025254118, 6 pp.

\vspace{-0.3cm}

\bibitem{bvy-lorentz21}
H. Brezis, J. Van Schaftingen and P.-L. Yung,
Going to Lorentz when fractional Sobolev, Gagliardo and Nirenberg estimates fail,
Calc. Var. Partial Differential Equations 60 (2021), Paper No. 129, 12 pp.

\vspace{-0.3cm}

\bibitem{b93}
S. M. Buckley,
Estimates for operator norms on weighted spaces and reverse Jensen inequalities, Trans. Amer. Math. Soc. 340 (1993), 253--272.

\vspace{-0.3cm}

\bibitem{bui82}
H.-Q. Bui,
Weighted Besov and Triebel spaces: Interpolation by the real method,
Hiroshima Math. J. 12 (1982), 581--605.

\vspace{-0.3cm}

\bibitem{b26} T. A. Bui, Hardy spaces, Campanato spaces and higher 
order Riesz transforms associated with Bessel operators, 
J. Fourier Anal. Appl. 32 (2026), Paper No. 39, 48 pp.

\vspace{-0.3cm}

\bibitem{bdl18}
T. A. Bui, X. T. Duong and F. K. Ly,
Maximal function characterizations for new local Hardy-type spaces on
spaces of homogeneous type,
Trans. Amer. Math. Soc. 370 (2018), 7229--7292.

\vspace{-0.3cm}

\bibitem{bdl20}
T. A. Bui, X. T. Duong and F. K. Ly,
Maximal function characterizations for Hardy spaces on spaces of
homogeneous type with finite measure and applications,
J. Funct. Anal. 278 (2020), Paper No. 108423, 55 pp.

\vspace{-0.3cm}

\bibitem{bz26sparse}
T. A. Bui and L. Zheng,
New sparse domination and weighted estimates for fractional operators
beyond Calder\'on--Zygmund theory,
J. Funct. Anal. 290 (2026), Paper No. 111183, 39 pp.

\vspace{-0.3cm}

\bibitem{cs07}
L. Caffarelli and L. Silvestre,
An extension problem related to the fractional Laplacian,
Comm. Partial Differential Equations 32 (2007), 1245--1260.

\vspace{-0.3cm}

\bibitem{cgn17}
G. Cleanthous, A. G. Georgiadis and M. Nielsen,
Anisotropic mixed-norm Hardy spaces,
J. Geom. Anal. 27 (2017), 2758--2787.

\vspace{-0.3cm}

\bibitem{cgn19}
G. Cleanthous, A. G. Georgiadis and M. Nielsen,
Fourier multipliers on anisotropic mixed-norm spaces of distributions,
Math. Scand. 124 (2019), 289--304.

\vspace{-0.3cm}

\bibitem{cf2013}
D. V. Cruz-Uribe and A. Fiorenza,
Variable Lebesgue Spaces. Foundations and Harmonic Analysis,
Applied and Numerical Harmonic Analysis,
Birkh\"auser/Springer, Heidelberg, 2013.

\vspace{-0.3cm}

\bibitem{dlyyz}
F. Dai, Y. Li, D. Yang, W. Yuan and Y. Zhao,
Sharp weak-type estimate for local lifted {H}ardy--{L}ittlewood
maximal operators with applications to generators of linear
operator families and {H}ardy(--{S}obolev) spaces, Adv. Math.
490 (2026), Paper No. 110822, 55 pp.

\vspace{-0.3cm}

\bibitem{dlyyz2}
F. Dai, Y. Li, D. Yang, W. Yuan and Y. Zhao,
Characterizations and their applications of Sobolev seminorms via lifted sharp maximal functions on complete Poincar\'e spaces, Adv. Math. 502 (2026), part B, Paper No. 111153, 54 pp.

\vspace{-0.3cm}

\bibitem{dly}
F. Dai, X. Lin, D. Yang, W. Yuan and Y. Zhang, Brezis--Van Schaftingen--Yung formulae in ball Banach function spaces with applications to fractional Sobolev and Gagliardo--Nirenberg inequalities, Calc. Var. Partial Differential Equations 62 (2023), Paper No. 56, 73 pp.

\vspace{-0.3cm}

\bibitem{dlyyz22}
F. Dai, X. Lin, D. Yang, W. Yuan and Y. Zhang,
Poincar\'e inequality meets Brezis--Van Schaftingen--Yung formula on metric measure spaces,
J. Funct. Anal. 283 (2022), Paper No. 109645, 52 pp.

\vspace{-0.3cm}

\bibitem{dhhr11}
L. Diening, P. Harjulehto, P. H\"ast\"o and M. R\r{u}\v{z}i\v{c}ka,
Lebesgue and Sobolev Spaces with Variable Exponents,
Lecture Notes in Mathematics 2017,
Springer, Heidelberg, 2011.

\vspace{-0.3cm}

\bibitem{dm22}
O. Dom\'inguez and M. Milman,
New Brezis--Van Schaftingen--Yung--Sobolev type inequalities connected with
maximal inequalities and one parameter families of operators,
Adv. Math. 411 (2022), Paper No. 108774, 76 pp.

\vspace{-0.3cm}

\bibitem{dssvy23}
O. Dom\'inguez, A. Seeger, B. Street, J. Van Schaftingen and P.-L. Yung,
Spaces of Besov-Sobolev type and a problem on nonlinear approximation,
J. Funct. Anal. 284 (2023), Paper No. 109775, 50 pp.

\vspace{-0.3cm}

\bibitem{evans10}
L. C. Evans,
Partial Differential Equations, Second Edition,
Graduate Studies in Mathematics 19,
American Mathematical Society, Providence, RI, 2010.

\vspace{-0.3cm}

\bibitem{frank24}
R. L. Frank,
A characterization of $\dot W^{1,p}(\mathbb R^d)$,
Pure Appl. Funct. Anal. 9 (2024), 53--68.

\vspace{-0.3cm}

\bibitem{gilbarg-trudinger01}
D. Gilbarg and N. S. Trudinger,
Elliptic Partial Differential Equations of Second Order, Classics in Mathematics,
Springer-Verlag, Berlin, 2001.

\vspace{-0.3cm}

\bibitem{gp25}
M. Gobbino and N. Picenni,
Gamma-liminf estimate for a class of non-local approximations of Sobolev and BV norms,
J. Funct. Anal. 289 (2025), Paper No. 111106, 25 pp.

\vspace{-0.3cm}

\bibitem{ghu24}
V. Gol'dshtein, P. Hashash and A. Ukhlov,
On differentiability of Sobolev functions with respect to the Sobolev norm,
Math. Nachr. 297 (2024), 3681--3699.

\vspace{-0.3cm}

\bibitem{gtm249}
L. Grafakos, Classical Fourier Analysis, Third Edition,
Grad. Texts in Math. 249,
Springer, New York, 2014.

\vspace{-0.3cm}

\bibitem{gh23}
Q. Gu and Q. Huang,
Anisotropic versions of the Brezis--Van Schaftingen--Yung approach at $s=1$ and $s=0$,
J. Math. Anal. Appl. 525 (2023), Paper No. 127156, 15 pp.

\vspace{-0.3cm}

\bibitem{hxz25}
B.-X. Han, Z.-F. Xu and Z.-N. Zhu,
A new characterization of Sobolev spaces on Lipschitz differentiability spaces,
arXiv:2504.16657.

\vspace{-0.3cm}

\bibitem{Ho2015}
K.-P. Ho, Atomic decomposition of Hardy--Morrey spaces with variable exponents, Ann. Acad. Sci. Fenn. Math. 40 (2015), 31--62.

\vspace{-0.3cm}

\bibitem{ho21slice}
K.-P. Ho,
Operators on Orlicz-slice spaces and Orlicz-slice Hardy spaces,
J. Math. Anal. Appl. 503 (2021), Paper No. 125279, 18 pp.

\vspace{-0.3cm}

\bibitem{ho22morrey}
K.-P. Ho,
Boundedness of operators and inequalities on Morrey--Banach spaces,
Publ. Res. Inst. Math. Sci. 58 (2022), 551--577.

\vspace{-0.3cm}

\bibitem{ho22slice}
K.-P. Ho,
Fractional integral operators on Orlicz slice Hardy spaces,
Fract. Calc. Appl. Anal. 25 (2022), 1294--1305.

\vspace{-0.3cm}

\bibitem{ho23orliczmorrey}
K.-P. Ho,
An unification of Orlicz-Morrey spaces and its applications,
Manuscripta Math. 172 (2023), 1201--1226.

\vspace{-0.3cm}

\bibitem{ho24bilinear}
K.-P. Ho,
Bilinear operators on ball Banach function spaces,
J. Geom. Anal. 34 (2024), Paper No. 337, 26 pp.

\vspace{-0.3cm}

\bibitem{hs25}
K.-P. Ho and Y. Sawano,
New characterization of Morrey-Herz spaces and Morrey-Herz-Hardy spaces
with applications to various linear operators,
Acta Math. Sin. (Engl. Ser.) 41 (2025), 327--354.

\vspace{-0.3cm}

\bibitem{hlyyz25}
P. Hu, Y. Li, D. Yang, W. Yuan and Y. Zhang,
Sharp Brezis--Seeger--Van Schaftingen--Yung formulae for higher-order
gradients in ball Banach function spaces,
J. Geom. Anal. 35 (2025), Paper No. 388, 76 pp.

\vspace{-0.3cm}

\bibitem{hlyy19}
L. Huang, J. Liu, D. Yang and W. Yuan,
Atomic and Littlewood--Paley characterizations of anisotropic mixed-norm
Hardy spaces and their applications,
J. Geom. Anal. 29 (2019), 1991--2067.

\vspace{-0.3cm}

\bibitem{hytonen26}
T. Hyt\"onen,
Schatten properties of commutators on metric spaces,
J. Funct. Anal. 291 (2026), Paper No. 111556, 98 pp.

\vspace{-0.3cm}

\bibitem{hlmorsu12}
T. P. Hyt\"onen, M. T. Lacey, H. Martikainen, T. Orponen,
M. C. Reguera, E. T. Sawyer and I. Uriarte-Tuero,
Weak and strong type estimates for maximal truncations of
Calder\'on--Zygmund operators on $A_p$ weighted spaces,
J. Anal. Math. 118 (2012), 177--220.

\vspace{-0.3cm}

\bibitem{hk25}
T. Hyt\"onen and R. Korte,
Characterisations of Sobolev spaces and constant functions over metric spaces,
arXiv:2508.07801.

\vspace{-0.3cm}

\bibitem{hyyz26}
T. Hyt\"onen, D. Yang, W. Yuan and Y. Zhao,
Weak-type characterizations of Sobolev and bounded variation spaces on metric measure spaces,
arXiv:2608.24106.

\vspace{-0.3cm}

\bibitem{ins22}
M. Izuki, E. Nakai and Y. Sawano,
Atomic and wavelet characterization of Musielak--Orlicz Hardy spaces
for generalized Orlicz functions,
Integral Equations Operator Theory 94 (2022), Paper No. 3, 33 pp.

\vspace{-0.3cm}

\bibitem{inns20}
M. Izuki, T. Nogayama, T. Noi and Y. Sawano,
Wavelet characterization of local Muckenhoupt weighted Lebesgue spaces
with variable exponent,
Nonlinear Anal. 198 (2020), Paper No. 111930, 14 pp.

\vspace{-0.3cm}

\bibitem{inns23}
M. Izuki, T. Nogayama, T. Noi and Y. Sawano,
Weighted local Hardy spaces with variable exponents,
Math. Nachr. 296 (2023), 5710--5785.

\vspace{-0.3cm}

\bibitem{kns23}
R. Kawasumi, E. Nakai and M. Shi,
Characterization of the boundedness of generalized fractional integral
and maximal operators on Orlicz--Morrey and weak Orlicz--Morrey spaces,
Math. Nachr. 296 (2023), 1483--1503.

\vspace{-0.3cm}

\bibitem{kk1991}
V. Kokilashvili and M. Krbec,
Weighted Inequalities in Lorentz and Orlicz Spaces,
World Scientific Publishing Co. Inc., River Edge, NJ, 1991.

\vspace{-0.3cm}

\bibitem{koskela-saksman08}
P. Koskela and E. Saksman,
Pointwise characterizations of Hardy--Sobolev functions,
Math. Res. Lett. 15 (2008), 727--744.

\vspace{-0.3cm}

\bibitem{ky11}
L. D. Ky,
A note on $H^p_w$-boundedness of Riesz transforms and
$\theta$-Calder\'on--Zygmund operators through molecular characterization,
Anal. Theory Appl. 27 (2011), 251--264.

\vspace{-0.3cm}

\bibitem{ky14}
L. D. Ky,
New Hardy spaces of Musielak--Orlicz type and boundedness of sublinear operators,
Integral Equations Operator Theory 78 (2014), 115--150.

\vspace{-0.3cm}

\bibitem{lahti25}
P. Lahti,
A sharp lower bound for a class of non-local approximations of the total variation,
Math. Ann. 392 (2025), 469--486.

\vspace{-0.3cm}

\bibitem{leoni-morini07}
G. Leoni and M. Morini,
Necessary and sufficient conditions for the chain rule in
$W_{\mathrm{loc}}^{1,1}$ $(\mathbb R^N;\mathbb R^d)$ and
$BV_{\mathrm{loc}}(\mathbb R^N;\mathbb R^d)$,
J. Eur. Math. Soc. (JEMS) 9 (2007), 219--252.

\vspace{-0.3cm}

\bibitem{ln19}
A. K. Lerner and F. Nazarov,
Intuitive dyadic calculus: the basics,
Expo. Math. 37 (2019), 225--265.

\vspace{-0.3cm}

\bibitem{scm}
Y. Li, D. Yang, W. Yuan, Y. Zhang and Y. Zhao,
Sharp weighted
Cohen--Dahmen--Daubechies--DeVore inequality
with applications to (weighted) critical Sobolev spaces,
Gagliardo--Nirenberg inequalities,
and Muckenhoupt weights, Sci. China
Math. (2026), DOI: 10.1007/s11425-025-2527-6.

\vspace{-0.3cm}

\bibitem{ms19}
M. Masty\l{}o and Y. Sawano,
Complex interpolation and Calder\'on--Mityagin couples of Morrey spaces,
Anal. PDE 12 (2019), 1711--1740.

\vspace{-0.3cm}

\bibitem{ms20}
M. Masty\l{}o and Y. Sawano,
Applications of interpolation methods and Morrey spaces to elliptic PDEs,
Ann. Sc. Norm. Super. Pisa Cl. Sci. (5) 21 (2020), 999--1021.

\vspace{-0.3cm}

\bibitem{mst18}
M. Masty\l{}o, Y. Sawano and H. Tanaka,
Morrey-type space and its K\"othe dual space,
Bull. Malays. Math. Sci. Soc. 41 (2018), 1181--1198.

\vspace{-0.3cm}

\bibitem{mattila95}
P. Mattila,
Geometry of Sets and Measures in Euclidean Spaces: Fractals and Rectifiability,
Cambridge Studies in Advanced Mathematics 44,
Cambridge University Press, Cambridge, 1995.

\vspace{-0.3cm}

\bibitem{mtt02}
C. Muscalu, T. Tao and C. Thiele,
Multi-linear operators given by singular multipliers,
J. Amer. Math. Soc. 15 (2002), 469--496.

\vspace{-0.3cm}

\bibitem{nguyen25}
H.-M. Nguyen,
Characterizations of the Sobolev norms and the total variation via nonlocal
functionals, and related problems,
C. R. Math. Acad. Sci. Paris 363 (2025), 1429--1455.

\vspace{-0.3cm}

\bibitem{noss21}
T. Nogayama, T. Ono, D. Salim and Y. Sawano,
Atomic decomposition for mixed Morrey spaces,
J. Geom. Anal. 31 (2021), 9338--9365.

\vspace{-0.3cm}

\bibitem{pssy26}
C. P\'erez, A. Seeger, B. Street and P.-L. Yung,
Two weight estimates for difference quotients,
arXiv:2606.19288.

\vspace{-0.3cm}

\bibitem{poliakovsky22}
A. Poliakovsky,
Some remarks on a formula for Sobolev norms due to Brezis, Van Schaftingen and Yung,
J. Funct. Anal. 282 (2022), Paper No. 109312, 47 pp.

\vspace{-0.3cm}

\bibitem{rolewicz57}
S. Rolewicz,
On a certain class of linear metric spaces,
Bull. Acad. Pol. Sci., Cl. III 5 (1957), 471--473.

\vspace{-0.3cm}

\bibitem{shyy17}
Y. Sawano, K.-P. Ho, D. Yang and S. Yang,
Hardy spaces for ball quasi-Banach function spaces,
Dissertationes Math. 525 (2017), 1--102.

\vspace{-0.3cm}

\bibitem{san19}
M. Shi, R. Arai and E. Nakai,
Generalized fractional integral operators and their commutators with
functions in generalized Campanato spaces on Orlicz spaces,
Taiwanese J. Math. 23 (2019), 1339--1364.

\vspace{-0.3cm}

\bibitem{wheeden76}
R. L. Wheeden,
A boundary value characterization of weighted $H^1$,
Enseign. Math. (2) 22 (1976), 121--134.

\vspace{-0.3cm}

\bibitem{yns26} S. Yamaguchi, E. Nakai and Y. Sawano, 
Bipredual spaces of generalized Campanato spaces with 
variable growth condition of higher order, 
J. Funct. Anal. 291 (2026), Paper No. 111659, 106 pp.

\vspace{-0.3cm}

\bibitem{zyyw2019}
Y. Zhang, D. Yang, W. Yuan and S. Wang,
Real-variable characterizations of Orlicz-slice Hardy spaces,
Anal. Appl. (Singap.) 17 (2019), 597--664.

\vspace{-0.3cm}

\bibitem{zlyyz24}
Y. Zhao, Y. Li, D. Yang, W. Yuan and Y. Zhang,
Generalized {F}rank characterizations of {M}uckenhoupt
weights and homogeneous ball {B}anach {S}obolev spaces,
Adv. Math. 458 (2024), Paper No. 109957, 63 pp.

\vspace{-0.3cm}

\bibitem{zyy}
C. Zhu, D. Yang and W. Yuan, Generalized Brezis--Seeger--Van Schaftingen--Yung formulae and their applications in ball Banach Sobolev spaces, Calc. Var. Partial Differential Equations 62 (2023), Paper No. 234, 76 pp.

\vspace{-0.3cm}

\bibitem{zyy23}
C. Zhu, D. Yang and W. Yuan,
Brezis--Seeger--Van Schaftingen--Yung-type characterization
of homogeneous ball Banach Sobolev spaces and its applications,
Commun. Contemp. Math. 26 (2024), Paper No. 2350041, 48 pp.

\end{thebibliography}
\end{document}